\documentclass[11pt]{article}
\usepackage{fullpage}
\usepackage{amsmath,amsfonts,amsthm,amssymb,epsfig}
\usepackage[all]{xy}
\usepackage{enumitem}
\usepackage[usenames,dvipsnames]{xcolor}
\usepackage[bookmarks=true,colorlinks=true, linkcolor=Mulberry, citecolor=Periwinkle]{hyperref}
\usepackage{graphicx}
\usepackage{float}
\newtheorem{theorem}{Theorem}[section]
\newtheorem{lemma}[theorem]{Lemma}
\newtheorem{proposition}[theorem]{Proposition}
\newtheorem{corollary}[theorem]{Corollary}

\newtheorem{definition}[theorem]{Definition}

\newtheorem{remark}[theorem]{Remark}

\newcommand\bcdot{\ensuremath{%
  \mathchoice%
   {\mskip\thinmuskip\lower0.2ex\hbox{\scalebox{1.5}{$\cdot$}}\mskip\thinmuskip}}%
   {\mskip\thinmuskip\lower0.2ex\hbox{\scalebox{1.5}{$\cdot$}}\mskip\thinmuskip}%   
   {\lower0.3ex\hbox{\scalebox{1.2}{$\cdot$}}}%
   {\lower0.3ex\hbox{\scalebox{1.2}{$\cdot$}}}%
   }

\begin{document}

\title{A K-theoretic approach to Artin maps}

\author{Dustin Clausen}

\maketitle

\begin{abstract}
We define a functorial ``Artin map'' attached to any small $\mathbb{Z}$-linear stable $\infty$-category, which in the case of perfect complexes over a global field $F$ recovers the usual Artin map from the idele class group of $F$ to the abelianized absolute Galois group of $F$.  In particular, this gives a new proof of the Artin reciprocity law.
\end{abstract}

\section*{Introduction}

For a global field $F$ with adele ring $\mathbb{A}_F$ and abelianized absolute Galois group $\operatorname{G}_F^{ab}$, the \emph{Artin map} attached to $F$ is a certain homomorphism of abelian groups
$$\operatorname{Art}_F: \mathbb{A}_F^\times/F^\times\rightarrow \operatorname{G}_F^{ab}$$
which plays a fundamental role in class field theory.  There are also Artin maps associated to other kinds of number-theoretic fields: if $F$ is a local field then we have
$$\operatorname{Art}_F: F^\times\rightarrow \operatorname{G}_F^{ab},$$
and if $F$ is a finite field then we have
$$\operatorname{Art}_F: \mathbb{Z}\rightarrow \operatorname{G}_F^{ab}.$$
We will denote the left hand side of each of these maps by $\operatorname{C}_F$, irrespective of the kind of field $F$.

The last map, where $F$ is a finite field, is simple to produce: it sends $1\in\mathbb{Z}$ to the Frobenius element of $\operatorname{G}_F^{ab}$.  Together with a certain functoriality in $F$, this uniquely determines the whole system of maps $\operatorname{Art}_F$, for all fields $F$ as above.  Nonetheless, the existence of such a system is a difficult theorem, equivalent to the Artin reciprocity law.\\

The goal of this paper is to provide a new and more general approach to the existence of such a system of maps $\operatorname{Art}_F$.  We will recover $\operatorname{Art}_F$ as $\pi_1$ of a certain map between K-theoretic spectra, which we will also call the Artin map.  One advantage of this approach is that it works for all $F$ simultaneously, and the appropriate functoriality in $F$ becomes transparent.  Another advantage is that the extra flexibility and structure of K-theory lets us reduce the entire construction to the base case $F=\mathbb{Z}$.  

In that base case the desired map will come from the homotopy-theoretic product formula of \cite{C}, which was already used there to give a new proof of the quadratic reciprocity law.

\begin{remark}
In fact, we only need a small piece of the product formula from \cite{C}, namely we need its $K(1)$-localization in the sense of stable homotopy theory.  This is one of the mysterious aspects of the story: most of the constructions make sense even without $K(1)$-localization, and it's conceivable that many of the results also hold without $K(1)$-localization.  However, the author is at a loss for how to prove this outside the case of imaginary quadratic fields.  See Remark \ref{beforeK(1)} for more.
\end{remark}

Our theory applies to more general objects than just the fields $F$ as above.  Basically, there is a version of the theory for anything $\mathbb{Z}$-linear in nature which one can take algebraic K-theory of.  For us, the most convenient context is that of \emph{small idempotent-complete $\mathbb{Z}$-linear stable $\infty$-categories}.  Denote the $\infty$-category of these by $\operatorname{PerfCat}_{\mathbb{Z}}$.  Every ring $R$ gives such a $\mathcal{P}\in \operatorname{PerfCat}_{\mathbb{Z}}$, namely $\mathcal{P}=\operatorname{Perf}(R)$.

We will attach two different spectra to any $\mathcal{P}$, as well as an ``Artin map'' between these spectra.  The source spectrum is the \emph{locally compact K-theory spectrum} $\operatorname{K}(\operatorname{lc}_\mathcal{P})$, defined as follows:

\begin{definition}
Let $\mathcal{P}\in \operatorname{PerfCat}_{\mathbb{Z}}$, and let
$$\operatorname{lc}_\mathcal{P} = \operatorname{Fun}_\mathbb{Z}(\mathcal{P},\operatorname{D}^b(\operatorname{LCA}_\aleph))$$
be the stable $\infty$-category of ``$\mathcal{P}^{op}$-modules in locally compact abelian groups".  Here $\operatorname{LCA}_\aleph$ is the exact category of second countable  Hausdorff locally compact abelian groups.

The \emph{locally compact K-theory} of $\mathcal{P}$ is defined to be the K-theory spectrum
$$\operatorname{K}(\operatorname{lc}_\mathcal{P}).$$
\end{definition}

When $\mathcal{P}=\operatorname{Perf}(F)$ for a field $F$ as above, let us abbreviate $\operatorname{lc}_F=\operatorname{lc}_{\operatorname{Perf}(F)}$.  Then there is a natural homomorphism $\operatorname{C}_F\rightarrow \pi_1\operatorname{K}(\operatorname{lc}_F)$.  It is an isomorphism if $F$ is finite or global.  If $F$ is local it is not an isomorphism, but it would be one if we had properly accounted for the topology on $F$ in the previous definition.  In any case, to recover the classical Artin maps we only need the homomorphism.

\begin{remark}\label{simplemech} To give an idea about how the ``reciprocity" arises in this framework, let us describe the map $\operatorname{C}_F\rightarrow \pi_1\operatorname{K}(\operatorname{lc}_F)$ when $F$ is a global field.  The adele group $\mathbb{A}_F$ can be viewed as an $F$-module in $\operatorname{LCA}_\aleph$, hence as an object in $\operatorname{lc}_F$.  As such it has an action by $\mathbb{A}_F^\times$, which gives a map $\mathbb{A}_F^\times\rightarrow \pi_1 \operatorname{K}(\operatorname{lc}_F)$.  We need to see that this map is trivial on $F^\times$.  Consider the short exact sequence
$$F\rightarrow \mathbb{A}_F\rightarrow \mathbb{A}_F/F$$
in $\operatorname{LCA}_\aleph$.  The first term is discrete and hence trivial in K-theory by an Eilenberg swindle with direct sums; the last term is compact and hence trivial in K-theory by an Eilenberg swindle with direct products.  Thus the middle term is trivial as well.  Since $F^\times$ acts by automorphisms on this whole short exact sequence, we deduce that our map $\mathbb{A}_F^\times\rightarrow \pi_1 \operatorname{K}(\operatorname{lc}_F)$ is trivial on $F^\times$.

Note that this argument takes the same form as Tate's observation (\cite{Ta} Theorem 4.3.1) that the theory of Haar measures formally implies the product formula for valuations of a global field.  Indeed, the theory of Haar measures induces $\pi_1\operatorname{K}(\operatorname{lc}_{\mathbb{Z}})\simeq \mathbb{R}_{>0}$.
\end{remark}

The target spectrum attached to $\mathcal{P}$ will be a certain kind of K-homology theory which we call \emph{Selmer K-homology}.  Morally, it is dual to etale sheafified algebraic K-theory; but since etale sheafified algebraic K-theory a priori is not defined for an arbitrary $\mathbb{Z}$-linear stable $\infty$-category and therefore does not have the appropriate functoriality, we replace it by the following:

\begin{definition}
Let $\mathcal{P}\in \operatorname{PerfCat}_{\mathbb{Z}}$.  Define the \emph{Selmer K-theory spectrum} of $\mathcal{P}$ to be the homotopy pullback
$$\operatorname{\operatorname{K^{\operatorname{\operatorname{\operatorname{Sel}}}}}}(\mathcal{P}):=L_1\operatorname{K}(\mathcal{P})\times_{L_1\operatorname{TC}(\mathcal{P})}\operatorname{TC}(\mathcal{P}).$$
Here $\operatorname{TC}$ is topological cyclic homology with its cyclotomic trace map $\operatorname{K}\rightarrow \operatorname{TC}$ (see \cite{DGM} for a textbook reference on the cyclotomic trace, and \cite{BGT} Section 10 for its construction in this precise setting) and $L_1:\operatorname{Sp}\rightarrow\operatorname{Sp}$ is the chromatic level $\leq 1$ localization functor on spectra, i.e.\ $L_1$ is Bousfield localization with respect to complex K-theory (\cite{Bo}  Section 4).
\end{definition}

There is a natural map $\operatorname{K}(\mathcal{P})\rightarrow \operatorname{K^{\operatorname{\operatorname{\operatorname{Sel}}}}}(\mathcal{P})$, and $\operatorname{K^{\operatorname{\operatorname{\operatorname{Sel}}}}}$ is designed to be the closest approximation to K-theory which satisfies etale descent and yet is still defined in the same generality with the same functoriality.

\begin{remark}
This definition is motivated on the one hand by the work of Thomason and Gabber-Suslin (\cite{Th1}) on $L_1\operatorname{K}(X)$, and on the other hand by the work of Geisser-Hesselholt (\cite{GH}, \cite{GH2}) on $\operatorname{TC}(X)$.  Taken together, these imply that $\operatorname{K}^{\operatorname{Sel}}$ is indeed essentially\footnote{What is true is the following.  Let $X$ be a noetherian algebraic space of virtual finite etale cohomological dimension and such that $(X\otimes\mathbb{F}_p)^{red}$ is regular.  Then:  $\operatorname{K}^{\operatorname{Sel}}(-)$ is a hypersheaf on $X_{et}$; the map $\operatorname{K}\rightarrow\operatorname{K}^{\operatorname{Sel}}$ is an equivalence rationally; and on $p$-completion it is an equivalence on etale stalks at points of characteristic $p$, whereas at points of characteristic $\neq p$ the former is the connective cover of the latter and the latter is the 2-periodization of the former.} just the etale hypersheafification of $\operatorname{K}$ on reasonable schemes.  The fact that the etale hypersheafification of algebraic K-theory is still a functor of the category $\operatorname{Perf}(X)$ (up to the natural periodization, see the footnote) is a miracle.
\end{remark}

\begin{remark}
When $X$ is global and mixed characteristic in nature, the definition of $\operatorname{K}^{\operatorname{\operatorname{\operatorname{Sel}}}}(X)$ bears a similarity to the definition of Selmer goups.  The theory $L_1\operatorname{K}$ plays the role of Galois cohomology, and gluing on $\operatorname{TC}$ plays the role of imposing integrality conditions over $\operatorname{Spec}(\mathbb{Z}_p)$.
\end{remark}

In Section \ref{jdual} we will define certain duality functors $d_{K(1)}$and $d_{TC}$, both close relatives of Anderson dualtiy, as well as a natural transformation $d_{K(1)}\rightarrow d_{TC}$, which let us dualize the above definition:

\begin{definition}
Let $\mathcal{P}\in \operatorname{PerfCat}_{\mathbb{Z}}$.  We define the \emph{Selmer K-homology} of $\mathcal{P}$ to be the homotopy pushout
$$\operatorname{dK^{\operatorname{\operatorname{\operatorname{Sel}}}}}(\mathcal{P}) := d_{K(1)}\operatorname{K}(\mathcal{P})\sqcup_{d_{K(1)}\operatorname{TC}(\mathcal{P})}d_{TC}\operatorname{TC}(\mathcal{P}).$$
\end{definition}

If $F$ is a field as above, then there is a canonical isomorphism $\operatorname{G}_F^{ab}\overset{\sim}{\rightarrow} \pi_1 \operatorname{dK^{\operatorname{\operatorname{\operatorname{Sel}}}}}(F)$, coming from an etale descent spectral sequence and Galois cohomological dimension estimates on $F$ (Theorem \ref{eiso}).

\begin{remark}
As far as I know, the required cohomological dimension estimates in the number field case can only be proved using the local-global principle for central simple algebras (as in \cite{Se} II.6).  This principle in turn is usually proved alongside the development of class field theory.  Thus, at this point our approach to Artin reciprocity is not exactly independent of the usual edifice.
\end{remark}

Our main theorem is the following:

\begin{theorem}
There is a canonical natural transformation
$$\operatorname{Art}_{\mathcal{P}}: \operatorname{K}(\operatorname{lc}_{\mathcal{P}})\rightarrow \operatorname{dK^{\operatorname{\operatorname{\operatorname{Sel}}}}}(\mathcal{P})$$
of functors $\operatorname{PerfCat}_{\mathbb{Z}}^{op}\rightarrow \operatorname{Sp}$ such that for all finite, local, and global fields $F$, the composition
$$\operatorname{C}_F\rightarrow \pi_1 \operatorname{K}(\operatorname{lc}_{F})\overset{\pi_1\operatorname{Art}}{\longrightarrow} \pi_1 \operatorname{dK^{\operatorname{\operatorname{\operatorname{Sel}}}}}(F)\overset{\sim}{\leftarrow}\operatorname{G}_F^{ab}$$
equals the usual Artin map.
\end{theorem}

\begin{remark}
If we take $\mathcal{P}=\operatorname{Perf}(X)$ for an algebraic space $X$ essentially of finite type over $\mathbb{Z}$, then we get a K-theoretic version of the Artin map from geometric class field theory, in a form for which reciprocity laws become tautological, or rather reduce to simple mechanisms as in Remark \ref{simplemech}.
Future work will be required to give a more detailed study of $\operatorname{K}(\operatorname{lc}_{\mathcal{P}})$ and $\operatorname{dK^{\operatorname{\operatorname{\operatorname{Sel}}}}}(\mathcal{P})$ in this case, as well as to investigate the question of to what extent $\operatorname{Art}_{\mathcal{P}}$ is an equivalence (compare with \cite{BlM} in the number field case at odd primes).
\end{remark}

\begin{remark}
The definition of locally compact K-theory bears an analogy to the definition of K-homology for C*-algebras.  In fact, it seems from several perspectives that $X\mapsto \operatorname{K}(\operatorname{lc}_X)$ plays the role of an algebraic K-homology theory.  On the other hand $X\mapsto \operatorname{dK}^{\operatorname{\operatorname{\operatorname{Sel}}}}(X)$ plays the role of an etale topological K-homology theory.  Thus another way of thinking of our Artin maps is that they are Riemann-Roch-style natural transformations in the sense of \cite{BFM}, going from algebraic K-homology to topological K-homology.  Recall that the existence of such natural transformations on the level of cohomology theories is essentially tautological, but its existence on the dual homology theories is surprising and encodes nontrivial information.
\end{remark}

\begin{remark}
Every profinite abelian group is the product of its pro-$p$-completions over all primes $p$.  Thus to produce the Artin map for the field $F$ it is enough to produce, for every prime $p$, a $p$-primary Artin map $\operatorname{C}_F\rightarrow (\operatorname{G}_F^{ab})_{\widehat{p}}$.  This is what we will do, and indeed our Selmer K-homology theory $\operatorname{dK}^{\operatorname{\operatorname{\operatorname{Sel}}}}$ will be $p$-complete for a prime $p$ which will be fixed throughout the paper.  To recover the statements as written above, one can take the product over all primes $p$.
\end{remark}

\subsection*{Acknowledgements}

I would like to thank Jacob Lurie for sharing so much relevant insight over the years in which I was his PhD student, as well as for suggesting the ``double Eilenberg swindle'' approach to the K-theory of the exact category of locally compact abelian groups.

I also thank Haynes Miller and John Tate for helpful comments on a previous version of this work, Clark Barwick and Steve Mitchell (through his papers) for teaching me many interesting and useful facts about algebraic K-theory, and Frank Calegari, Lars Hesselholt, Akhil Mathew, and Peter Scholze for helpful exchanges.

Finally, I thank the MIT math department, the National Science Foundation, the University of Copenhagen, and Lars Hesselholt's Niels Bohr Professorship for their support.

\tableofcontents

\section{J-homomorphisms and dualizing objects}\label{jdual}

In this section we will recall a construction from \cite{C}.  Once the definitions are in place, this construction will supply the the ``fundamental class" which completely specifies our system of Artin maps.  We will also define the dualizing objects which go into the definition of Selmer K-homology.

For a ring $R$, let $\operatorname{K}(R)$ denote the connective algebraic K-theory spectrum of $R$.  For an $E_\infty$-ring $R$, let $\operatorname{Pic}(R)$ denote the connective Picard spectrum of $R$, obtained by delooping the space of invertible $R$-modules using the smash product $E_\infty$-structure.  For a prime $p$, let $(-)_{\widehat{p}}:\operatorname{Sp}\rightarrow \operatorname{Sp}$ denote the $p$-completion functor on spectra; thus $S_{\widehat{p}}$ is the $p$-complete sphere.

\subsection{The J-homomorphisms and their reciprocity}

Recall (see \cite{BhS} Appendix for an exposition) that if $R$ is a commutative ring with associated Eilenberg-Maclane $E_\infty$-ring $HR$, then there is a canonical ``determinant" map
$$\operatorname{det}:\operatorname{K}(R)\rightarrow \operatorname{Pic}(HR)$$
which sends the class of a f.g. projective $R$-module $M$ to the line bundle $\Lambda^{\operatorname{dim}(M)}M$ placed in degree $\operatorname{dim}(M)$ (suitably interpreted if $\operatorname{Spec}(R)$ is disconnected). This map identifies $\operatorname{Pic}(H(-))$ with the Zariski-sheafication of the degree $\leq 1$ Postnikov truncation of $\operatorname{K}(-)$.

Fundamental to the construction of our Artin maps is the following theorem:

\begin{theorem}\label{Jthm}
Let $p$ be a prime.  The determinant map $\operatorname{K}(\mathbb{Z}_p)\rightarrow \operatorname{Pic}(H\mathbb{Z}_p)$ lifts along $\operatorname{Pic}(S_{\widehat{p}})\rightarrow \operatorname{Pic}(H\mathbb{Z}_p)$ to a map
$$J_{\mathbb{Z}_p}:\operatorname{K}(\mathbb{Z}_p)\rightarrow \operatorname{Pic}(S_{\widehat{p}})$$
whose restriction to $\operatorname{K}(\mathbb{Z})$ factors through $\operatorname{K}(\mathbb{R})$ to a map
$$J_{\mathbb{R}}:\operatorname{K}(\mathbb{R})\rightarrow \operatorname{Pic}(S_{\widehat{p}}).$$
\end{theorem}
This follows from combining Theorem 5.1 and Lemma 3.2 from \cite{C}, noting that the $J_{\mathbb{Z}_p}$ we mean here is actually the negative of the $J_{\mathbb{Z}_p}$ of \cite{C}.

We immediately see:

\begin{corollary}\label{jcor}
\begin{enumerate}
\item  The map $J_{\mathbb{Z}_p}$ sends $1 \in \pi_0\operatorname{K}(\mathbb{Z}_p)$ to the class of $S^1$ in $\pi_0 \operatorname{Pic}(S_{\widehat{p}})$.
\item For $u\in\mathbb{Z}_p^\times$, the map $J_{\mathbb{Z}_p}$ sends $[u]\in \pi_1\operatorname{K}(\mathbb{Z}_p)$ to the class of the equivalence $\cdot u:S_{\widehat{p}}\simeq S_{\widehat{p}}$ in $\pi_1 \operatorname{Pic}(S_{\widehat{p}})$.
\end{enumerate}
\end{corollary}

\begin{remark}\label{geomj} The above theorem contains all the information we will need on $J_{\mathbb{Z}_p}$ and $J_{\mathbb{R}}$, but let us nonetheless recall that the $J_{\mathbb{Z}_p}$ and $J_{\mathbb{R}}$ considered in \cite{C} are certain specific maps with transparent topological meaning.

Namely, if $M_{\mathbb{Z}_p}$ is a finite free $\mathbb{Z}_p$-module, then $J_{\mathbb{Z}_p}[M_{\mathbb{Z}_p}]$ is the invertible $p$-adic spectrum which controls duality on the classifying topos $BM_{\mathbb{Z}_p}$ (compare \cite{Ba}).  If we give $M_{\mathbb{Z}_p}$ an integral structure $M_{\mathbb{Z}}$, then this identifies with the analogous dualizing object for $BM_\mathbb{Z}$.  The required factoring through $\operatorname{K}(\mathbb{R})$ comes from the fact that $BM_{\mathbb{Z}}$ can be modeled as the torus $M_\mathbb{R}/M_{\mathbb{Z}}$, where Atiyah duality shows that the dualizing object for this torus only depends on the real vector space $M_\mathbb{R}$, and indeed is the one-point compactification of $M_{\mathbb{R}}$, which defines $J_{\mathbb{R}}$ (compare \cite{K} Theorem 10.1).
\end{remark}

\subsection{A K(1)-local dualizing object}

Now fix the prime $p$ and let $L_{K(1)}$ mean $K(1)$-localization at $p$.  The J-homomorphisms of the previous section have target $\operatorname{Pic}(S_{\widehat{p}})$, which somewhat motivates the following definition.

\begin{definition}
Define
$$\omega_{K(1)} = L_{K(1)}\operatorname{Pic}(S_{\widehat{p}}),$$
and define a functor $d_{K(1)}:\operatorname{Sp}^{op}\rightarrow \operatorname{Sp}$ by
$$d_{K(1)}(X) = \operatorname{map}(X,\omega_{K(1)}).$$
\end{definition}
\begin{remark}
Note that $d_{K(1)}$ lands in $K(1)$-local spectra and factors through $K(1)$-localization, so one can restrict $d_{K(1)}$ to a functor $L_{K(1)}\operatorname{Sp}^{op}\rightarrow L_{K(1)}\operatorname{Sp}$ without loss of generality.
\end{remark}

\begin{remark}\label{K(1)anderson} There is an equivalence of spectra $\omega_{K(1)}\simeq\Sigma L_{K(1)}S$ (see below).  It follows from this and \cite{HM} Theorem 8.8 that $\omega_{K(1)}$ is equivalent to a twist of the $p$-adic Anderson dual of $L_{K(1)}S$, the twist being given by a certain ``exotic" $K(1)$-local invertible spectrum $W$ (equivalent to $S^0$ if $p>2$, but not if $p=2$).  Thus on $K(1)$-local spectra $d_{K(1)}$ is a $W$-twist of $p$-adic Anderson duality, and in particular there are short exact sequences
$$0\rightarrow \operatorname{Ext}_{\mathbb{Z}_p}(\pi_{n-1}(W\wedge M),\mathbb{Z}_p)\rightarrow \pi_{-n} d_{K(1)}M\rightarrow \operatorname{Hom}_{\mathbb{Z}_p}(\pi_{n}(W\wedge M),\mathbb{Z}_p)\rightarrow 0$$
for all $n\in \mathbb{Z}$ and $M\in L_{K(1)}\operatorname{Sp}$.\end{remark}

\begin{remark}When $p=2$, then despite the appearance of the strange $W$, in some respects $d_{K(1)}$ is better behaved than $p$-adic Anderson duality.  For example $KO_{\widehat{p}}$ is $d_{K(1)}$ self-dual, and the duality pairing is canonically induced by $J_{\mathbb{R}}$. Its $p$-adic Anderson dual is $\Sigma^4 KO_{\widehat{p}}$, and that duality pairing, like Anderson duality itself, is apparently only canonical in the homotopy category.
\end{remark}

There is a canonical identification $\pi_{-1}L_{K(1)}S = \operatorname{Hom}(\mathbb{Z}_p^\times,\mathbb{Z}_p)$ coming from the $KU$-based Adams spectral sequence, interpreted as a continuous homotopy fixed point spectral sequence (\cite{DH}). To describe it concretely, recall that $\mathbb{Z}_p^\times$ acts on the $K(1)$-local ring spectrum $KU_{\widehat{p}}$ by Adams operations $\psi^x$, $x\in\mathbb{Z}_p^\times$, and that $\pi_{-1}KU_{\widehat{p}}=0$.  Given a class in $\pi_{-1}L_{K(1)}S$ represented by a point $c\in\Sigma L_{K(1)}S$, choose a null-homotopy $\kappa:c\overset{\sim}{\rightarrow} \ast$ in $\Sigma KU_{\widehat{p}}$.  Then for $x\in\mathbb{Z}_p^\times$, the value $c(x)\in\mathbb{Z}_p$ via the above identification is the homotopy class of the element
$$\ast\overset{\kappa^{-1}}{\rightarrow} c\simeq \psi^x(c)\overset{\psi^x\kappa}{\rightarrow}\psi^x(\ast)\simeq\ast$$
in $\pi_0\Omega\Sigma KU_{\widehat{p}}=\pi_0 KU_{\widehat{p}}=\mathbb{Z}_p$.

Thus, if $X$ is a $K(1)$-local spectrum and $\varphi\in\operatorname{Hom}(\mathbb{Z}_p^\times,\mathbb{Z}_p)$, then $\varphi$ gives a functorial degree $-1$ operation on $\pi_\ast X$:
$$\alpha \in \pi_n X\mapsto \alpha\cdot\varphi \in \pi_{n-1}X.$$
Our main result in this section is the following:

\begin{proposition}\label{omegaK(1)}
\begin{enumerate}
\item The $\mathbb{Z}_p$-module $\pi_0 \omega_{K(1)}$ is free on $[S^1]$, the image of the class of $S^1$ in $\operatorname{Pic}(S_{\widehat{p}})$.
\item The homomorphism $\mathbb{Z}_p^\times=\pi_1\operatorname{Pic}(S_{\widehat{p}})\rightarrow \pi_1\omega_{K(1)}$, which we will denote $x\mapsto [x]$, has kernel $\mu_{p-1}$ and cokernel isomorhic to $\mathbb{Z}_p/2\mathbb{Z}_p$.
\item For $\varphi\in \operatorname{Hom}(\mathbb{Z}_p^\times,\mathbb{Z}_p) = \pi_{-1}L_{K(1)}S$ and $x\in\mathbb{Z}_p^\times$, there is the relation
$$[x]\cdot\varphi = 2\varphi(x)\cdot[S^1]$$
in $\pi_0\omega_{K(1)}$.
\end{enumerate}
\end{proposition}

To prove this, we will study $\omega_{K(1)}$ using Rezk's logarithm (\cite{Re}):

\begin{definition}\label{log}
Let $\operatorname{log}_{K(1)}: \omega_{K(1)}\overset{\sim}{\rightarrow} \Sigma L_{K(1)}S$ denote the result of applying the $K(1)$-local Bousfield-Kuhn functor to the pointed map
$$\Omega^\infty \Omega\operatorname{Pic}(S_{\widehat{p}})=\operatorname{GL}_1(S_{\widehat{p}})\rightarrow \Omega^\infty S_{\widehat{p}}$$
given by $x\mapsto x-1$.
\end{definition}

The reason $\operatorname{log}_{K(1)}$ is an equivalence is that $x\mapsto x-1$ is an equivalence after looping once.

Note that $\pi_{-1} L_{K(1)}S = \operatorname{Hom}(\mathbb{Z}_p^\times,\mathbb{Z}_p)$ is a free $\mathbb{Z}_p$-module of rank one.  Recall also that
$\pi_0L_{K(1)}S = \mathbb{Z}_p\oplus (\mathbb{Z}_p/2\mathbb{Z}_p),$
with the first factor generated by $1$ and the second factor generated by $\eta\cdot \ell$ for any generator $\ell$ of $\pi_{-1}L_{K(1)}S$.

The following is proved using Rezk's formula (\cite{Re} Theorem 1.9), or more precisely the special case of Rezk's formula which says that $\operatorname{log_{K(1)}}:\operatorname{Pic}(KU_{\widehat{p}})\rightarrow \Sigma KU_{\widehat{p}}$ on $\pi_1$ is the map $x\mapsto \frac{1}{p}\operatorname{log}(x^{p-1}):\mathbb{Z}_p^\times\rightarrow\mathbb{Z}_p$.

\begin{proposition} Consider the composition $\operatorname{Pic}(S_{\widehat{p}})\rightarrow \omega_{K(1)}\overset{\log_{K(1)}}{\longrightarrow} \Sigma L_{K(1)}S$.
\begin{enumerate}
\item On $\pi_0$, it sends the class of $S^1$ to the generator $\ell$ of $\operatorname{Hom}(\mathbb{Z}_p^\times,\mathbb{Z}_p)=\pi_{-1}L_{K(1)}S$ defined by
$$\ell(x) = \frac{1}{2p}\log(x^{p-1}).$$
\item On $\pi_1$, it sends the class of $u:S_{\widehat{p}}\simeq S_{\widehat{p}}$ in $\mathbb{Z}_p^\times$ to the element of $\pi_0L_{K(1)}S$ given by
$$\frac{1}{p}\operatorname{log}(u^{p-1})+\frac{u-1}{2}\cdot \eta\cdot \ell.$$
\end{enumerate}
\end{proposition}
\begin{proof}
This is explained in \cite{C} Proposition 4.2 for $p>2$; some slight additions allow us to handle $p=2$.

For statement 1, consider the map
$$\operatorname{Pic}(S_{\widehat{p}})\rightarrow \operatorname{Pic}(KU_{\widehat{p}})^{\mathbb{Z}_p^\times}$$
induced by functoriality.  The Bott periodicity class $\beta\in \pi_2KU$ gives a nullhomotopy $S^2\wedge KU_{\widehat{p}}\simeq KU_{\widehat{p}}$ of the image of $S^2$ in $\operatorname{Pic}(KU_{\widehat{p}})$, and $\psi^x(\beta)=x\cdot\beta$ for $x\in\mathbb{Z}_p^\times$.  Thus $[S^2]\in\pi_0\operatorname{Pic}(S_{\widehat{p}})$ is detected in $H^1(B\mathbb{Z}_p^\times;\pi_1\operatorname{Pic}(KU_{\widehat{p}}))=\operatorname{Hom}(\mathbb{Z}_p^\times,\mathbb{Z}_p^\times)$ by the identity map $x\mapsto x$.  By Rezk's formula for the logarithm on $KU_{\widehat{p}}$ and naturality it follows that $(\pi_0\log_{K(1)})([S^2])\in \pi_{-1}L_{K(1)}S$ is detected in $H^1(B\mathbb{Z}_p^\times,\pi_0KU_{\widehat{p}})=\operatorname{Hom}(\mathbb{Z}_p^\times,\mathbb{Z}_p)$ by the map which sends $x$ to $\frac{1}{p}\log(x^{p-1})$.  This forces the claim since $\operatorname{Hom}(\mathbb{Z}_p^\times,\mathbb{Z}_p)$ is torsion-free.

For statement 2, the fact that $\pi_1(\log_{K(1)})([u])$ projects to $\frac{1}{p}\log(u^{p-1})$ on the $\mathbb{Z}_p$-factor of $\pi_0L_{K(1)}S$ follows immediately from Rezk's formula on $KU_{\widehat{p}}$.  For $p>2$ we are finished; for $p=2$ we deduce that $\pi_1(\log_{K(1)})([u]) - \frac{1}{p}\log(u^{p-1})=\epsilon(u)\cdot\eta\cdot\ell$ for some homomorphism $\epsilon:\mathbb{Z}_2^\times\rightarrow\mathbb{Z}/2\mathbb{Z}$, and we need to see that $\epsilon(u)$ is trivial if and only if $u$ is $1$ (mod $4$).

Since $\frac{1}{p}\log(u^{p-1})$ is always divisible by $2$ and hence is killed by $\eta$ whereas $\eta^2\cdot \ell\neq 0$, multiplying by $\eta$ shows that $\epsilon$ is also characterized by
$$\pi_2(\log_{K(1)})(\eta\cdot [u]) = \epsilon(u)\cdot\eta^2\cdot\ell.$$

Since every order 2 character on $\mathbb{Z}_2^\times$ factors through $(\mathbb{Z}/8\mathbb{Z})^\times$, it suffices to see that $\epsilon(-1)$ is nontrivial and $\epsilon(5)$ is trivial.  The non-triviality of $\epsilon(-1)$ follows from the identity $[-1]=\eta\cdot [S^1]$ and statement 1.  For the triviality of $\epsilon(5)$, we claim in fact that $\eta\cdot [p]=0$ in $\operatorname{Pic}(S[1/p])$ for any prime $p$ which is $1$ (mod $4$).  For this, note that for any prime $p$ there is a map of spectra $\Sigma \operatorname{K}(\mathbb{F}_p)\rightarrow \operatorname{Pic}(S[1/p])$ which carries $1\in \pi_0 \operatorname{K}(\mathbb{F}_p)$ to $[p]\in \pi_1 \operatorname{Pic}(S[1/p])$ (namely the map $J_{\mathbb{F}_p}$ of \cite{C} Section 3.1, perhaps first considered by Quillen and/or Tornehave, see \cite{Q1}).  When $-1$ is a square (mod $p$), we have that $\eta=[-1] \in \pi_1\operatorname{K}(\mathbb{F}_p)$ vanishes on any map to an order 2 group.  Since $\pi_2\operatorname{Pic}(S[1/p])=\pi_1S$ is of order 2, we deduce the claim.

Statement 3 is immediate from statement 1.
\end{proof}

Translating this knowledge back to $\omega_{K(1)}$ using the log equivalence $\omega_{K(1)}\simeq \Sigma L_{K(1)}S$, we deduce Proposition \ref{omegaK(1)} above.

\subsection{A $p$-adic dualizing module for $\operatorname{TC}(\mathbb{Z})$}\label{tcdual}

\begin{remark}
In this section, we will use known calculations of Galois cohomology with Tate-twist coefficients for the field $\mathbb{Q}_p$ (as in \cite{Se}  II.5).  Such calculations are of course closely tied up with local class field theory.  It is certainly possible, and possibly preferable, to avoid the use of Galois cohomology by giving purely TC-theoretic proofs of the results in this section.  But we decided to use Galois cohomology for brevity.

In any case, the proof of Artin reciprocity we eventually give can be easily modified so as not to rely on the results of this section, and in particular doesn't require such detailed knowledge of Galois cohomology, or of TC-theory.  For more on this, see Remark \ref{noselmer}.
\end{remark}

Consider the diagram of maps of $E_\infty$-rings
$$\operatorname{TC}(\mathbb{Z})\rightarrow \operatorname{TC}(\mathbb{Z}_p)\overset{\operatorname{tr}}{\longleftarrow} \operatorname{K}(\mathbb{Z}_p)\rightarrow \operatorname{K}(\mathbb{Q}_p).$$
The first map is an equivalence on $p$-completion, the second map is an equivalence on $p$-completion in degrees $\geq 0$, and the third map is an equivalence on $p$-completion in degrees $\geq 2$.  Thus all the maps are equivalences on $L_{K(1)}$.  Let us denote by
$$R$$
the $K(1)$-localization of any of the above spectra, with the understanding that one always uses the above maps to pass between different choices of which spectrum has been localized.  Recalling the map $J_{\mathbb{Z}_p}:\operatorname{K}(\mathbb{Z}_p)\rightarrow\operatorname{Pic}(S_{\widehat{p}})$ from Theorem \ref{Jthm}, let us also denote $L_{K(1)}J_{\mathbb{Z}_p}$ by
$$j_{\mathbb{Z}_p}:R\rightarrow \omega_{K(1)}.$$
Equivalently, we can view $j_{\mathbb{Z}_p}$ as a point of the $R$-module $d_{K(1)}R$.

Here are some relevant classes in $\pi_\ast R$ in degrees $\ast=-1,0,1$.

\begin{definition}\label{classes}
Define the following classes in $\pi_\ast R$:
\begin{enumerate}
\item For $x\in\mathbb{Q}_p^\times=\pi_1\operatorname{K}(\mathbb{Q}_p)$, let $[x]\in \pi_1 R$ denote its image in $R$.
\item \begin{enumerate}
\item For $\varphi\in \operatorname{Hom}(\mathbb{Z}_p^\times,\mathbb{Z}_p)=\pi_{-1}L_{K(1)}S$, let $[\varphi]\in\pi_{-1}R$ denote its image in $R$.
\item Consider the generator $f\in \pi_{-1}\operatorname{TC}(\mathbb{Z})_{\widehat{p}}\overset{\sim}{\rightarrow}\pi_{-1}\operatorname{TC}(\mathbb{F}_p)_{\widehat{p}}$, detected in the homotopy fixed point spectral sequence for $\operatorname{TC}(\mathbb{F}_p)_{\widehat{p}}=\operatorname{TC}(\overline{\mathbb{F}_p})_{\widehat{p}}^{\operatorname{Frob}^\mathbb{Z}}=(H\mathbb{Z}_p)^{\operatorname{Frob}^\mathbb{Z}}$ by the homomorphism $\operatorname{Frob}^\mathbb{Z}\rightarrow\mathbb{Z}_p=\pi_0\operatorname{TC}(\overline{\mathbb{F}_p})_{\widehat{p}}$ sending $\operatorname{Frob}$ to $1$.  Denote also the image of $f$ in $\pi_{-1}R$ by $f$.
\end{enumerate}
\item Consider the unit $1\in \pi_0R$, as well as the element $\epsilon:= f\cdot [p]\in\pi_0R$.
\end{enumerate}
\end{definition}

Now we define another duality functor.

\begin{definition}
Define a $\operatorname{TC}(\mathbb{Z})$-module $\omega_{TC}$ by the cofiber sequence
$$\operatorname{TC}(\mathbb{Z})_{\widehat{p}}\rightarrow R\rightarrow \omega_{TC},$$
and define a functor $d_{TC}:\operatorname{Mod}_{\operatorname{TC}(\mathbb{Z})}^{op}\rightarrow\operatorname{Sp}$ by
$$d_{TC}(M) = \operatorname{map}_{\operatorname{TC}(\mathbb{Z})}(M,\omega_{TC}).$$
\end{definition}

\begin{remark} Since $\omega_{TC}$ is $p$-complete, $d_{TC}$ factors through $p$-completion and lands in $p$-complete spectra.\end{remark}

The main results in this section are as follows.

\begin{theorem}\label{jchar}
The ring $\pi_0 R$ is $\mathbb{Z}_p[\epsilon]/\epsilon^2$.  Elements of $[R,\omega_{K(1)}]$ are uniquely characterized by their effect on $\pi_0$, and for $j_{\mathbb{Z}_p}$ we have
$$(\pi_0 j_{\mathbb{Z}_p})(1) = [S^1], \hspace{10pt}(\pi_0 j_{\mathbb{Z}_p})(\epsilon) = 2[S^1].$$
\end{theorem}

\begin{theorem}\label{Rdual}
The map
$$R\overset{\cdot j_{\mathbb{Z}_p}}{\longrightarrow}d_{K(1)}R$$
is an equivalence.
\end{theorem}

\begin{corollary}\label{n}
We can define a natural transformation $\operatorname{n}:d_{K(1)}\rightarrow d_{TC}$ on $\operatorname{TC}(\mathbb{Z})$-modules as the composition
$$\operatorname{map}(-,\omega_{K(1)})\simeq \operatorname{map}_{\operatorname{TC}(\mathbb{Z})}(-,d_{K(1)}\operatorname{TC}(\mathbb{Z}))\rightarrow \operatorname{map}_{\operatorname{TC}(\mathbb{Z})}(-,\omega_{TC}),$$
where the last map is given by composing with
$$d_{K(1)}\operatorname{TC}(\mathbb{Z})=d_{K(1)}R\overset{\sim}{\leftarrow} R\rightarrow \omega_{TC}.$$
Here the wrong-way equivalence is $\cdot j_{\mathbb{Z}_p}$ and the last map is the tautological one from the cofiber sequence defining $\omega_{TC}$.
\end{corollary}

\begin{theorem}\label{anderson}
The $\operatorname{TC}(\mathbb{Z})_{\widehat{p}}$-module $\omega_{TC}$ is equivalent to the $\mathbb{Z}_p$-Anderson dual $d_{\mathbb{Z}_p}\operatorname{TC}(\mathbb{Z})_{\widehat{p}}$, via an equivalence (unique up to homotopy) which carries $\epsilon\in \pi_0\omega_{TC}$ to the canonical class in $\pi_0 d_{\mathbb{Z}_p}\operatorname{TC}(\mathbb{Z})_{\widehat{p}}$.  Thus on $p$-complete modules, $d_{TC}$ identifies with $\mathbb{Z}_p$-Anderson duality.
\end{theorem}

\begin{remark}
These results should be compared with the very similar K-theoretic local duality result in \cite{BlM} Theorem 1.2.
\end{remark}

\begin{remark}
The proof of Theorem \ref{anderson} will also show that $\operatorname{n}$ is an equivalence on $K(1)$-local $\operatorname{TC}(\mathbb{Z})$-modules.  Thus, if $p>2$, then  by combining Theorem \ref{anderson} and Remark \ref{K(1)anderson} we obtain two identifications of $d_{K(1)}$ with $d_{\mathbb{Z}_p}$ on $R$-modules.  These identifications necessarily differ by a non-trivial unit in $\pi_0R$.   This is the main source of subtlety in our definition of Selmer K-homology: even when $p>2$ it is not equivalent to the Anderson dual of the Selmer K-theory defined the introduction.
\end{remark}

We start by studying some relations between the classes of Definition \ref{classes}.  For this, take the interpretation $R=L_{K(1)}\operatorname{K}(\mathbb{Q}_p)$, and use Thomason's descent spectral sequence (\cite{Th1}), or in other words the spectral sequence associated to the etale-sheafified Postnikov filtration of $L_{K(1)}\operatorname{K}(-)$.  This takes the form
$$E^1_{i,j} = H^{j-i}(\operatorname{BG}_{\mathbb{Q}_p};\mathbb{Z}_p(j/2))\Rightarrow \pi_iL_{K(1)}\operatorname{K}(\mathbb{Q}_p),$$
where as usual $\mathbb{Z}_p(j/2)=0$ for $j$ odd.  Since the cohomological dimension of $\operatorname{BG}_{\mathbb{Q}_p}$ is two, there are no differentials and the result is as follows:
\begin{lemma}\label{K(1)ses}
For $n=2j-1$ odd, there is a canonical identification
$$\pi_nR=H^1(\operatorname{BG}_{\mathbb{Q}_p};\mathbb{Z}_p(j)).$$
For $n=2j$ even, there is a canonical short exact sequence
$$0\rightarrow H^2(\operatorname{BG}_{\mathbb{Q}_p};\mathbb{Z}_p(j+1))\rightarrow \pi_nR\rightarrow H^0(\operatorname{BG}_{\mathbb{Q}_p};\mathbb{Z}_p(j))\rightarrow 0.$$
\end{lemma}

It is easy to see what the classes of Definition \ref{classes} look like in these terms:

\begin{lemma}
\begin{enumerate}
\item For $x\in\mathbb{Q}_p^\times$, the class $[x]\in \pi_1R=H^1(\operatorname{BG}_{\mathbb{Q}_p};\mathbb{Z}_p(1))$ corresponds to the Kummer torsor of $p$-power roots of $x$.  In other words, if we choose a collection $(x^{1/p^n})_n$ of compatible $p^n$-roots of $x$ in an algebraic closure, then $[x]$ is represented by the 1-cocycle $\sigma\mapsto (z_n)_n$ where
$$\sigma(x^{1/p^n})=z_n\cdot x^{1/p^n}.$$
\item \begin{enumerate}
\item For $\varphi\in\operatorname{Hom}(\mathbb{Z}_p^\times,\mathbb{Z}_p)$, the class $[\varphi]\in\pi_{-1}R=\operatorname{Hom}_{cont}(\operatorname{G}_{\mathbb{Q}_p},\mathbb{Z}_p)$ corresponds to the composition
$$\operatorname{G}_{\mathbb{Q}_p}\rightarrow \operatorname{Gal}(\mathbb{Q}_p(\zeta_{p^\infty})/\mathbb{Q}_p)=\mathbb{Z}_p^\times\overset{\varphi}{\longrightarrow}\mathbb{Z}_p.$$
\item The class $f\in\pi_{-1}R=\operatorname{Hom}_{cont}(\operatorname{G}_{\mathbb{Q}_p},\mathbb{Z}_p)$ corresponds to the composition
$$\operatorname{G}_{\mathbb{Q}_p}\rightarrow \operatorname{Gal}(\mathbb{Q}_p^{unr}/\mathbb{Q}_p)_{\widehat{p}}\overset{\simeq}{\longrightarrow}\mathbb{Z}_p,$$
where the isomorphism sends $\operatorname{Frob}$ to $1$.
\end{enumerate}
\item \begin{enumerate}
\item The class $1\in \pi_0R$ projects to $1\in H^0(\operatorname{BG}_{\mathbb{Q}_p};\mathbb{Z}_p)=\mathbb{Z}_p$ in the above short exact sequence.
\item The class $\epsilon\in\pi_0R$ lives in the $H^2(\operatorname{BG}_{\mathbb{Q}_p};\mathbb{Z}_p(1))$-subgroup, and there equals $1\in\mathbb{Z}_p$ under the identification
$$H^2(\operatorname{BG}_{\mathbb{Q}_p};\mathbb{Z}_p(1))\overset{\sim}{\longrightarrow} T_p(\mathbb{Q}/\mathbb{Z})=\mathbb{Z}_p$$
coming from the $\operatorname{inv}$ identification $H^2(\operatorname{BG}_{\mathbb{Q}_p};\mathbb{Q}/\mathbb{Z}(1))\simeq \mathbb{Q}/\mathbb{Z}$.
\end{enumerate}
\end{enumerate}
\end{lemma}
\begin{proof}
Claim 1 is standard and straightforward.

For claim 2(a), note that Snaith's presentation of $KU$ (\cite{Sn}) realizes $KU_{\widehat{p}}$ as the $p$-completed Bott-localization of $(\Sigma^\infty_+ B\mu_{p^\infty})_{\widehat{p}}$.  Further, since the Adams operations are characterized by their effect on $\pi_2$, it follows that the automorphisms of the group $\mu_{p^\infty}$ induce the Adams operations on $KU_{\widehat{p}}$.  Using this one makes a map of $E_\infty$-ring spectra
$$KU_{\widehat{p}}\rightarrow L_{K(1)}\operatorname{K}(\mathbb{Q}_p(\zeta_{p^\infty}))$$
which is $\mathbb{Z}_p^\times$-equivariant for the Adams action on the left and the Galois action on the right.  Then claim 1(b) follows by functoriality.

Claim 2(b) is also immediate from functoriality.

Claim 3 follows from the multiplicativity of Thomason's descent spectral sequence, and standard cup product calculations in the Galois cohomology of $\mathbb{Q}_p$.
\end{proof}

Comparing with known calculations in the Galois cohomology of $\mathbb{Q}_p$, we deduce:

\begin{proposition}\label{relations}
\begin{enumerate}
\item We have
$$\pi_1 R = \mathbb{Z}_p^\times/\mu_{p-1}\oplus\mathbb{Z}_p,$$ with the first factor given by the $[x]$ for $x\in\mathbb{Z}_p^\times$ and the second factor generated by $[p]$.
\item We have
$$\pi_{-1} R = \operatorname{Hom}(\mathbb{Z}_p^\times,\mathbb{Z}_p)\oplus\mathbb{Z}_p,$$
with the first factor given by the $[\varphi]$ for $\varphi\in\operatorname{Hom}(\mathbb{Z}_p^\times,\mathbb{Z}_p)$ and the second factor generated by $f$.
\item We have
$$\pi_0 R = \mathbb{Z}_p \oplus \mathbb{Z}_p,$$
with the first factor generated by $1$ and the second by $\epsilon$.  In these terms, the filtration of $\pi_0 R$ induced by Thomason's spectral sequence REF is given by the subgroup $0\oplus\mathbb{Z}_p$. (Concretely, $0\oplus\mathbb{Z}_p$ is the kernel of $\pi_0R\rightarrow \pi_0L_{K(1)}\operatorname{K}(\overline{\mathbb{Q}_p})=\mathbb{Z}_p$.)
\end{enumerate}
Furthermore, the following multiplicative relations hold:
\begin{enumerate}
\item $\epsilon^2=0$;
\item For $x\in\mathbb{Z}_p^\times$ and $\varphi\in\operatorname{Hom}(\mathbb{Z}_p^\times,\mathbb{Z}_p)$,
$$[x]\cdot [\varphi] = \varphi(x)\cdot\epsilon.$$
\end{enumerate}
\end{proposition}

This gives in particular $\pi_0R=\mathbb{Z}_p[\epsilon]/\epsilon^2$, which is part of Theorem \ref{jchar}.  We can also deduce another part:

\begin{corollary}\label{jonpio}
Consider again $j_{\mathbb{Z}_p}$ as a map $R\rightarrow \omega_{K(1)}$.  We have
$$(\pi_0j_{\mathbb{Z}_p})(1)=[S^1],\text{      }(\pi_0j_{\mathbb{Z}_p})(\epsilon)=2[S^1].$$
\end{corollary}
\begin{proof}
If we $K(1)$-localize Corollary \ref{jcor} we see that $\pi_0j_{\mathbb{Z}_p}$ sends $1$ to $[S^1]$ and $\pi_1j_{\mathbb{Z}_p}$ sends $[u]$ to $[u]$ for $u\in\mathbb{Z}_p^\times$.  But combining Proposition \ref{omegaK(1)} and Proposition \ref{relations} Relation 2 shows that this second statement implies that $\pi_0j_{\mathbb{Z}_p}$ sends $\epsilon$ to $2\cdot [S^1]$.
\end{proof}

Next we turn to Theorem \ref{Rdual}, for which we need to analyze $d_{K(1)}\operatorname{K}(\mathbb{Q}_p)$.  Following \cite{Mi}, this can also be done using a Thomason-style descent spectral sequence.  Namely, as we will see in more detail later, we can consider a version of the Brown-Comenentz dual of $d_{K(1)}\operatorname{K}$ (Section \ref{dualtrick}) and its associated descent spectral sequence (Corollary \ref{descentss}, plus Theorem \ref{homotopysheaves}  to identify the $E^1$-page).  If $(-)^\#$ denotes Pontryagin duality, then this takes the form
$$E^1_{i,j} = H^{j-i}(\operatorname{BG}_{\mathbb{Q}_p};\mathbb{Q}_p/\mathbb{Z}_p(j/2))\Rightarrow \left(\pi_{-i}d_{K(1)}\operatorname{K}(\mathbb{Q}_p)\right)^\#.$$

Dualizing back again, we deduce:

\begin{lemma}\label{dK(1)ses}
For $n=2j+1$ odd, there is a canonical identification
$$\pi_nd_{K(1)}R=H^1(\operatorname{BG}_{\mathbb{Q}_p};\mathbb{Q}_p/\mathbb{Z}_p(-j))^\#.$$
For $n=2j$ even, there is a canonical short exact sequence
$$0\rightarrow H^0(\operatorname{BG}_{\mathbb{Q}_p};\mathbb{Q}_p/\mathbb{Z}_p(-j))^\#\rightarrow \pi_nd_{K(1)}R\rightarrow H^2(\operatorname{BG}_{\mathbb{Q}_p};\mathbb{Q}_p/\mathbb{Z}_p(-j+1))^\#\rightarrow 0.$$
\end{lemma}

By Corollary \ref{jonpio}, we know that $j_{\mathbb{Z}_p}$ sends $\epsilon$ to $2\cdot [S^1]$.  Thus to prove Theorem \ref{Rdual}, it suffices to show:

\begin{proposition}\label{localduality}
Let $c\in [R,\omega_{K(1)}]=\pi_0d_{K(1)}R$.  The following are equivalent:
\begin{enumerate}
\item $(\pi_0c)(\epsilon)=2\cdot [S^1]$.
\item $c$ projects to the usual generator $\operatorname{inv}\in H^2(\operatorname{BG}_{\mathbb{Q}_p};\mathbb{Q}_p/\mathbb{Z}_p(1))^\#$ in the short exact sequence of Lemma \ref{dK(1)ses}.
\end{enumerate}
Under these conditions, the multiplication by $c$ map $R\rightarrow d_{K(1)}R$ is an equivalence.
\end{proposition}

\begin{proof}
We use that the the Brown-Comenentz dual of $d_{K(1)}\operatorname{K}(-)$ is a module over $L_{K(1)}\operatorname{K}(-)$, so that the descent spectral sequence for the former is a module over the descent spectral sequence for the latter.  We claim that on the $E^1$-page, this module structure is induced by the cup product.  To verify this, we need to see that the induced pairing on homotopy group sheaves is the tautological product structure.  This follows from Corollary \ref{k(1)stalk}, which gives that the homotopy groups of the stalks of the former identify with those of the latter tensored with $\mathbb{Q}_p/\mathbb{Z}_p$, compatibly with the module structure and with all cospecialization maps.  Thus by Lemma \ref{stalktosheaf} (which is actually trivial in the case of a field), we get the same identification on homotopy group sheaves, giving the claim.

We deduce that the multiplication by $c$ map
$$\pi_nR\rightarrow \pi_n d_{K(1)}R$$
respects the filtrations of Lemmas \ref{K(1)ses} and \ref{dK(1)ses}, and on the associated graded pieces is given by the map
$$H^i(\operatorname{BG}_{\mathbb{Q}_p};\mathbb{Z}_p(j))\rightarrow H^{2-i}(BG_{\mathbb{Q}_p};\mathbb{Q}_p/\mathbb{Z}_p(1-j))^\#$$
of capping with the image of $c$ in $H^2(\operatorname{BG}_{\mathbb{Q}_p};\mathbb{Q}_p/\mathbb{Z}_p(1))^\#$.

If 2 holds, then local duality implies that this capping with the image of $c$ is an isomorphism, hence multiplication by $c$ is an equivalence $R\rightarrow d_{K(1)}R$.  To finish, we need to show 1 $\Leftrightarrow$ 2.  Fix a $c'$ satisfying 2.  Given the isomorphism $\cdot c':\pi_0R\rightarrow \pi_0d_{K(1)}R$ we just proved, it suffices to show that $c'$ sends $\epsilon$ to $2\cdot[S^1]$ (assuming this, write an arbitrary $c$ as $(a+b\epsilon)c'$; then $c$ projects to $a\cdot\operatorname{inv}$ in $H^2(\operatorname{BG}_{\mathbb{Q}_p};\mathbb{Q}_p/\mathbb{Z}_p(1))^\#$ and sends $\epsilon$ to $2a\cdot [S^1]$, proving the claim). Note by multiplicativity and the description of $\epsilon$ in the descent spectral sequence (Theorem \ref{relations}) that $c'\cdot\epsilon$ is the generator $1$ of the subgroup $\mathbb{Z}_p=H^0(\operatorname{BG}_{\mathbb{Q}_p};\mathbb{Q}_p/\mathbb{Z}_p)^\#\subset \pi_0d_{K(1)}R$.  On the other hand, by the construction of the spectral sequence and the identification of the stalks in Corollary \ref{k(1)stalk}, this generator corresponds to the map $\operatorname{K}(\mathbb{Q}_p)\rightarrow \omega_{K(1)}$ given as the composition
$$\operatorname{K}(\mathbb{Q}_p)\rightarrow \operatorname{K}(\overline{\mathbb{Q}_p})\rightarrow \omega_{K(1)}$$
where the second map sends $1$ to $2\cdot[S^1]$.  The claim follows.
\end{proof}

This proves Theorem \ref{Rdual}.  We can also finish the proof Theorem \ref{jchar}: what's missing is to show that an element of $[R,\omega_{K(1)}]$ is determined by its effect on $\pi_0$.  But this is obvious from the fact that $j_{\mathbb{Z}_p}$ and $j_{\mathbb{Z}_p}\cdot\epsilon$ form a basis, and the description of their effect on $\pi_0$ (Corollary \ref{jonpio} plus $\epsilon^2=0$).

Finally, we turn to Theorem \ref{anderson}.  First, note that since $\pi_{-1}R\simeq \mathbb{Z}_p\oplus\mathbb{Z}_p$ is a free $\mathbb{Z}_p$-module, Exts out of it vanish, and thus a class in the $\mathbb{Z}_p$-Anderson dual
$$c\in \pi_0 d_{\mathbb{Z}_p}R=[R,I_{\mathbb{Z}_p}]$$
is uniquely determined by its effect on $\pi_0$, which will be a homomorphism
$$\pi_0R\rightarrow\mathbb{Z}_p.$$
One can analyze $d_{\mathbb{Z}_p}R$ exactly as we analyzed $d_{K(1)}R$ above, and the analog of Proposition \ref{localduality} implies the following, which in fact was proved in \cite{BlM} Theorem 1.2:

\begin{proposition}\label{ddual}
Let $c\in d_{\mathbb{Z}_p}R$ correspond to the homomorphism $\pi_0R\rightarrow\mathbb{Z}_p$ sending $1$ to $0$ and $\epsilon$ to $1$.  Then the multiplication by $c$ map
$$R\rightarrow d_{\mathbb{Z}_p}R$$
is an equivalence.
\end{proposition}

To analyze $\omega_{TC}$, we need the following.

\begin{lemma}\label{homotopyTC}
The map $\operatorname{TC}(\mathbb{Z})_{\widehat{p}}\rightarrow L_{K(1)}\operatorname{TC}(\mathbb{Z})=R$ has the following properties on $\pi_\ast$:
\begin{enumerate}
\item It is an isomorphism in degrees $\ast \geq 2$.
\item In degree $\ast=1$, it is injective with image $\mathbb{Z}_p^\times/\mu_{p-1}\oplus 0$ (\ref{relations}).
\item In degree $\ast=0$, it is injective with image $\mathbb{Z}_p\oplus 0$.
\item In degree $\ast=-1$, it is injective with image $0\oplus\mathbb{Z}_p$.
\item In degrees $\ast \leq -2$, the source is $0$.
\end{enumerate}
\end{lemma}
\begin{proof}
By McCarthy's theorem (\ref{mccarthy}) and $p$-adic continuity (\ref{padiccontinuity}), it follows that $\operatorname{K}(\mathbb{Z}_p)_{\widehat{p}}\rightarrow \operatorname{TC}(\mathbb{Z}_p)_{\widehat{p}}$ is an isomorphism in degrees $\geq 0$ and $\leq -2$, and in degree $-1$ we have that $\pi_{-1}\operatorname{TC}(\mathbb{Z}_p)_{\widehat{p}}$ is a free $\mathbb{Z}_p$-module on $f$.  Then all the claims except claim 1 follow from Proposition \ref{relations}.

For claim 1, one can either refer to the calculation of the homotopy type of the spectrum $\operatorname{TC}(\mathbb{Z})_{\widehat{p}}$ given in \cite{BoM}, \cite{Ts} for $p>2$ and \cite{Ro} for $p=2$ to directly verify the claim, or else one can transfer the claim to $\operatorname{K}(\mathbb{Q}_p)_{\widehat{p}}$ and use the norm residue isomorphism theorem plus the fact that $\mathbb{Q}_p$ has Galois cohomological dimension $2$ (cf \cite{W} VI.4).
\end{proof}

It follows that $\pi_1\omega_{TC}$ is a free $\mathbb{Z}_p$-module on the image of $[p]$, that $\pi_0\omega_{TC}$ is a free $\mathbb{Z}_p$-module on the image of $\epsilon$, and that $\pi_{-1}\omega_{TC}$ identifies with the free $\mathbb{Z}_p$-module of rank one $\operatorname{Hom}(\mathbb{Z}_p^\times,\mathbb{Z}_p)$.  In particular there is a unique class class in $d_{\mathbb{Z}_p}\omega_{TC}=[\omega_{TC},I_{\mathbb{Z}_p}]$ sending $\epsilon$ to $1$.  We can formally extend this class to a unique up to homotopy $\operatorname{TC}(\mathbb{Z})_{\widehat{p}}$-module map
$$\alpha:\omega_{TC}\rightarrow d_{\mathbb{Z}_p}\operatorname{TC}(\mathbb{Z})_{\widehat{p}}$$
sending $\epsilon$ to the tautological generator of $\pi_0d_{\mathbb{Z}_p}\operatorname{TC}(\mathbb{Z})_{\widehat{p}}=\operatorname{Hom}(\mathbb{Z}_p,\mathbb{Z}_p)$ given by the identity.  By construction one sees that there is a commutative square
$$\xymatrix{ R\ar[r]^-{\cdot c}\ar[d] & d_{\mathbb{Z}_p}R\ar[d] \\
\omega_{TC}\ar[r]^-{\alpha} & d_{\mathbb{Z}_p}\operatorname{TC}(\mathbb{Z})_{\widehat{p}} }$$
Comparing Lemma \ref{homotopyTC} parts 1 and 5 with Proposition \ref{ddual}, we deduce that $\alpha$ is an isomorphism in degrees $\geq 2$ and $\leq -2$.  It is an isomorphism in degree $0$ by construction.  In degree $1$ it is an isomorphism by the identity $[p]\cdot f=\epsilon$.  In degree $-1$ it is an isomorphism by the identity $[\varphi]\cdot [u]=\varphi(u)\cdot\epsilon$ (Proposition \ref{relations}).  This proves Theorem \ref{anderson}.

\section{Selmer K-homology}\label{selmer}

Again in this section we will fix a prime $p$.

\subsection{Definition and first properties}

Using the duality functors $d_{K(1)}$ and $d_{TC}$ and the natural transformation $\operatorname{n}$  from the previous section, we can define our chosen $p$-adic dual to etale K-theory.

\begin{definition}
Let $\mathcal{P}\in \operatorname{PerfCat}_{\mathbb{Z}}$ be a small idempotent-complete $\mathbb{Z}$-linear stable $\infty$-category.  Define the ($p$-adic) \emph{Selmer K-homology} spectrum of $\mathcal{P}$, denoted $\operatorname{dK^{\operatorname{\operatorname{\operatorname{Sel}}}}}(\mathcal{P})$, to be the pushout of
$$d_{K(1)}\operatorname{K}(\mathcal{P})\overset{d_{K(1)}\operatorname{tr}}{\longleftarrow} d_{K(1)}\operatorname{TC}(\mathcal{P})\overset{\operatorname{n}}{\longrightarrow} d_{TC}\operatorname{TC}(\mathcal{P}).$$
Here $\operatorname{tr}:\operatorname{K}(\mathcal{P})\rightarrow \operatorname{TC}(\mathcal{P})$ is the cyclotomic trace map between the localizing invariants $\operatorname{K}$ and $\operatorname{TC}$ (as in \cite{BGT} Section 10).

For a ring (or DGA) $A$ we set
$$\operatorname{dK^{\operatorname{\operatorname{\operatorname{Sel}}}}}(A)= \operatorname{dK^{\operatorname{\operatorname{\operatorname{Sel}}}}}(\operatorname{Perf}(A)),$$
and similarly for a qcqs algebraic space (or derived algebraic space) $X$ over $\mathbb{Z}$ we set
$$\operatorname{dK^{\operatorname{\operatorname{\operatorname{Sel}}}}}(X)=\operatorname{dK^{\operatorname{\operatorname{\operatorname{Sel}}}}}(\operatorname{Perf}(X)).$$
\end{definition}

\begin{remark}
The version of algebraic K-theory which satisfies localization is the \emph{nonconnective} algebraic K-theory.  On the other hand we have a convention that $\operatorname{K}(A)$ for a ring (or connective DGA) $A$ denotes the \emph{connective} algebraic K-theory of $A$.  Thus, according to our conventions $\operatorname{K}(A)$ is the connective cover of $\operatorname{K}(\operatorname{Perf}(A))$.  Note, however, that K-theory only contributes to $\operatorname{dK^{\operatorname{\operatorname{\operatorname{Sel}}}}}$ through its $K(1)$-localization, so this distinction doesn't affect the above definition.  With TC there is no ambiguity: $\operatorname{TC}(A)=\operatorname{TC}(\operatorname{Perf}(A))$ for any DGA $A$.  Even when $A$ is connective $\operatorname{TC}(A)$ is not necessarily connective, but it's close: $\pi_n(\operatorname{TC}(A)_{\widehat{p}})=0$ for $n<-1$.
\end{remark}

\begin{remark}
By construction, $\operatorname{dK^{\operatorname{\operatorname{\operatorname{Sel}}}}}$ is a functor
$$\operatorname{PerfCat}_{\mathbb{Z}}^{op}\rightarrow \operatorname{Sp}$$
which satisfies localization, i.e.\ it sends fiber-cofiber sequences to fiber-cofiber sequences, i.e.\ if
$$\mathcal{M}\rightarrow\mathcal{N}\rightarrow\mathcal{P}$$
is a Verdier quotient sequence up to idempotent completion, then
$$\operatorname{dK^{\operatorname{\operatorname{\operatorname{Sel}}}}}(\mathcal{P})\rightarrow \operatorname{dK^{\operatorname{\operatorname{\operatorname{Sel}}}}}(\mathcal{N})\rightarrow \operatorname{dK^{\operatorname{\operatorname{\operatorname{Sel}}}}}(\mathcal{M})$$
is a fiber sequence of spectra.  In particular, by the standard Thomason-Trobaugh argument (\cite{Ta} Section 2, \cite{CMNN} Appendix A), $X\mapsto \operatorname{dK^{\operatorname{\operatorname{\operatorname{Sel}}}}}(X)$ is a Nisnevich co-sheaf on qcqs derived algebraic spaces over $\mathbb{Z}$.  But one of the main points is that it is even an etale co-sheaf, see Section \ref{dualtrick}.
\end{remark}

\begin{remark}
For concreteness, let us posit that our derived algebraic spaces are locally modeled on $\operatorname{Spec}$ of a connective $E_\infty$-algebra over $\mathbb{Z}$.  Most of the statements hold in the weaker context where the local models are connective $E_2$-algebras, and some hold with local models just quasi-commutative $E_1$-algebras (or DGAs).  However, none of our examples of interest will be derived, so the reader can ignore these kinds of technicalities.
\end{remark}

There is a strong dichotomy between the behavior of Selmer K-homology at $p$ and away from $p$: away from $p$ only the term $d_{K(1)}\operatorname{K}(\mathcal{P})$ contributes, while at $p$ only the term $d_{TC}\operatorname{TC}(\mathcal{P})$ does.   Before saying this more precisely, we need to recall some facts about K-theory and TC-theory, starting with the fundamental result of McCarthy (\cite{Mc}):

\begin{theorem}\label{mccarthy}
Define the functor $\operatorname{F}$ as the fiber of $\operatorname{tr}:\operatorname{K}\rightarrow \operatorname{TC}$.  Let $A\rightarrow B$ be a  morphism of connective DGAs over $\mathbb{Z}$ such that $\pi_0A\rightarrow \pi_0B$ is surjective with nilpotent kernel.  Then
$$\operatorname{F}(A)\overset{\sim}{\rightarrow} \operatorname{F}(B).$$
\end{theorem}

Next we give a slight amplification of some results from \cite{GH2} on $p$-adic continuity in K-theory and TC-theory:

\begin{theorem}\label{padiccontinuity}
Let $A$ be a connective DGA over $\mathbb{Z}$, and for all $n\in\mathbb{N}$ let $A\otimes_\mathbb{Z}\mathbb{Z}/p^n\mathbb{Z}$ denote the derived tensor product of $A$ with $\mathbb{Z}/p^n\mathbb{Z}$.
\begin{enumerate}
\item For all $\ast\in\mathbb{Z}$, the map
$$\pi_\ast\left(\operatorname{TC}(A)/p\right)\rightarrow ``\varprojlim_n" \pi_\ast\left(\operatorname{TC}(A\otimes_\mathbb{Z}\mathbb{Z}/p^n\mathbb{Z})/p\right)$$
is an isomorphism in the Pro-category of abelian groups.  (This is even true ``uniformly in $A$", meaning  it is an isomorphism in the Pro-category of functors from connective DGAs to abelian groups.)
\item Let $R=\pi_0A$, and suppose the following:
\begin{enumerate}
\item $A$ is quasi-commutative, i.e.\ the action of $R$ on $\pi_\ast A$ is commutative (ensuring a reasonable theory of Zariski localization on $A$, see \cite{L1} 7.2.3);
\item  $(R,pR)$ is a henselian pair;
\item  $R$ has bounded $p$-torsion ($\exists N\in\mathbb{N}$ such that $R[p^{N+1}]=R[p^N]$; this holds if $R$ is noetherian, or $R$ lives over either $\mathbb{Z}[1/p]$ or $\mathbb{Z}/p^n\mathbb{Z}$ for some $n$, or $R_{(p)}$ is flat over $\mathbb{Z}_{(p)}$).
\end{enumerate}
Then if $\pi_\ast^{Zar}$ denotes the functor of taking homotopy groups and then Zariski sheafifying over $\operatorname{Spec}(A)$, the map
$$\pi_\ast^{Zar}\left(\operatorname{K}(-)/p\right)\rightarrow ``\varprojlim_n" \pi_\ast^{Zar}\left(\operatorname{K}(-\otimes_\mathbb{Z}\mathbb{Z}/p^n\mathbb{Z})/p\right)$$
is an isomorphism in the Pro-category of abelian group sheaves on $\operatorname{Spec}(A)_{Zar}$ for each $\ast\in\mathbb{Z}$.  (This is also true uniformly in $A$ provided we fix $N$.)
\end{enumerate}
\end{theorem}
\begin{proof}
For 1, by the standard devissage (cf \cite{GH2} Section 3) it suffices to show that
$$A\rightarrow ``\varprojlim_n" A\otimes_\mathbb{Z}\mathbb{Z}/p^n\mathbb{Z},$$
is a pro-equivalence in $\operatorname{Mod}(H\mathbb{Z})$ after applying $(-)/p=-\otimes_\mathbb{Z}\mathbb{Z}/p\mathbb{Z}$.  But the fiber of this map is the pro-system represented by $(\ldots A\overset{p}{\rightarrow}A\overset{p}{\rightarrow}A)$.  Tensoring with $\mathbb{Z}/p\mathbb{Z}$ makes all the maps nullhomotopic, hence the system is pro-zero.

For 2, \cite{GH2} Remark 1.3.2 (based on the argument of \cite{Sus} Section 3) gives\footnote{\cite{GH2} requires that $R$ be $p$-torsionfree, but the proof only requires information of the deep enough $p$-power congruence subgroups, and these are the same for $R$ and the $p$-torsion-free ring $R/R[p^N]$.  (Thanks to Akhil Mathew for this remark.)}  
$$\pi_\ast^{Zar}\operatorname{K}(-)/p\overset{\sim}{\longrightarrow} ``\varprojlim_n"\pi_\ast^{Zar} \operatorname{K}(-/(p^n))/p,$$
where here the dash runs over Zariski localizations of $R$ and by $-/(p^n)$ we mean the ordinary modding out by the ideal generated by $p^n$.  On the other hand, since $R[p^{N+1}]=R[p^N]$ (and therefore also the same for any Zariski localization of $R$), an argument similar to that of claim 1 gives
$$\pi_\ast\operatorname{TC}(-)/p\overset{\sim}{\longrightarrow} ``\varprojlim_n" \pi_\ast\operatorname{TC}(-/(p^n))/p$$
where again $(-)$ runs over Zariski localizations of $R$.  Then looking at the compatible comparison maps $A\rightarrow R$ and $A\otimes_{\mathbb{Z}}\mathbb{Z}/p^n\mathbb{Z}\rightarrow R/(p^n)$, McCarthy's theorem \ref{mccarthy} lets us reduce claim 2 to claim 1.
\end{proof}

\begin{remark}
Using the Milnor sequence, it's easy to see that 1 implies $\operatorname{TC}(A)_{\widehat{p}}\overset{\sim}{\longrightarrow} \varprojlim_n \operatorname{TC}(A\otimes_{\mathbb{Z}}\mathbb{Z}/p^n\mathbb{Z})_{\widehat{p}}$.  Similarly, since $\operatorname{K}(-)$ satisfies Zariski descent, the conclusion of 2 implies that $\operatorname{K}(A)_{\widehat{p}}\overset{\sim}{\longrightarrow}\varprojlim_n\operatorname{K}(A\otimes_{\mathbb{Z}}\mathbb{Z}/p^n\mathbb{Z})_{\widehat{p}}$ if we add the hypothesis that the topological space $\operatorname{Spec}(A)$ is e.g.\ noetherian of finite Krull dimension, so that the Zariski descent spectral sequences converge to the correct answer.
\end{remark}

The following is a version of \cite{GH2} Theorem B.

\begin{theorem}\label{GHetale}
Let $A$ be a connective DGA over $\mathbb{Z}$ which is quasi-commutative (i.e.\ the action of $\pi_0A$ on $\pi_\ast A$ is commutative; this ensures a reasonable theory of etale $A$-algebras, \cite{L1}  7.5.1), and let $(-)^h$ denote the functor of henselization at $p$.  Then:
\begin{enumerate}
\item $\operatorname{TC}(A)/p\rightarrow \operatorname{TC}(A^h)/p$ is an equivalence.
\item Consider the cyclotomic trace map
$$\operatorname{tr}:\operatorname{K}(A^h)\rightarrow \operatorname{TC}(A^h).$$
Assume the following:
\begin{enumerate}
\item $\pi_0(A\otimes_{\mathbb{Z}}\mathbb{F}_p)$ is a nilpotent thickening of a regular ring;
\item $\pi_0A$ has bounded $p$-torsion.
\end{enumerate}
Then if $\pi_\ast^{et}$ denotes the functor of taking homotopy groups and then sheafifying over $\operatorname{Spec}(A)_{et}$, the map
$$\pi_\ast^{et}(\operatorname{tr}/p):\pi_\ast^{et}\operatorname{K}(A^h)/p\rightarrow \pi_\ast^{et}\operatorname{TC}(A^h)/p$$
is an isomorphism.
\end{enumerate}
\end{theorem}

\begin{proof}
For the first claim, note $A\otimes_{\mathbb{Z}}\mathbb{F}_p\overset{\sim}{\rightarrow} A^h\otimes_{\mathbb{Z}}\mathbb{F}_p$, so $\operatorname{TC}(A)/p\overset{\sim}{\rightarrow}\operatorname{TC}(A^h)/p$ by $p$-adic continuity (Theorem \ref{padiccontinuity}).

For the second claim, consider the fiber $\operatorname{F}(A^h)/p$ of $\operatorname{tr}/p$, and let $x\in \pi_\ast(\operatorname{F}(A^h)/p)$.  It suffices to show that $x$ vanishes locally on $\operatorname{Spec}(A)_{et}$.  Since even $(-)^h$ vanishes over the open $\operatorname{Spec}(A[1/p])$, it suffices to see that $x$ vanishes etale-locally around every characteristic $p$ point of $\operatorname{Spec}(A)$.

First assume the claim for $A':=\pi_0(A\otimes_{\mathbb{Z}}\mathbb{F}_p)^{red}$ replacing $A$.  Then there is an etale cover $\{A'_i\}$ of $A'$ such that $x$ vanishes in each $\operatorname{F}(A'_i)/p$.  Lift each $A'_i$ arbitrarily to an etale $A$-algebra $A_i$; then each characteristic $p$ point is covered by some $A_i$.  By $p$-adic continuity for TC and K we get $p$-adic continuity for F, so the homotopy groups of the fiber of $\operatorname{F}(A_i^h)/p \rightarrow \varprojlim_n \operatorname{F}(A_i\otimes_{\mathbb{Z}} \mathbb{Z}/p^n\mathbb{Z})/p$ vanish Zariski-locally on $(A_i)^h$, and therefore vanish etale-locally at each characteristic $p$ point of $A_i$.  Thus it suffices to see that the image of $x$ in $\varprojlim_n \operatorname{F}(A_i\otimes_{\mathbb{Z}} \mathbb{Z}/p^n\mathbb{Z})/p$ vanishes.  But McCarthy's theorem implies that this limit diagram is constant and even further has value $\operatorname{F}(A_i')/p$, in which $x$ vanishes by assumption, whence the claim.

Thus we have reduced to the case where $A$ is a regular $\mathbb{F}_p$-algebra, and we want to see that $\operatorname{tr}(A)/p:\operatorname{K}(A)/p\rightarrow \operatorname{TC}(A)/p$ is a $\pi_\ast^{et}$-isomorphism.  Since $\mathbb{F}_p$ is perfect, $A$ is necessarily regular \emph{over} $\mathbb{F}_p$.  If furthermore $A$ is smooth over $\mathbb{F}_p$, then this is proved in \cite{GH} Theorem 4.2.2.  By Popescu's theorem every regular $\mathbb{F}_p$-algebra is a filtered colimit of smooth ones, so  it suffices to see that both $\operatorname{K}(-)/p$ and $\operatorname{TC}(-)/p$ commute with filtered colimits of smooth $\mathbb{F}_p$-algebras.  For $\operatorname{K}(-)/p$ this is automatic because $\operatorname{K}(-)$ commutes with all filtered colimits.  For $\operatorname{TC}(-)/p$, recall that $\operatorname{TC}(-)$ is the inverse limit of a tower $\operatorname{TC}^n(-)$ where each $\operatorname{TC}^n(-)$ commutes with filtered colimits.  But \cite{H} Theorem B implies that on each homotopy group, the limit diagram $\operatorname{TC}(-)/p\rightarrow \varprojlim_n \operatorname{TC}^n(-)/p$ is pro-constant in the category of functors $\operatorname{SmAlg}_{\mathbb{F}_p}\rightarrow \operatorname{Ab}$.  Thus it remains a limit diagram after passing to filtered colimits, proving the claim.
\end{proof}

\begin{remark} We crucially used the theorem from \cite{GH} that $\operatorname{tr}:\operatorname{K}(A)/p\rightarrow \operatorname{TC}(A)/p$ is an isomorphism on $\pi_\ast^{et}$ for a smooth $\mathbb{F}_p$-algebra $A$.  This is not an easy theorem --- it relies on a number of nontrivial precursors in both K-theory and in TC-theory.  In any case, if this theorem is true for an arbitrary $\mathbb{F}_p$-algebra $A$, then the regularity hypothesis on $\pi_0(-\otimes_\mathbb{Z}\mathbb{F}_p)$ can be dropped here and everywhere else it appears in this paper.\end{remark}

\begin{lemma} Recall that $\operatorname{F}$ denotes the fiber of $\operatorname{tr}:\operatorname{K}\rightarrow \operatorname{TC}$.  Let $\mathcal{P}\in \operatorname{PerfCat}_\mathbb{Z}$.
\begin{enumerate}
\item Suppose that $\operatorname{TC}(\mathcal{P})_{\widehat{p}}=0$.  Then $d_{K(1)}\operatorname{K}(\mathcal{P})\overset{\sim}{\longrightarrow} \operatorname{dK^{\operatorname{\operatorname{\operatorname{Sel}}}}}(\mathcal{P})$.
\item Suppose that $L_{K(1)}\operatorname{F}(\mathcal{P})=0$.  Then $d_{TC}\operatorname{TC}(\mathcal{P})\overset{\sim}{\longrightarrow} \operatorname{dK^{\operatorname{\operatorname{\operatorname{Sel}}}}}(\mathcal{P})$.
\end{enumerate}
\end{lemma}
\begin{proof}
Both claims are immediate from the pushout square defining $\operatorname{dK^{\operatorname{\operatorname{\operatorname{Sel}}}}}$.
\end{proof}

\begin{proposition} Let $\mathcal{P}\in \operatorname{PerfCat}_{\mathbb{Z}}$.
\begin{enumerate}
\item We have $\operatorname{TC}(\mathcal{P})_{\widehat{p}}=0$, and hence $d_{K(1)}\operatorname{K}(\mathcal{P})\overset{\sim}{\longrightarrow} \operatorname{dK^{\operatorname{\operatorname{\operatorname{Sel}}}}}(\mathcal{P})$, whenever $\mathcal{P}$ is $\mathbb{Z}[1/p]$-linear.
\item We have $L_{K(1)}\operatorname{F}(\mathcal{P})=0$, and hence $d_{TC}\operatorname{TC}(\mathcal{P})\overset{\sim}{\longrightarrow} \operatorname{dK^{\operatorname{\operatorname{\operatorname{Sel}}}}}(\mathcal{P})$, whenever one of the following holds:
\begin{enumerate}
\item $\mathcal{P}$ admits an $\mathbb{F}_p$-linear structure;
\item $\mathcal{P}=\operatorname{Perf}(A)$ for a connective DGA $A$ over $\mathbb{Z}/p^n\mathbb{Z}$, or $\mathcal{P}=\operatorname{Perf}(X)$ for a qcqs derived algebraic space $X$ over $\mathbb{Z}/p^n\mathbb{Z}$.
\item $\mathcal{P}=\operatorname{Perf}(A)$ for a quasi-commutative connective DGA $A$ over $\mathbb{Z}$ such that $(R,pR)$ is henselian, $R$ has bounded $p$-torsion, and either $R$ is local or $\operatorname{Spec}(R)$ is noetherian of finite Krull dimension ($R=\pi_0A$).
\end{enumerate}
\end{enumerate}
\end{proposition}

\begin{proof}
For 1, if $\mathcal{P}$ is $\mathbb{Z}[1/p]$-linear, then $\operatorname{TC}(\mathcal{P})_{\widehat{p}}$ is a module over $\operatorname{TC}(\mathbb{Z}[1/p])_{\widehat{p}}=\operatorname{TC}(\mathbb{Z}[1/p]_{\widehat{p}})_{\widehat{p}}=0$, hence is zero.

For 2(a), If $\mathcal{P}$ is $\mathbb{F}_p$-linear, then $L_{K(1)}F(\mathcal{P})_{\widehat{p}}$ is a module over $L_{K(1)}\operatorname{K}(\mathbb{F}_p)=L_{K(1)}H\mathbb{Z}_p=0$, hence is zero.

For 2(b), if $\mathcal{P}$ is $\mathbb{Z}/p^n\mathbb{Z}$-linear and the conclusion of McCarthy's theorem holds for $\mathcal{P}\rightarrow \mathcal{P}\otimes_{\mathbb{Z}/p^n\mathbb{Z}}\mathbb{F}_p$, then we deduce $L_{K(1)}F(\mathcal{P})=0$ from the $\mathbb{F}_p$-linear case.  By McCarthy's theorem this holds if $\mathcal{P}=\operatorname{Perf}(A)$ for a connective DGA $A$.  By localization it then also holds if $\mathcal{P}=\operatorname{Perf}(X)$ for a qcqs derived algebraic space $X$.

2(c) follows from the $p$-adic continuity (Theorem \ref{padiccontinuity}) and claim 2(b).\end{proof}

\begin{remark}
The upshot is that away from $p$, the theory $\operatorname{dK^{\operatorname{\operatorname{\operatorname{Sel}}}}}$ is controlled by $L_{K(1)}\operatorname{K}$, whereas formally completed at $p$ (or even henselian at $p$), the theory $\operatorname{dK^{\operatorname{\operatorname{\operatorname{Sel}}}}}$ is controlled by $\operatorname{TC}_{\widehat{p}}$.
\end{remark}
\begin{remark}
In contrast, when $\mathcal{P}$ is genuinely global and mixed characteristic, then all three of the terms contribute to the pushout defining $\operatorname{dK^{\operatorname{\operatorname{\operatorname{Sel}}}}}(\mathcal{P})$, in a manner reminiscent of the definition of Selmer groups.
\end{remark}
\begin{remark}
When $\mathcal{P}=\operatorname{Perf}(X)$ for an algebraic space $X$ of finite type over $\mathbb{Z}$ with $(X\otimes_\mathbb{Z}\mathbb{F}_p)^{red}$ regular, one can think that the term $d_{K(1)}\operatorname{K}(X)$ is associated to the etale theory of the algebraic space $X_{\mathbb{Z}[1/p]}$, that the term $d_{K(1)}\operatorname{TC}(X)$ is associated to the the etale theory of the rigid-analytic space $X^{an}_{\mathbb{Q}_p}$, and that the term $d_{TC}\operatorname{TC}(X)$ is associated to the de Rham theory of the formal algebraic space $X_{\widehat{p}}$.  The gluing together of these terms has something to do with $p$-adic Hodge theory.
\end{remark}

For the purposes of this paper, the main thing we will want to prove about $\operatorname{dK^{\operatorname{\operatorname{\operatorname{Sel}}}}}$ is the following, which will be proved in Section \ref{proofeiso}.

\begin{theorem}\label{eiso}
Let $X$ be a locally noetherian derived algebraic space over $\mathbb{Z}$, and suppose $(X\otimes_\mathbb{Z}\mathbb{F}_p)^{red}$ is regular.

Then there is a natural map
$$e_X:H_1(X_{et};\mathbb{Z}_p)\rightarrow \pi_1 \operatorname{dK^{\operatorname{\operatorname{\operatorname{Sel}}}}}(X),$$
where $H_1(X_{et};\mathbb{Z}_p)$ stands for the pro-$p$-abelian completion of the etale fudamental group of $X$, or equivalently the Pontryagin dual of $H^1(X_{et};\mathbb{Q}_p/\mathbb{Z}_p)$.

The map $e_X$ is an isomorphism if one of the following conditions holds:
\begin{enumerate}
\item $X$ has (mod $p$) etale cohomological dimension $\leq 2$ and $X\otimes\mathbb{F}_p$ has (mod $p$) etale cohomological dimension $\leq 1$.
\item $X=\operatorname{Spec}(F)$ for a field $F$ of virtual (mod $p$) Galois cohomological dimension $\leq 2$.
\end{enumerate}

\end{theorem}

\begin{remark}
More specifically, this map $e_X$ will be an edge map in a ``co-descent" spectral sequence for $\operatorname{dK^{\operatorname{\operatorname{\operatorname{Sel}}}}}$.  Especially when $X$ lives over $\mathbb{Z}[1/p]$, it behaves similarly to the comparison map $H_1(X;\mathbb{Z})\rightarrow K_1(X)$ between singular homology and K-homology in topology, which is an edge map in the Atiyah-Hirzebruch spectral sequence. 
\end{remark}

The above theorem will follow from two others: an etale co-descent result for $\operatorname{dK^{\operatorname{\operatorname{\operatorname{Sel}}}}}$, and a (partial) calculation of the co-stalks of $\operatorname{dK^{\operatorname{\operatorname{\operatorname{Sel}}}}}$.  We start with the latter.

\subsection{The etale co-stalks of Selmer K-homology}

To study the ``etale co-stalks of $\operatorname{dK^{\operatorname{\operatorname{\operatorname{Sel}}}}}$" means to study the values $\operatorname{dK^{\operatorname{\operatorname{\operatorname{Sel}}}}}(A)$ when $A$ is a strictly henselian local ring (or, in the derived case, a quasi-commutative connective DGA over $\mathbb{Z}$ whose $\pi_0$ is a strictly henselian local ring; we will also call such an object a strictly henselian local ring).  We start with the case of residue characteristic $\neq p$.  The following is a standard consequence of Gabber-Suslin rigidity (\cite{G}, \cite{Sus}).

\begin{proposition}\label{GaSu}
Let $A$ be a strictly henselian local ring with residue characteristic $\neq p$.  Then:
\begin{enumerate}
\item Let $\mathbb{Z}_p(1)=T_p(A^\times)$ denote the Tate module of $p$-power roots of unity in $A$ (or $\pi_0A$).  There are natural isomorphisms
$$\mathbb{Z}_p(1)\simeq \pi_2 \operatorname{K}(A)_{\widehat{p}}\overset{\sim}{\rightarrow} \pi_2L_{K(1)}\operatorname{K}(A),$$
and $\pi_\ast \operatorname{K}(A)_{\widehat{p}}$ (resp. $\pi_\ast L_{K(1)}\operatorname{K}(A)$) is a polynomial algebra (resp. Laurent polynomial algebra) over $\mathbb{Z}_p$ on the invertible $\mathbb{Z}_p$-module $\mathbb{Z}_p(1)$ placed in degree 2.
\item The above identifications are functorial under all ring homomorphisms between such $A$'s (not just the local ones).
\end{enumerate}
\end{proposition}
\begin{proof}
Since $A$ is local with residue characteristic $\neq p$, it follows that $p$ is a unit in $A$.  We deduce $\operatorname{K}(A)_{\widehat{p}}\overset{\sim}{\rightarrow} \operatorname{K}(\pi_0A)_{\widehat{p}}$, either by comparison either with $\operatorname{TC}$ using McCarthy's theorem (Theorem \ref{mccarthy}) or by comparison with the cohomology of $\operatorname{BGL}(-)$ using the group-completion theorem. Hence we can assume $A$ is an ordinary commutative ring.  (This also explains why $\pi_\ast K(A)_{\widehat{p}}$ should have a product structure even when $A$ is only a DGA.)

Then this is a standard consequence of Gabber-Suslin rigidity (\cite{G}, \cite{Sus}), which gives $\operatorname{K}(A)_{\widehat{p}}\simeq ku_{\widehat{p}}$, the equivalence being induced by a zig-zag of commutative ring homomorphisms between $A$ and $\mathbb{C}$.  Let us just recall that the identification $\mathbb{Z}_p(1)=\pi_2 \operatorname{K}(A)_{\widehat{p}}$ arises from the short exact sequence calculating the homotopy group of a $p$-completion (\cite{Bo} Proposition 2.5) and the identification $A^\times=\pi_1\operatorname{K}(A)$.
\end{proof}

\begin{corollary}\label{k(1)stalk}
With $A$ as in the previous proposition, there is a unique class
$$j_A\in \pi_0 \operatorname{dK^{\operatorname{\operatorname{\operatorname{Sel}}}}}(A)=\pi_0d_{K(1)}\operatorname{K}(A) = [\operatorname{K}(A),\omega_{K(1)}]$$
such that
$$(\pi_0j_A)(1) = 2\cdot [S^1] \in \pi_0\omega_{K(1)}.$$
This class $j_A$ is functorial under all ring homomorphisms, and multiplication by $j_A$ gives an equivalence $L_{K(1)}\operatorname{K}(A)\overset{\sim}{\rightarrow} \operatorname{dK}^{\operatorname{Sel}}(A)$.  In particular there are canonical functorial isomorphisms $\pi_{2j}  \operatorname{dK}^{\operatorname{Sel}}(A)\simeq\operatorname{Hom}(\mathbb{Z}_p(-j),\mathbb{Z}_p)$ and $\pi_{2j+1}  \operatorname{dK}^{\operatorname{Sel}}(A)=0$ for all $j\in\mathbb{Z}$.
\end{corollary}
\begin{proof}
If we choose an equivalence $\operatorname{K}(A)_{\widehat{p}}\simeq ku_{\widehat{p}}$ of ring spectra and use the equivalence $\operatorname{log}:\omega_{K(1)}\simeq \Sigma L_{K(1)}S$, we get $[\operatorname{K}(A),\omega_{K(1)}]\simeq [KU,\Sigma L_{K(1)}S]$.  Then the existence and uniqueness of a class $j_A\in [\operatorname{K}(A),\omega_{K(1)}]$ such that $(\pi_0 j_A)(1)=2\cdot[S^1]$, as well as the property $\cdot j_A:L_{K(1)}\operatorname{K}(A)\simeq d_{K(1)}\operatorname{K}(A)$, follow from an easy calculation in the $K(1)$-local category using the $KU$-based Adams spectral sequence (cf \cite{HM} Lemma 8.16), and the functoriality follows from the uniqueness.
\end{proof}

\begin{remark}\label{etalej} Just as with the $j_{\mathbb{Z}_p}$ and $j_{\mathbb{R}}$ from earlier (Remark \ref{geomj}), when $A$ is an ordinary commutative ring this homotopy class of maps $j_A\in [L_{K(1)}\operatorname{K}(A), \omega_{K(1)}]$ can be identified with the $K(1)$-localization of a certain canonical map of spectra $J^{et}_A:\operatorname{K}(A)\rightarrow \operatorname{Pic}(S_{\widehat{p}})$.  This is the etale J-homomorphism of \cite{Q2} and \cite{F}, which sends the class of a finite free $A$-module $M$ to the stable etale $p$-adic homotopy type of the cofiber $\underline{M}/(\underline{M}-0_{\operatorname{Spec}(A)})$, where $\underline{M} = \operatorname{Spec}(\operatorname{Sym}_A(M^\vee))$ is the vector bundle corresponding to $M$.  One sees that even this unlocalized $J^{et}_A$ is well-defined and functorial in $A$ using standard facts in etale homotopy theory.  Furthermore $J^{et}_{\mathbb{C}}$ defined in this way is naturally homotopic to the composition $\operatorname{K}(\mathbb{C})\rightarrow \operatorname{K}(\mathbb{R})\overset{J_{\mathbb{R}}}{\longrightarrow} \operatorname{Pic}(S_{\widehat{p}})$, by the comparison between etale homotopy and classical homotopy over $\mathbb{C}$.  In these terms it is geometrically clear why $j_A$ should send $1$ to $2\cdot [S^1]$: when $M$ is the unit $A$-module, $\underline{M}/(\underline{M}-0_{\operatorname{Spec}(A)}) = \mathbb{A}^1/(\mathbb{A}^1-0)$ is a $2$-sphere in etale homotopy theory.
\end{remark}

The characteristic $p$ counterpart to Proposition \ref{GaSu} is the following:

\begin{proposition}
Let $A$ be a strictly henselian local ring of residue characteristic $p$.  Suppose furthermore that $\pi_0(A\otimes_\mathbb{Z}\mathbb{F}_p)$ is a nilpotent thickening of a regular ring and $\pi_0A$ has bounded $p$-torsion.  Then:
\begin{enumerate}
\item The trace map $\operatorname{K}(A)_{\widehat{p}}\rightarrow \operatorname{TC}(A)_{\widehat{p}}$ is an equivalence.
\item We have $\pi_0 \operatorname{TC}(A)_{\widehat{p}}=\mathbb{Z}_p$ and $\pi_\ast \operatorname{TC}(A)_{\widehat{p}}=0$ for $\ast < 0$.
\end{enumerate}
\end{proposition}
\begin{proof}
Since $k$ has characteristic $p$ and $A$ is henselian, the pair $(\pi_0A,p\pi_0A)$ is henselian.  Thus Theorem \ref{GHetale} shows that $\operatorname{K}(A)\rightarrow \operatorname{TC}(A)$ is an isomorphism on (mod $p$) homotopy after sheafification over $\operatorname{Spec}(A)_{et}$.  But $A$ is strictly henselian, so this means it is already an isomorphism on (mod $p$) homotopy, hence on $p$-completion, giving claim 1. Claim 2 is immediate from claim 1.
\end{proof}
\begin{corollary}\label{tcstalk}
Let $A$ be as in the previous proposition. Then there is a unique class
$$c_A\in \pi_0 \operatorname{dK^{\operatorname{\operatorname{\operatorname{Sel}}}}}(A) = \pi_0d_{TC}\operatorname{TC}(A) = [\operatorname{TC}(A),\omega_{TC}]_{\operatorname{TC}(\mathbb{Z})}$$
such that
$$(\pi_0c_A)(1) = \epsilon \in \pi_0\omega_{TC}.$$
Furthermore this class $c_A$ is functorial in $A$ for all ring homomorphisms, $\pi_0 \operatorname{dK}^{\operatorname{Sel}}(A)$ is a free $\mathbb{Z}_p$-module on $c_A$, and $\pi_n  \operatorname{dK}^{\operatorname{Sel}}(A)=0$ for all $n>0$.
\end{corollary}
\begin{proof}
From Theorem \ref{anderson} we have that $\omega_{TC}$ identifies with the $p$-adic Anderson dual of $\operatorname{TC}(\mathbb{Z})_{\widehat{p}}$. Thus for a $\operatorname{TC}(\mathbb{Z})_{\widehat{p}}$-module $M$ and $n\in\mathbb{Z}$, there is a short exact sequence
$$0\rightarrow \operatorname{Ext}_{\mathbb{Z}_p}(\pi_{n-1}M,\pi_0\omega_{TC})\rightarrow \pi_{-n}d_{TC}M\rightarrow \operatorname{Hom}_{\mathbb{Z}_p}(\pi_nM,\pi_0\omega_{TC})\rightarrow 0$$
where the quotient map records the effect of a map $\Sigma^{-n}M\rightarrow \omega_{TC}$ on $\pi_0$.  Furthermore, $\pi_0\omega_{TC}$ is a free $\mathbb{Z}_p$-module on $\epsilon\in \pi_0\omega_{TC}$.  The claim follows.
\end{proof}

\begin{remark}\label{zpremark}
The $\pi_0$ part of the corollary also holds for $A=\mathbb{Z}_p$.  Indeed, $\pi_0\operatorname{TC}(\mathbb{Z}_p)_{\widehat{p}}=\mathbb{Z}_p$, whereas $\pi_{-1}\operatorname{TC}(\mathbb{Z}_p)_{\widehat{p}}$ is free of rank $1$ over $\mathbb{Z}_p$ and hence Ext's out of it vanish.  Thus the same argument works.
\end{remark}

Now we put the two cases (residue characteristic $p$ vs. $\neq p$) together to gain more global information about $\pi_0\operatorname{dK^{\operatorname{\operatorname{\operatorname{Sel}}}}}$.  For this the key is the following, which is already implicit in the proof of Proposition \ref{localduality}:

\begin{lemma}\label{gluelemma}
Let $\overline{\mathbb{Q}_p}$ denote an algebraic closure of $\mathbb{Q}_p$.  The composition
$$R=L_{K(1)}\operatorname{K}(\mathbb{Z}_p)\rightarrow L_{K(1)}\operatorname{K}(\overline{\mathbb{Q}_p})\overset{j_{\overline{\mathbb{Q}_p}}}{\longrightarrow}\omega_{K(1)}$$
is homotopic to $j_{\mathbb{Z}_p}\cdot \epsilon$.
\end{lemma}
\begin{proof}
By Theorem \ref{jchar}, we can check this by seeing that the two maps in question have the same image on $1,\epsilon\in \pi_0 R$.  Again by \ref{jchar}, $j_{\mathbb{Z}_p}\cdot\epsilon$ kills $\epsilon$ and sends $1$ to $2\cdot [S^1]$.  The above composition kills $\epsilon$ because $\pi_1 L_{K(1)}\operatorname{K}(\overline{\mathbb{Q}_p})\simeq \pi_1 KU_{\widehat{p}}=0$, and it sends $1$ to $2\cdot [S^1]$ by the construction of $j_{\overline{\mathbb{Q}_p}}$ (\ref{k(1)stalk}).
\end{proof}

\begin{theorem}\label{costalks}
Let $A$ denote a strictly henselian local ring.  If the residue characteristic equals $p$, assume that $\pi_0A$ has bounded $p$-torsion and that $\pi_0(A\otimes\mathbb{F}_p)$ is a nilpotent thickening of a regular ring.
\begin{enumerate}
\item The isomorphism
$$\pi_0 \operatorname{dK^{\operatorname{\operatorname{\operatorname{Sel}}}}}(A)\simeq \mathbb{Z}_p,$$
given by the class $j_A$ in residue characteristic $\neq p$ and $c_A$ in residue characteristic $p$, is functorial under all ring homomorphisms between such $A$'s, where $\mathbb{Z}_p$ means the constant functor with value $\mathbb{Z}_p$.
Furthermore,
$$\pi_1 \operatorname{dK^{\operatorname{\operatorname{\operatorname{Sel}}}}}(A)=0.$$
\item If we require $A$ to have residue characteristic $\neq p$, then functorially
$$\pi_{2n} \operatorname{dK^{\operatorname{\operatorname{\operatorname{Sel}}}}}(A) \simeq \operatorname{Hom}_{\mathbb{Z}_p}(\mathbb{Z}_p(-n),\mathbb{Z}_p),$$
and
$$\pi_{2n+1} \operatorname{dK^{\operatorname{\operatorname{\operatorname{Sel}}}}}(A)=0.$$
\end{enumerate}
\end{theorem}
\begin{proof}
All of the claims are clear from Corollaries \ref{k(1)stalk} and \ref{tcstalk} except that the isomorphism $\pi_0 \operatorname{dK^{\operatorname{\operatorname{\operatorname{Sel}}}}}(A)\simeq \mathbb{Z}_p$ is functorial under homomorphisms $A\rightarrow B$ when $A$ and $B$ have different residue characteristics.  In this case $A$ has residue characteristic $p$ and $B$ has residue characteristic $\neq p$.  But since we know functoriality in residue characteristic $p$ and $\neq p$ separately, and the functor $\pi_0\operatorname{dK^{\operatorname{\operatorname{\operatorname{Sel}}}}}(-)$ lands in isomorphisms in both cases, we can use zig-zags of maps to reduce to where $A=W(\overline{\mathbb{F}_p})$ and $B=\overline{\operatorname{Frac}(A)}$.  By Remark \ref{zpremark} we can even take $A=\mathbb{Z}_p$ instead, so $B=\overline{\mathbb{Q}_p}$.  In this case, we find after unwinding the definitions that the claim is equivalent to saying that the composition
$$R\rightarrow L_{K(1)}\operatorname{K}(\overline{\mathbb{Q}_p})\overset{j_{\overline{\mathbb{Q}_p}}}{\longrightarrow}\omega_{K(1)}$$
is homotopic to $\epsilon\cdot j_{\mathbb{Z}_p}$ up to a $\mathbb{Z}_p$-multiple of $j_{\mathbb{Z}_p}$.  But Lemma \ref{gluelemma} says that it is exactly homotopic to $\epsilon\cdot j_{\mathbb{Z}_p}$.
\end{proof}

\subsection{Etale co-descent for Selmer K-homology}\label{dualtrick}

It is a bit awkward to talk about co-sheaves and co-descent, so following \cite{Mi} we will use a trick with duality to reduce to sheaves instead.  We start with the observation that $\operatorname{dK^{\operatorname{\operatorname{\operatorname{Sel}}}}}$ canonically takes values in a more refined $\infty$-category than just $\operatorname{Sp}$:

\begin{definition}
Let $\operatorname{Sp}^{\pi}\subset \operatorname{Sp}$ denote the thick subcategory consisting of those spectra all of whose homotopy groups are finite, and let $\operatorname{Pro}(\operatorname{Sp}^\pi)$ denote its pro-category.  Since $\operatorname{Sp}^\pi$ is stable, it is co-tensored over finite spectra; therefore $\operatorname{Pro}(\operatorname{Sp}^\pi)$ is co-tensored over Ind of finite spectra, which is all spectra.  So $\operatorname{map}(X,P)\in \operatorname{Pro}(\operatorname{Sp}^\pi)$ naturally for $X\in \operatorname{Sp}$ and $P\in \operatorname{Pro}(\operatorname{Sp}^\pi)$.

If $\omega$ is a $p$-complete spectrum such that $\omega/p\in \operatorname{Sp}^\pi$, then we can canonically lift $\omega$ to an object of $\operatorname{Pro}(\operatorname{Sp}^\pi)$, namely
$$\omega = ``\varprojlim_n" \omega/p^n,$$
and thereby also
$$\operatorname{map}(X,\omega)\in \operatorname{Pro}(\operatorname{Sp}^\pi)$$
functorially in $X\in \operatorname{Sp}$.  Explicitly, if we write $X = \varinjlim_i X_i$ with $X_i$ finite, then
$$\operatorname{map}(X,\omega) = ``\varprojlim_{i,n}" \operatorname{map}(X_i,\omega/p^n).$$
\end{definition}

A similar definition works if $\omega$ is module over a ring spectrum, e.g. $\omega_{TC}$ or $R$ over $\operatorname{TC}(\mathbb{Z})$.  Since the homotopy of both $\omega_{K(1)}/p$ and $\omega_{TC}/p$ is finite in all degrees, we can thus canonically view $d_{K(1)}$ and $d_{TC}$, and therefore $\operatorname{dK^{\operatorname{\operatorname{\operatorname{Sel}}}}}$, as taking values in $\operatorname{Pro}(\operatorname{Sp}^\pi)$.  From now on we do this implicitly.

\begin{remark}
The standard $t$-structure on $\operatorname{Sp}$ restricts to a $t$-structure on $\operatorname{Sp}^\pi$ whose heart is the abelian category of finite abelian groups.  It follows that $\operatorname{Pro}(\operatorname{Sp}^\pi)$ gets a $t$-structure whose heart is Pro of the category of finite abelian groups, i.e. the category of profinite abelian groups.  In particular each homotopy group $\pi_\ast \operatorname{dK^{\operatorname{\operatorname{\operatorname{Sel}}}}}(\mathcal{P})$ is canonically a profinite abelian group.  The isomorphisms of the previous section all promote to isomorphisms of profinite abelian groups (necessarily, since the structure of profinite group on $\mathbb{Z}_p$ is unique).
\end{remark}

Let $(-)^\#$ denote the Pontryagin duality $Hom_c(-;\mathbb{Q}/\mathbb{Z})$, implementing an anti-equivalence between torsion abelian groups and profinite abelian groups.  Recall that $(-)^\#$ admits a lift to spectra: there is a contravariant involutive self-equivalence
$$X\mapsto X^\#$$
on $\operatorname{Sp}^\pi$, called Brown-Comenentz duality, such that there are canonical isomorphisms $$\pi_n (X^\#)\simeq (\pi_{-n}X)^\#$$
for all $n\in\mathbb{Z}$.  More precisely, we have $X^\#:=\operatorname{map}(X,I_{\mathbb{Q}/\mathbb{Z}})$ where $I_{\mathbb{Q}/\mathbb{Z}}$ is a certain spectrum with $\pi_0I_{\mathbb{Q}/\mathbb{Z}}=\mathbb{Q}/\mathbb{Z}$, and the canonical isomorphism $\pi_n (X^\#)\simeq (\pi_{-n}X)^\#$ comes from the induced pairing
$$\pi_n X^\# \otimes \pi_{-n}X\rightarrow \pi_0I_{\mathbb{Q}/\mathbb{Z}}.$$

If we apply Brown-Comenentz duality termwise to an object $X\in \operatorname{Pro}(\operatorname{Sp}^\pi)$, we get an object of $\operatorname{Ind}(\operatorname{Sp}^\pi)$.  But we may as well take the colimit of that Ind-system in spectra to get a spectrum which we also denote $X^\#$, whose homotopy groups are canonically
$$\pi_n (X^\#) = (\pi_{-n}X)^\#.$$
Thus the homotopy of the plain spectrum $X^\#$ as a plain abelian group is torsion, and it determines and is uniquely determined by the homotopy of $X\in \operatorname{Pro}(\operatorname{Sp}^\pi)$ as a profinite abelian group, by Pontryagin duality.

The trick from \cite{Mi} it to consider the presheaf $X\mapsto \operatorname{dK}^{\operatorname{\operatorname{\operatorname{Sel}}}}(X)^\#$ of honest spectra instead of the co-presheaf $X\mapsto \operatorname{dK}^{\operatorname{\operatorname{\operatorname{Sel}}}}(X)$ of Pro-$\pi_\ast$-finite spectra; by the above discussion, they carry the same information on homotopy groups anyway.

\begin{remark}\label{moore}
The functor $d_{K(1)}(-)^\#$ from spectra to spectra preserves colimits.  Therefore we have a canonical equivalence
$$d_{K(1)}(X)^\# \simeq X\wedge (d_{K(1)}S)^\#.$$
Thus $d_{K(1)}(-)^\#$ is equivalent to smashing with a certain $p$-power torsion $L_1$-local spectrum, namely $(d_{K(1)}S)^\#$.  (The spectrum is also $W\wedge M\mathbb{Q}_p/\mathbb{Z}_p$, where $W$ is the $K(1)$-local ``fake $S^0$" of Remark \ref{K(1)anderson} and \cite{HM} Lemma 7.3).

Similarly, since $d_{TC}$ is equivalent to $p$-adic Anderson duality (Theorem \ref{anderson}), the functor $d_{TC}(-)^\#$ is equivalent to smashing with a Moore spectrum $M\mathbb{Q}_p/\mathbb{Z}_p$.
\end{remark}

With all this as background, we have:

\begin{theorem}\label{sheaf}
The contravariant functor
$$X\mapsto \operatorname{dK^{\operatorname{\operatorname{\operatorname{Sel}}}}}(X)^\#$$
from qcqs derived algebraic spaces to spectra is an etale sheaf.

If we restrict to $X$ which are virtually of finite (mod $p$) etale cohomological dimension and which have bounded $p$-torsion in $\pi_0$ of their structure sheaf, then it is even an etale Postnikov sheaf (it is the inverse limit of its etale-sheafified Postnikov tower).
\end{theorem}

\begin{proof}
It suffices to prove the claims separately for $(d_{TC}\operatorname{TC})^\#$, $(d_{K(1)}\operatorname{TC})^\#$, and $(d_{K(1)}\operatorname{K})^\#$.  For the last two, since they satisfy localization and are valued in $L_1$-local spectra, \cite{CMNN} Theorem A.14 implies that they are automatically etale sheaves, and in \cite{CM} it will be seen that under the bounded $p$-torsion and finite virtual dimension hypotheses they are even automatically etale Postnikov sheaves. (Both works use some of Thomason's methods from \cite{Th1}; one can also deduce the claim in a majority of cases of interest from Thomason's work itself.)  Thus it suffices to consider $d_{TC}\operatorname{TC}^\#$.

Since $(d_{TC})^\#$ is equivalent to smashing with a Moore spectrum $M\mathbb{Q}_p/\mathbb{Z}_p$ (Remark \ref{moore}), it has $t$-amplitude $\leq 1$. It therefore suffices to see that $\operatorname{TC}(-)$ is an etale Postnikov sheaf on qcqs derived algebraic spaces.  This is basically \cite{GH} Corollary 3.2.2 in a slightly different language; a proof in the current language, following the same idea as \cite{GH}, will be recorded in \cite{CM}.
\end{proof}

\begin{remark}
There is also a more general result: if $\mathcal{P}$ is tensored over $\operatorname{Perf}(X)$, then $(U\overset{et}{\longrightarrow} X)\mapsto \operatorname{dK^{\operatorname{\operatorname{\operatorname{Sel}}}}}(\mathcal{P}\otimes_{\operatorname{Perf}(X)}\operatorname{Perf}(U))^\#$ is an etale sheaf or etale Postnikov sheaf under the same hypotheses, with the same proof.
\end{remark}

\begin{corollary}\label{descentss}

In the situation above when $\mathcal{F}=\operatorname{dK^{\operatorname{\operatorname{\operatorname{Sel}}}}}(-)^\#$ is an etale Postnikov sheaf, we get a conditionally convergent ``descent" spectral sequence
$$E^1_{i,j} = H^{j-i}(X_{et};\pi^{et}_j \mathcal{F})\Rightarrow \pi_i(\mathcal{F}(X)).$$
Here $\pi^{et}_j(-)$ denotes the etale sheafification of $\pi_j(-)$.  We have indexed this spectral sequence so that $d_1$ goes from $E^1_{i,j}$ to $E^1_{i-1,j+1}$ and $d_2$ from $E^2_{i,j}$ to $E^2_{i-1,j+2}$ and so on.

\end{corollary}
\subsection{Sheafified homotopy groups and the edge map}
We would like to dualize Theorem \ref{costalks} to obtain information on the stalks of $\pi_j (\operatorname{dK^{\operatorname{\operatorname{\operatorname{Sel}}}}}(-)^\#)=(\pi_{-j}\operatorname{dK^{\operatorname{\operatorname{\operatorname{Sel}}}}}(-))^\#$.  But for this we need to be sure that $\operatorname{dK^{\operatorname{\operatorname{\operatorname{Sel}}}}}(-)^\#$ commutes with the appropriate colimits, so that its stalks are indeed the same as its values on strictly henselian local rings.  Here is the lemma which assures this:

\begin{lemma}
Let $X$ be a derived algebraic space such that classical algebraic space underlying $(X\otimes_{\mathbb{Z}}\mathbb{F}_p)$ is locally a nilpotent thickening of something regular, and $X$ locally has bounded $p$-torsion in its structure sheaf on $\pi_0$.  If
$$Y = \varprojlim_{i\in I} U_i$$
is a filtered limit of affine schemes etale over $X$ defining a strict henselization $Y$ of $X$, then $(\operatorname{dK^{\operatorname{\operatorname{\operatorname{Sel}}}}}(-))^\#$ sends this limit diagram to a colimit diagram.
\end{lemma}
\begin{proof}
Both functors $(d_{K(1)})^\#$ and $(d_{TC})^\#$ are equivalent to smashing with some $p$-primary torsion spectrum (Remark \ref{moore}), so it suffices to see that $\operatorname{K}(-)/p$ and $\operatorname{TC}(-)/p$ commute with the filtered colimit in question.  For $\operatorname{K}(-)/p$ this is automatic.  For $\operatorname{TC}(-)/p$, recall the natural transformation
$$\operatorname{K}((-)^h)/p\rightarrow \operatorname{TC}(-)/p$$
from Theorem \ref{GHetale}.  The functor $\operatorname{K}((-)^h)$ preserves filtered colimits, so it suffices to see that the fiber $G$ of this natural transformation preserves the filtered colimit in question.  But Theorem \ref{GHetale} implies both that the filtered colimit of the $G(U_i)$ is zero, and that $G(Y)$ is zero.
\end{proof}

Combining with Theorem \ref{costalks}, we get:

\begin{proposition}\label{stalks}
For $X$ a derived algebraic space such that the classical algebraic space underlying $(X\otimes_{\mathbb{Z}}\mathbb{F}_p)$ is locally a nilpotent thickening of something regular and $X$ locally has bounded $p$-torsion in $\pi_0$ of its structure sheaf, we have the following information on the stalks of $\pi_\ast (\operatorname{dK^{\operatorname{\operatorname{\operatorname{Sel}}}}}(-)^\#)$ over $X_{et}$:
\begin{enumerate}
\item The stalks of $\pi_0 (\operatorname{dK^{\operatorname{\operatorname{\operatorname{Sel}}}}}(-)^\#)$ are canonically $\mathbb{Q}_p/\mathbb{Z}_p$, compatibly with all co-specialization maps.
\item The stalks of $\pi_{2n}(\operatorname{dK^{\operatorname{\operatorname{\operatorname{Sel}}}}}(-)^\#)$ at points of characteristic $\neq p$ are canonically $\mathbb{Q}_p/\mathbb{Z}_p(n)$, compatibly with all co-specialization maps between characteristc $\neq p$ points, for all $n\in\mathbb{Z}$;
\item The stalks of $\pi_{2n+1}(\operatorname{dK^{\operatorname{\operatorname{\operatorname{Sel}}}}}(-)^\#)$ at points of characteristic $\neq p$ are zero, for all $n\in\mathbb{Z}$.
\item The stalks of $\pi_j(\operatorname{dK^{\operatorname{\operatorname{\operatorname{Sel}}}}}(-)^\#)$ at points of characteristic $p$ are zero for $j<0$.
\end{enumerate}
\end{proposition}

To go from knowledge of the stalks to knowledge of the sheaves themselves, we use the following lemma.

\begin{lemma}\label{stalktosheaf}
Let $X$ be an algebraic space whose underlying topological space is ``locally path connected" in the sense that it has a basis of open subsets whose category of points is connected. (Note that locally noetherian implies locally path connected). Denote by
$$p^\ast: \operatorname{Sh}(X)\rightarrow \operatorname{PSh}(\operatorname{pt}_X)$$
the functor from etale sheaves (of sets) on $X$ to presheaves on the category of points of the etale topos of $X$, given by taking stalks.

Suppose $\mathcal{F}\in \operatorname{Sh}(X)$ is such that $p^\ast\mathcal{F}$ sends every morphism in $\operatorname{pt}_X$ to an isomorphism (that is, every co-specialization map for $\mathcal{F}$ is an isomorphism).  Then for any $\mathcal{G}\in \operatorname{Sh}(X)$, the map
$$Hom(\mathcal{G},\mathcal{F})\rightarrow Hom(p^\ast\mathcal{G},p^\ast\mathcal{F})$$
is a bijection.
\end{lemma}
\begin{proof}
The functor $p^\ast$ has a right adjoint $p_\ast$, which can be determined as follows: let $F\in \operatorname{PSh}(\operatorname{pt}_X)$, and let $U\rightarrow X$ be be etale.  Then we need
$$Hom(U,p_\ast F) = Hom(p^\ast U,F),$$
which gives us the formula
$$(p_\ast F)(U) = \varprojlim_{\operatorname{pt}_U}F.$$
By adjunction, the claim holds if and only if the unit map $\mathcal{F}\rightarrow p_\ast p^\ast\mathcal{F}$ is an isomorphism.
Since $X_{et}$ has enough points, the functor $p^\ast$ detects isomorphisms.  Therefore by adjunction identities it suffices to see that if $F\in \operatorname{PSh}(\operatorname{pt}_X)$ sends every morphism in $\operatorname{pt}_X$ to an isomorphism, then the counit
$$p^\ast p_\ast F\rightarrow F$$
is a monomorphism.  In other words, given a point $x$ of $X_{et}$ and an etale neighborhood $U$ of $x$, we need to see that every element of $\varprojlim_{\operatorname{pt}_U}F$ is determined by its value at $x$, potentially after shrinking $U$.  But since $F$ sends all morphisms to isomorphisms, the map $\varprojlim_{\operatorname{pt}_U}F\rightarrow F(x)$ will itself be injective provided that every two objects in $\operatorname{pt}_U$ are connected by a zig-zag of maps.  This can be arranged by our local path connectivity hypothesis.\end{proof}

The following is an immediate corollary:
\begin{corollary}
Notation as in the lemma, if $\mathcal{F},\mathcal{G}\in \operatorname{Sh}(X)$ both have the property that all their co-specialization maps are isomorphisms, then to give an isomorphism $\mathcal{F}\simeq \mathcal{G}$ is equivalent to giving isomorphisms $\mathcal{F}_x\simeq \mathcal{G}_x$ for all $x\in \operatorname{pt}_X$ compatible with all co-specialization maps.
\end{corollary}

Combining this corollary with the description of the stalks of $\pi_j \operatorname{dK^{\operatorname{\operatorname{\operatorname{Sel}}}}}(-)^\#$ (Proposition \ref{stalks}), we get:

\begin{theorem}\label{homotopysheaves}
Let $X$ be a locally noetherian derived algebraic space such that $(X\otimes\mathbb{F}_p)^{red}$ is regular.  The etale sheafified homotopy group
$$\pi_j^{et}(\operatorname{dK^{\operatorname{\operatorname{\operatorname{Sel}}}}}(-)^\#)\in \operatorname{Sh}(X_{et};\operatorname{Ab})$$
satisfies:
\begin{enumerate}
\item When $j=0$, it is canonically isomorphic to the constant sheaf $\mathbb{Q}_p/\mathbb{Z}_p$.
\item When $j=-1$, it is zero.
\item If $X$ lives over $\mathbb{Z}[1/p]$, then for $j=2n$ it is canonically isomorphic to $(\mathbb{Q}_p/\mathbb{Z}_p)(n)$ and for $j=2n+1$ it is zero.
\end{enumerate}
\end{theorem}

In particular, the vanishing of this sheaf when $j=-1$ and its identification with $\mathbb{Q}_p/\mathbb{Z}_p$ when $j=0$ implies:

\begin{corollary} Let $X$ be as in the previous theorem.  There is an edge map in degree $-1$ for the descent spectral spectral sequence of $dK^{\operatorname{\operatorname{\operatorname{Sel}}}}(-)^\#$ over $X_{et}$ which goes
$$\pi_{-1} \operatorname{dK^{\operatorname{\operatorname{\operatorname{Sel}}}}}(X)^\#\rightarrow H^1(X_{et};\mathbb{Q}_p/\mathbb{Z}_p).$$
\end{corollary}

Note that one doesn't even need convergence of the spectral sequence, i.e.  that $\operatorname{dK^{\operatorname{\operatorname{\operatorname{Sel}}}}}(-)^\#$ is an etale Postnikov sheaf, for the edge map to be defined.

\begin{definition}
Let $X$ be as in the previous theorem.  Define the map
$$e_X:H_1(X_{et};\mathbb{Z}_p)\rightarrow \pi_1 \operatorname{dK^{\operatorname{\operatorname{\operatorname{Sel}}}}}(X)$$
to be the negative of the Pontryagin dual of the edge map of the previous corollary.
\end{definition}

The definition of $e_X$ proves the first part of Theorem \ref{eiso}.  Before proving the rest (giving conditions on when $e_X$ is an isomorphism), we record the following unwinding of the description of $e_X$ when $X=\operatorname{Spec}(F)$ for a field $F$ of characteristic $\neq p$.  There is a similar description (using $c_{\overline{F}}$ instead of $j_{\overline{F}}$) when $F$ has characteristic $p$, but we won't need that.

\begin{proposition}
Let $F$ be a field of characteristic $\neq p$, let $\widehat{g}\in H_1(\operatorname{Spec}(F)_{et};\mathbb{Z}_p)  = (\operatorname{G}_F)^{ab}_{\widehat{p}}$, and fix a separable closure $\overline{F}$ of $F$.  Denote by $i:K(F)\rightarrow K(\overline{F})$ the map induced by the inclusion.  Recall the homotopy class
$$j_{\overline{F}}\in [\operatorname{K}(\overline{F}),\omega_{K(1)}]$$
from Corollary \ref{k(1)stalk}.  Make the following choices:
\begin{enumerate}
\item Choose a representative $\operatorname{K}(\overline{F})\rightarrow\omega_{K(1)}$ for $j_{\overline{F}}$, and denote it just by $j$.
\item Choose a representative $g\in \operatorname{Gal}(\overline{F}/F)$ of $\widehat{g}$.
\item Choose a homotopy $\kappa:j\simeq j\circ g$.
\end{enumerate}
Then $e_{\operatorname{Spec}(F)}(\widehat{g})\in \pi_1 \operatorname{dK^{\operatorname{\operatorname{\operatorname{Sel}}}}}(F)=\pi_1\operatorname{Map}(\operatorname{K}(F), \omega_{K(1)})$ identifies with the class of the following self-homotopy of the map $j\circ i:\operatorname{K}(F)\rightarrow \omega_{K(1)}$:
$$j\circ i \overset{\kappa}{\simeq} (j\circ g)\circ i\simeq j\circ (g\circ i)=j\circ i.$$
\end{proposition}

\begin{corollary}\label{cyclo}
Let $F$ be a field of characteristic $\neq p$.  Then the composite
$$\operatorname{G}_F\rightarrow H_1(\operatorname{Spec}(F)_{et};\mathbb{Z}_p)\overset{e_F}{\longrightarrow} \pi_1 \operatorname{dK^{\operatorname{\operatorname{\operatorname{Sel}}}}}(F)=[\Sigma \operatorname{K}(F),\omega_{K(1)}]\overset{ev_{1\in \pi_0\operatorname{K}(F)}}{\longrightarrow} \pi_1\omega_{K(1)}$$
sends $g\in \operatorname{G}_F$ to $[\chi(g)]\in \pi_1 \omega_{K(1)}$, where $\chi:\operatorname{G}_F\rightarrow\mathbb{Z}_p^\times$ is the $p$-cyclotomic character.
\end{corollary}
\begin{proof}
If we choose a ring spectrum equivalence $\operatorname{K}(\overline{F})_{\widehat{p}}\simeq ku_{\widehat{p}}$ using Gabber-Suslin, then by looking on $\pi_2$ we deduce that the action of $g\in \operatorname{G}_F$ on $\operatorname{K}(\overline{F})_{\widehat{p}}$ identifies with the Adams operation $\psi^{\chi(g)}$ on $ku_{\widehat{p}}$.  Thus we need to see that the Toda bracket of
$$S\rightarrow ku_{\widehat{p}}\overset{1-\psi^u}{\longrightarrow}ku_{\widehat{p}}\rightarrow \omega_{K(1)}$$
is $[u]\in \pi_1\omega_{K(1)}$ for $u\in\mathbb{Z}_p^\times$, where the first map is the unit and the last map is the unique (up to homotopy) map which sends $1$ to $2\cdot [S^1]\in \pi_0\omega_{K(1)}$.  Note that this Toda bracket is independent of any choices, because $[S,\Omega ku_{\widehat{p}}]=[ku_{\widehat{p}},\Omega\omega_{K(1)}]=0$.

As usual, this is a simple calculation in $K(1)$-local homotopy theory using the log equivalence $\omega_{K(1)}\simeq \Sigma L_{K(1)}S$ (\ref{log}).  (The proof of the Adams conjecture also gives a less calculational proof -- see the following remark.)\end{proof}

\begin{remark}\label{Jedge}
If we use the representative $J^{et}_{\overline{F}}:\operatorname{K}(\overline{F})\rightarrow \operatorname{Pic}(S_{\widehat{p}})$ of $j_{\overline{F}}$ given by the etale J-homomorphism (Remark \ref{etalej}), then the previous proposition and corollary become much more vivid and canonical.

Indeed, in that case (following \cite{Sul}) we have a canonical choice of homotopy $J^{et}_{\overline{F}}\circ g\simeq J^{et}_{\overline{F}}$ coming from the action of $g$ on the pair of schemes $(\underline{M},\underline{M}- 0)$, so we get a canonical map of spectra $E(g):\operatorname{K}(F)\rightarrow \Sigma^{-1}\operatorname{Pic}(S_{\widehat{p}})=\operatorname{Aut}(S_{\widehat{p}})$ associated to $g\in\operatorname{Gal}(\overline{F}/F)$ which refines $e(\widehat{g})\in [\Sigma \operatorname{K}(F),\omega_{K(1)}]$.

Explicitly, $\Omega^\infty E(g)$ sends the class of the finite free $F$-module $M$ to the automorphism of the sphere $\mathbb{M}_{\overline{F}}/(\mathbb{M}_{\overline{F}}- 0)$ induced by the action of $g\in \operatorname{G}_F$.  Then Corollary \ref{cyclo} results from the fact that when $M$ is the unit module, the action of $g$ on the 2-sphere $\mathbb{M}_{\overline{F}}/(\mathbb{M}_{\overline{F}}- 0)= \mathbb{A}^1/\mathbb{G}_m$ is through the cyclotomic character on $H_2$.
\end{remark}

\subsection{The proof of Theorem \ref{eiso}}\label{proofeiso}

In the previous section we dualized the edge map in the descent spectral sequence for $\operatorname{dK^{\operatorname{\operatorname{\operatorname{Sel}}}}}(-)^\#$ to produce the comparison map
$$e_X: H_1(X_{et};\mathbb{Z}_p)\rightarrow \pi_1 \operatorname{dK^{\operatorname{\operatorname{\operatorname{Sel}}}}}(X)$$
whenever $X$ is a locally noetherian derived algebraic space with $(X\otimes\mathbb{F}_p)^{red}$ regular.  Now we prove Theorem \ref{eiso}, which gives conditions under which $e_X$ is an isomorphism.

\begin{theorem}
Suppose $X$ as above has (mod $p$) etale cohomological dimension $\leq 2$, and that $X\otimes\mathbb{F}_p$ has (mod $p$) etale cohomological dimension $\leq 1$.  Then $e_X$ is an isomoprhism.
\end{theorem}
\begin{proof}
By Theorem \ref{sheaf}, the descent spectral sequence 
$$E^1_{i,j} = H^{j-i}(X_{et};\pi^{et}_j \mathcal{F})\Rightarrow \pi_{i}\mathcal{F}(X)$$
for $\mathcal{F}=\operatorname{dK^{\operatorname{\operatorname{\operatorname{Sel}}}}}(-)^\#$ converges.  But:
\begin{enumerate}
\item All terms of cohomological degree $>2$ are zero, by the cohomological dimension hypothesis on $X$.
\item For $j$ odd all terms of cohomological degrees $>1$ are zero, because $\pi_j^{et}\mathcal{F}$ is then supported on $X\otimes\mathbb{F}_p$ by Theorem \ref{homotopysheaves}.
\item The sheaf $\pi_{-1}^{et}\mathcal{F}$ is $0$ by Theorem \ref{homotopysheaves}.
\end{enumerate}
Thus the only term contributing to $\pi_{-1} \operatorname{dK^{\operatorname{\operatorname{\operatorname{Sel}}}}}(X)^\#$ is $H^1(X_{et};\mathbb{Q}_p/\mathbb{Z}_p)$, and all differentials to and from $H^1(X_{et};\mathbb{Q}_p/\mathbb{Z}_p)$ are zero.  Hence the edge map in that degree is an isomorphism, as desired.  In other words, the claim is true for degree reasons.
\end{proof}

To finish the proof of Theorem \ref{eiso}, we need to see that $e_X$ is also an isomorphism if $X=\operatorname{Spec}(F)$ for a field $F$ of \emph{virtual} (mod $p$) Galois cohomological dimension $\leq 2$.  By the cases already handled above, we can assume that $F$ does not have honest (mod $p$) Galois cohomological dimension $\leq 2$.  This implies that $p=2$ and $F$ has characteristic zero.  (This follows by combining Artin-Schreier theory with results in the cohomology of profinite groups, as is nicely explained in \cite{Sc}.)

The difficulty with this case is that there are non-zero differentials in the spectral sequence, and they need to behave just right to ensure the result.  This is somewhat subtle: if we had for example used $\mathbb{Z}_p$-Anderson duality in place of $d_{K(1)}$ to define $\operatorname{dK}^{\operatorname{\operatorname{\operatorname{Sel}}}}(-)$, then the $E^1$-terms of the spectral sequence would be the same, but the differentials would behave differently and the result would fail, e.g.\ when $F=\mathbb{R}$ or $F=\mathbb{Q}$.

More specifically, a quick look at the spectral sequence shows that it suffices to prove the following:

\begin{proposition}\label{differentials}
Take $p=2$.  Let $F$ a field of virtual (mod $2$) Galois cohomological dimension $\leq d$ for some $d\in\mathbb{Z}_{\geq 0}$.  For brevity, let $\mathcal{F}$ denote the Postnikov sheaf $\operatorname{dK^{\operatorname{\operatorname{\operatorname{Sel}}}}}(-)^\#$ on the topos $\operatorname{Spec}(F)_{et}=\operatorname{BG}_F$.  Then in the region $j-i>d$ of the descent spectral sequence
$$E^1_{i,j} = H^{j-i}(\operatorname{BG}_F;\pi^{et}_j\mathcal{F})\Rightarrow \pi_{i}(\mathcal{F}),$$
the group $E^1_{i,j}=E^2_{i,j}$ is zero except when $j$ is even and $2i-j$ is $2$ (mod $4$), and the $d_2$-differential $E^2_{i,j}\rightarrow E^2_{i-1,j+2}$ is an isomorphism when $2i-j$ is $2$ (mod $8$).
\end{proposition}

Indeed, if we take $d=2$ then it follows from this proposition that the $d_2$ differential $E^2_{-1,0}\rightarrow E^2_{-2,2}$ must be zero (because the following $d_2:E^2_{-2,2}\rightarrow E^2_{-3,4}$ is an isomorphism), and then all of the higher differentials on $E_{-1,0}$ will also be zero because their targets will be zero from the $E^3$ page onwards.  Thus the full group $E^1_{-1,0}=H^1(\operatorname{BG}_F;\mathbb{Q}_2/\mathbb{Z}_2)$ contributes to $\pi_{-1}\mathcal{F}(F)$.  But on the other hand nothing else contributes, because all other terms of total degree $-1$ in the spectral sequence will be zero from $E^3$ onwards.

Thus, it suffices to prove Proposition \ref{differentials}.  For this, we first reduce to the case where $F$ is real-closed.  Recall the following theorem from \cite{Sc}:

\begin{theorem}
Let $F$ be a field of virtual (mod 2) cohomological dimension $\leq d$.  There exists a profinite space $X_F$ and a map of toposes
$$f:BC_2\times X_F\rightarrow \operatorname{BG}_F$$
such that:
\begin{enumerate}
\item The points of $X_F$ are the orderings on $F$;
\item For such a point $x\in X_F$, the corresponding map $f_x:BC_2\rightarrow \operatorname{BG}_F$ is induced by $BC_2=\operatorname{BG}_{F^x}\rightarrow \operatorname{BG}_F$ where $F^x$ denotes the real closure of $F$ with respect to the ordering $x$.
\item For any 2-torsion sheaf of abelian groups $M$ on $\operatorname{BG}_F$, the map on cohomology $H^i(\operatorname{BG}_F;M)\rightarrow H^i(BC_2\times X_F;f^\ast M)$ is an isomorphism for $i>d$.
\end{enumerate}
\end{theorem}
It follows from this that for an etale Postnikov sheaf of spectra $\mathcal{M}$ on $\operatorname{BG}_F$, the comparison map of descent spectral sequences for $\mathcal{M}$ and for $f^\ast \mathcal{M}$ is an isomorphism on $E^1$ in the region $j-i>d$.  Because our sheaf of interest $\mathcal{M}=dK^{\operatorname{\operatorname{\operatorname{Sel}}}}(-)^\#$ has $d_1=0$ for degree reasons, it follows we need only prove the analog of Proposition \ref{differentials} for $f^\ast\mathcal{M}$.  On the other hand, since $X_F$ is profinite and hence of cohomological dimension 0, for $M$ as in 3 the assignment $U \mapsto H^i(BC_2\times U;f^\ast M)$ on opens $U\subset X_F$ is a sheaf (namely, it is $R^i\pi_\ast f^\ast M$ for $\pi:BC_2\times X_F\rightarrow X_F$), and indeed the descent spectral sequence for $f^\ast \mathcal{M}$ is the global sections of a spectral sequence of sheaves over $X_F$.  The claims in Proposition \ref{differentials} can be checked on the stalks, which are the descent spectral sequences for $f_x^\ast\mathcal{M}$ by a simple base-change comparison.  This gives the reduction to the case where $F=F^x$ is real-closed.

Thus, to prove Proposition \ref{differentials} and hence the last remaining case of Theorem \ref{eiso} it suffices to show:

\begin{theorem}
Let $F$ be a real-closed field.  Consider the homotopy fixed point spectral sequence
$$E^1_{i,j}=H^{j-i}(BC_2; \pi_j N)\Rightarrow \pi_{i} (N^{C_2})$$
for $C_2=\operatorname{Gal}(F(\sqrt{-1})/F)$ acting on $N:=d_{K(1)}\operatorname{K}(F(\sqrt{-1}))^\#$.  In the range $j-i>0$, the group $E^1_{i,j}$ is zero except when $j$ is even and $2i-j$ is $2$ (mod $4$), and the $d_2$-differential $E^2_{i,j}\rightarrow E^3_{i-1,j+2}$ is an isomorphism when $2i-j$ is $2$ (mod $8$).
\end{theorem}

\begin{proof}
By a zig-zag of maps of real-closed fields we can connect $F$ to $\mathbb{R}$.  Thus Suslin rigidity (\cite{Sus}) implies that the homotopy type of the $C_2$-spectrum $\operatorname{K}(F(\sqrt{-1}))_{\widehat{2}}$ is the same for $F$ as for $\mathbb{R}$, so we can assume $F=\mathbb{R}$.  There $\operatorname{K}(\mathbb{R}(\sqrt{-1}))_{\widehat{2}}=ku_{\widehat{p}}$ again by Suslin, so $N= (d_{K(1)}KU)^\#$ with $C_2$ acting on $KU$ by complex conjugation.

Since we already know $\pi_{2n} N = \mathbb{Q}_2/\mathbb{Z}_2(n)$ and $\pi_{2n+1}N=0$ (Proposition \ref{stalks}), the claim about the $E^1_{i,j}$ groups is an easy calculation in group cohomology.  Furthermore, we see that when these groups are nonzero (in the region $j-i>0$) they are cyclic of order two.  It remains to calculate the $d_2$ differentials.

Recall that all the differentials in the homotopy fixed point spectral sequence for the complex conjugation $C_2$-action on $KU$
$$E^1_{i,j}=H^{j-i}(BC_2;\pi_jKU)\Rightarrow \pi_{i}(KO)$$
are well-known, namely the $E^1=E^2$ page as a CDGA is $\mathbb{Z}[t^{\pm 1},e]/2e$ with $e\in E^1_{1,1}$ and $t\in E^1_{4,4}$, with differential $d_2$ determined by $d_2e = e^4\cdot t^{-1}$ and $d_2 t=e^3$.  Furthermore $E^3=E^\infty$.

Now we use that $N=(d_{K(1)}KU)^\#$ is a $C_2$-equivariant module over $KU$, so that the homotopy fixed point spectral sequence for $N$ is a module over the homotopy fixed point spectral sequence for $KU$. Moreover, from Proposition \ref{k(1)stalk} it follows that $\pi_\ast N$ with its $C_2$ action identifies with $(\pi_\ast KU)\otimes\mathbb{Q}_2/\mathbb{Z}_2$, compatibly with the module structure.  Thus on the $E^1$-page the module structure is given by the usual cup product.  One sees that the whole region $j-i>0$ on the $E^1$-page is generated as a module over $\mathbb{Z}[t^{\pm 1},e]/2e$ by the non-trivial class $c$ in $E^1_{2,2}=H^0(C_2;\mathbb{Q}_2/\mathbb{Z}_2(1))=\mu_2$.  On the other hand, one calculates as usual that $\pi_{-1} d_{K(1)}KO=[KO,\Sigma^2 L_{K(1)}S]=0$, so that $\pi_1 N^{C_2}=0$.  This forces that $d_2$ is nontrivial on $c$, which in turn determines all the differentials using the module structure over $\mathbb{Z}[t^{\pm 1},e]/2e$, in particular verifying the claim.
\end{proof}

\begin{remark}
The proof we just gave of Proposition \ref{differentials}, and hence of the second claim in Theorem \ref{eiso}, doesn't quite need the hypothesis that $X$ is $\operatorname{Spec}$ of a field: it only needs that the real spectrum of $X$ is profinite (see \cite{Sc}).
\end{remark}

\section{Locally compact K-theory}

\begin{definition}
Let $\operatorname{LCA}_\aleph$ denote the category of second-countable locally compact Hausdorff abelian groups and continuous homomorphisms.  Certainly $\operatorname{LCA}_\aleph$ is additive, so we can consider it as an exact category with respect to its maximal exact structure, which amounts to saying that a null-composite $A\rightarrow B\rightarrow C$ sequence is short exact provided $A\rightarrow B$ identifies $A$ with a closed subgroup of $B$ and $B\rightarrow C$ identifies $C$ with the quotient by this closed subgroup.
Let $\operatorname{lc}_\mathbb{Z}$ denote the $\mathbb{Z}$-linear stable $\infty$-category given as
$$\operatorname{lc}_\mathbb{Z} = \operatorname{D}^b(\operatorname{LCA}_\aleph),$$
the bounded derived $\infty$-category of the exact category $\operatorname{LCA}_\aleph$ (the natural enhancement of the usual bounded derived category).
\end{definition}

\begin{definition}
Let $\mathcal{P}\in \operatorname{PerfCat}_{\mathbb{Z}}$.  Define
$$\operatorname{lc}_{\mathcal{P}} = \operatorname{Fun}_\mathbb{Z}(\mathcal{P},\operatorname{lc}_\mathbb{Z}),$$
the $\mathbb{Z}$-linear stable $\infty$-category of $\mathbb{Z}$-linear functors from $\mathcal{P}$ to $\operatorname{lc}_{\mathbb{Z}}$.

If $R$ is a ring (or DGA), we set
$$\operatorname{lc}_R = \operatorname{lc}_{\operatorname{Perf}(R)},$$
and if $X$ is a qcqs derived algebraic space we set
$$\operatorname{lc}_X = \operatorname{lc}_{\operatorname{Perf}(X)}.$$
\end{definition}

\begin{remark}
Since $\operatorname{LCA}_\aleph$ is idempotent-complete and essentially small, so is $\operatorname{D}^b(\operatorname{LCA}_\aleph)$ (\cite{BaS} Theorem 2.8, \cite{L1} Lemma 1.2.4.6); thus so is $\operatorname{lc}_\mathcal{P}$ for any $\mathcal{P}\in\operatorname{PerfCat}_{\mathbb{Z}}$, and hence $\mathcal{P}\mapsto \operatorname{lc}_{\mathcal{P}}$ can be viewed as a functor $\operatorname{PerfCat}_{\mathbb{Z}}^{op}\rightarrow \operatorname{PerfCat}_{\mathbb{Z}}$.
\end{remark}

We have two goals in this section.  First, we would like to produce the map $\operatorname{C}_F\rightarrow \pi_1 \operatorname{K}(\operatorname{lc}_F)$ claimed in the introduction, for $F$ a finite, local, or global field.  We won't do the work to fully understand $\operatorname{lc}_F$ in these cases, though it is possible.

Our other goal, crucial for producing the Artin maps, is to fully understand $\operatorname{K}(\operatorname{lc}_\mathbb{Z})$, or more generally any localizing invariant of $\operatorname{lc}_\mathbb{Z}$.  The result, suggested by Lurie, is as follows:

\begin{theorem}\label{koflcz}
Let $A:\operatorname{PerfCat}_\mathbb{Z}\rightarrow \mathcal{A}$ be any functor to a stable $\infty$-category which satisfies localization, i.e.\ sends fiber-cofiber sequences to (fiber-)cofiber sequences.  Then there is a canonical cofiber sequence
$$A(\operatorname{Perf}(\mathbb{Z}))\rightarrow A(\operatorname{Perf}(\mathbb{R}))\rightarrow A(\operatorname{lc}_\mathbb{Z}).$$
\end{theorem}

\begin{remark}
A description of the data defining this cofiber sequence will fall out from the proof.  Namely, the first map is induced by $-\otimes_{\mathbb{Z}}\mathbb{R}$, the second map is induced by the inclusion of finite-dimensional real vector spaces inside $\operatorname{LCA}_\aleph$, and the nullhomotopy of the composite of these two maps is gotten as follows: consider the short exact sequence
$$\mathbb{Z}\rightarrow\mathbb{R}\rightarrow \mathbb{T}$$
in $\operatorname{LCA}_\aleph$.  The first term is discrete, and hence canonically trivialized by $A$ using an Eilenberg swindle with direct sums; similarly the last term is compact and hence canonically trivialized by $A$ using an Eilenberg swindle with products.  Thus the middle term $\mathbb{R}$ is canonically trivialized by $A$ as well.  We can tensor this short exact sequence with an arbitrary element of $\operatorname{Perf}(\mathbb{Z})$ to deduce a trivialization of the composite $\operatorname{Perf}(\mathbb{Z})\rightarrow \operatorname{Perf}(\mathbb{R})\rightarrow \operatorname{lc}_\mathbb{Z}$ after applying $A$, as desired.
\end{remark}The category $\operatorname{lc}_\mathbb{Z}$ (or rather, its variant without the second-countability restriction --- but that is just a technicality) has been studied comprehensively in \cite{HS}.  From that study it is quite easy to ``calculate" $\operatorname{lc}_\mathbb{Z}$:

\begin{theorem}\label{lczmaps}
Fix a countable index set $I$, and for later clarity also denote $I$ by $J$.  As a stable $\infty$-category, $\operatorname{lc}_\mathbb{Z}$ is generated by the objects $\oplus_I\mathbb{Z}$ and $\prod_I\mathbb{T}$ of $\operatorname{LCA}_\aleph$ sitting in degree zero, and the hom-complexes between these generators are as follows:
\begin{enumerate}
\item The maps $\operatorname{hom}(\oplus_I\mathbb{Z},\oplus_J\mathbb{Z})\rightarrow \prod_I \operatorname{hom}(\mathbb{Z},\oplus_J\mathbb{Z})\leftarrow \prod_{I}\oplus_{J}\operatorname{hom}(\mathbb{Z},\mathbb{Z})$ are equivalences, and $\operatorname{hom}(\mathbb{Z},\mathbb{Z})$ is $\mathbb{Z}$ concentrated in degree zero, generated by the identity map.
\item The map $\operatorname{hom}(\oplus_I\mathbb{Z},\prod_J\mathbb{T})\rightarrow \prod_{I\times J}\operatorname{hom}(\mathbb{Z},\mathbb{T})$ is an equivalence, and $\operatorname{hom}(\mathbb{Z},\mathbb{T})$ is the group $\mathbb{R}/\mathbb{Z}$ concentrated in degree zero.
\item The map $\operatorname{hom}(\prod_I\mathbb{T},\oplus_J\mathbb{Z})\leftarrow\oplus_{I\times J}\operatorname{hom}(\mathbb{T},\mathbb{Z})$ is an equivalence, and $\operatorname{hom}(\mathbb{T},\mathbb{Z})$ is $\mathbb{Z}$ concentrated in degree $-1$, with the generator correpsonding to the extension class of $\mathbb{Z}\rightarrow\mathbb{R}\rightarrow\mathbb{T}$.
\item The maps $\operatorname{hom}(\prod_I\mathbb{T},\prod_J\mathbb{T})\rightarrow\prod_J \operatorname{hom}(\prod_I\mathbb{T},\mathbb{T})\leftarrow \prod_J\oplus_I \operatorname{hom}(\mathbb{T},\mathbb{T})$ are equivalences, and $\operatorname{hom}(\mathbb{T},\mathbb{T})$ is $\mathbb{Z}$ concentrated in degree zero, generated by the identity map.
\end{enumerate}
\end{theorem}
\begin{proof}
Certainly $\operatorname{lc}_{\mathbb{Z}}$ is generated by the objects of $\operatorname{LCA}_{\aleph}$ sitting in degree zero.  By the structure theory of locally compact abelian groups, every object of $\operatorname{LCA}_{\aleph}$ 3-step filtration where the associated gradeds are compact, discrete, or $\mathbb{R}^n$.  Thus $\operatorname{lc}_{\mathbb{Z}}$ is generated by the compact groups, the discrete groups, and $\mathbb{R}$.  Thanks to the short exact sequence $\mathbb{Z}\rightarrow\mathbb{R}\rightarrow\mathbb{T}$, we actually only need the compact groups and the discrete groups.  Every countable discrete abelian group has a two-step resolution by free countable discrete abelian groups, so the discrete ones are generated by $\oplus_{I}\mathbb{Z}$.  By Pontryagin duality, the compact ones are also generated by $\prod_{I}\mathbb{T}$.  Thus in total $\operatorname{lc}_\mathbb{Z}$ is generated by $\oplus_I\mathbb{Z}$ and $\prod_I\mathbb{T}$.

Calculations 1 and 2 are easy, because $\mathbb{Z}$ and $\oplus_I \mathbb{Z}$ are projective objects of $\operatorname{LCA}_\aleph$.  Similarly 4 is easy because by duality $\mathbb{T}$ and $\prod_I\mathbb{T}$ are injective.  For 3, one can calculate $\operatorname{hom}(\mathbb{T},\mathbb{Z})$ by the projective resolution $\mathbb{Z}\rightarrow\mathbb{R}$ of $\mathbb{T}$, getting the second claim in 3.  What remains is the first claim in 3; that is the trickiest, since $\prod_I\mathbb{T}$ has no projective resolution and $\oplus_J\mathbb{Z}$ has no injective resolution.  We refer to \cite{HS}  Example 4.10 for the proof.
\end{proof}

To prove Theorem \ref{koflcz}, we will use this calculation to recognize $\operatorname{lc}_\mathbb{Z}$ as a Verdier quotient $\mathcal{D}'/\mathcal{C}$, where $\mathcal{C}$ is equivalent to $\operatorname{Perf}(\mathbb{Z})$ and $\mathcal{D}'$ is equivalent to $\operatorname{Perf}(\mathbb{R}$) ``up to Eilenberg swindles".  In more detail, we will define a ``cone" construction, which takes as input a functor
$$F:\mathcal{C}\rightarrow\mathcal{D}$$
of stable $\infty$-categories and outputs a new stable $\infty$-category $\operatorname{cone}(F)$.  To build $\operatorname{cone}(F)$, we enlarge $\mathcal{D}$ along Eilenberg swindles in such a way that $F$ becomes fully faithful, and then we take the usual Verdier quotient.  Theorem \ref{koflcz} follows from an identification of $\operatorname{lc}_\mathbb{Z}$ with the cone of the base-change functor $\operatorname{Perf}(\mathbb{Z})\rightarrow \operatorname{Perf}(\mathbb{R})$.  We will also interpret the desired map $\operatorname{C}_F\rightarrow \pi_1 \operatorname{lc}_F$ (for $F$ a finite, local, or global field) in terms of the cone construction.

\subsection{The cone construction}

We first define minimalist versions of the standard Ind and Pro categories.

\begin{definition}
Let $\mathcal{C}$ be a stable $\infty$-category.  We define $\operatorname{ind}(\mathcal{C})$ to be the smallest stable full subcategory of $\operatorname{Ind}(\mathcal{C})$ which contains the constant countable coproduct $\oplus_\mathbb{N}X$ for all $X\in\mathcal{C}$, and we define $\operatorname{pro}(\mathcal{C})$ to be the dual construction, $\operatorname{pro}(\mathcal{C})=\operatorname{ind}(\mathcal{C}^{op})^{op}\subset \operatorname{Ind}(\mathcal{C}^{op})^{op}=\operatorname{Pro}(\mathcal{C})$.
\end{definition}

\begin{remark} Note that $\mathcal{C}\subset \operatorname{ind}(\mathcal{C})$, since any $X\in\mathcal{C}$ is the cofiber of the ``shift by one'' map $\oplus_\mathbb{N}X\rightarrow\oplus_\mathbb{N}X$.
\end{remark}

\begin{remark} The Milnor telescope construction shows that $\operatorname{ind}(\mathcal{C})$ is closed under colimits along diagrams of the type
$$X\overset{f_1}{\rightarrow} X\overset{f_2}{\rightarrow} X\overset{f_3}{\rightarrow}\ldots.$$
In particular $\operatorname{ind}(\mathcal{C})$ is idempotent-complete.
\end{remark}

\begin{remark}
If $\mathcal{C}$ is essentially small, then so is $\operatorname{ind}(\mathcal{C})$.
\end{remark}

\begin{remark}
As in \cite{L2} Proposition 5.3.5.10, one sees that $\operatorname{ind}(\mathcal{C})$ satisfies the following universal property: it is closed under constant countable coproducts, and if $\mathcal{E}$ is any stable $\infty$-category which is closed under constant countable coproducts, then giving an exact functor $\operatorname{ind}(\mathcal{C})\rightarrow\mathcal{E}$ which commutes with constant countable coproducts is equivalent (via restriction) to giving an exact functor $\mathcal{C}\rightarrow\mathcal{E}$.
\end{remark}

\begin{remark}
If $\mathcal{C}$ has a $R$-linear structure for some $E_2$-ring $R$, i.e.\ it is a module over $\operatorname{Perf}(R)$, then $\operatorname{ind}(\mathcal{C})$ acquires a canonical such structure as well, and in fact a unique one which preserves countable constant coproducts.  This can be seen by applying the universal property of the previous remark.  From this one deduces that all the material of this section is valid without change in the $R$-linear context as well.  We will apply this with $R=\mathbb{Z}$.
\end{remark}

Now, let $F:\mathcal{C}\rightarrow\mathcal{D}$ be an exact functor between stable $\infty$-categories.  Our category $\operatorname{cone}(F)$ will be defined as a Verdier quotient of another category $\operatorname{precone}(F)$, which in turn is an ``extension'' of $\operatorname{ind}(\mathcal{C})$ by $\mathcal{D}$ by $\operatorname{pro}(\mathcal{C})$.  Before giving the definition of $\operatorname{precone}(F)$, we note that there is an obvious notion of a map from an object of $\operatorname{Ind}(\mathcal{C})$ to an object of $\mathcal{C}$, from an object of $\mathcal{C}$ to an object of $\operatorname{Pro}(\mathcal{C})$, and from an object of $\operatorname{Ind}(\mathcal{C})$ to an object of $\operatorname{Pro}(\mathcal{C})$; moreover, maps of the first two types can be composed to obtain a map of the third type.  For example we can consider all of these as full subcategories of $\operatorname{Ind}(\operatorname{Pro}(\mathcal{C}))$ or of $\operatorname{Pro}(\operatorname{Ind}(\mathcal{C}))$; the mapping spaces we are interested in will be the same in either case.

\begin{definition}\label{precone}
Let $F:\mathcal{C}\rightarrow\mathcal{D}$ be an exact functor between stable $\infty$-categories.  We define a stable $\infty$-category $\operatorname{precone}(F)$ to consist of tuples $(I,V,P,\alpha,\beta)$, where:
\begin{enumerate}
\item $I\in \operatorname{ind}(\mathcal{C})$, $V\in\mathcal{D}$, and $P\in \operatorname{pro}(\mathcal{C})$;
\item $\alpha$ is a map $I\rightarrow P$;
\item $\beta$ is a factorization of $F(\alpha)$ through $V$.
\end{enumerate}
\end{definition}

Calculations in $\operatorname{precone}(F)$ are easy right from the definition.  We express this as follows.

\begin{definition}
Let $F:\mathcal{C}\rightarrow\mathcal{D}$ be an exact functor between stable $\infty$-categories.  Define functors
$$\delta:\operatorname{ind}(\mathcal{C})\rightarrow \operatorname{precone}(F), \hspace{10pt} \nu:\mathcal{D}\rightarrow \operatorname{precone}(F),\hspace{10pt}\kappa:\operatorname{pro}(\mathcal{C})\rightarrow \operatorname{precone}(F)$$
by $\delta(I)=(\Sigma^{-1}I,0,0,0,0)$, $\nu(V)=(0,V,0,0,0)$, and $\kappa(P)=(0,0,\Sigma P,0,0)$.
\end{definition}

Then one calculates:

\begin{proposition}\label{semiorth}
\begin{enumerate}
\item The functors $\delta$, $\nu$, and $\kappa$ are fully faithful.
\item Their essential images semi-orthogonally decompose $\operatorname{precone}(F)$ as $\langle \operatorname{ind}(\mathcal{C}),\mathcal{D},\operatorname{pro}(\mathcal{C})\rangle$.  (Thus $precone(F)$ is generated by these three full subcategories, and maps $X\rightarrow Y$ vanish if $X\in \operatorname{pro}(\mathcal{C})$ and $Y\in \operatorname{ind}(\mathcal{C})\cup\mathcal{D}$ or $X\in\mathcal{D}$ and $Y\in\operatorname{ind}(\mathcal{C})$.)
\item The remaining mapping spectra in $\operatorname{precone}(F)$ are as follows.  For $I\in \operatorname{ind}(\mathcal{C}), V\in\mathcal{D}$, and $P\in \operatorname{pro}(\mathcal{C})$, we have:
\begin{itemize}
\item $\operatorname{map}(\delta(I),\nu(V))=\operatorname{map}(F(I),V)$;
\item $\operatorname{map}(\nu(V),\kappa(P))=\operatorname{map}(V,F(P))$;
\item $\operatorname{map}(\delta(I),\kappa(P)) = \operatorname{cofib}(\operatorname{map}(I,P)\rightarrow \operatorname{map}(F(I),F(P)))$.
\end{itemize}
Moreover, the composition maps are the obvious ones.
\end{enumerate}
\end{proposition}

\begin{corollary} There is a canonical null-composite sequence
$$\delta(X)\rightarrow \nu(F(X))\rightarrow \kappa(X),$$
functorial in $X\in\mathcal{C}$.
\end{corollary}

This null-composite sequence is the subject of the following universal property of $\operatorname{precone}(F)$:

\begin{proposition}\label{preconeuniv}
Let $F:\mathcal{C}\rightarrow\mathcal{D}$ be an exact functor between stable $\infty$-categories, let $\operatorname{precone}(F)$ be as in Definition \ref{precone}, and suppose given another stable $\infty$-category $\mathcal{E}$.

Let $\operatorname{Fun}^{precone}(F;\mathcal{E})$ denote the $\infty$-category whose objects are tuples $(d,v,k,c)$ where $d,k:\mathcal{C}\rightarrow\mathcal{E}$ and $v:\mathcal{D}\rightarrow\mathcal{E}$ are exact functors and $c$ is a null-composite sequence
$$d\rightarrow v\circ F\rightarrow k,$$
such that the constant countable coproduct $\oplus_\mathbb{N} d(X)$ and the constant countable product $\prod_\mathbb{N} k(X)$ exist in $\mathcal{E}$ for each $X\in\mathcal{C}$, and let $\operatorname{Fun}^{precone}(\operatorname{precone}(F),\mathcal{E})$ denote the full subcategory of $\operatorname{Fun}(\operatorname{precone}(F),\mathcal{E})$ consisting of those functors which are exact and preserve the countable constant coproducts of $\delta(X)$'s and constant countable products of $\kappa(X)$'s.

Then the functor $\operatorname{Fun}^{precone}(\operatorname{precone}(F),\mathcal{E})\rightarrow \operatorname{Fun}^{precone}(F;\mathcal{E})$, given by composition with $\delta$, $\nu$ and $\kappa$, is an equivalence.
\end{proposition}

\begin{proof}
Let us simply describe how to produce produce the functor in the other direction, which is all we will use in the end; checking that it is indeed an inverse is anyway straightforward.  Thus suppose given the data $(d,v,k,c)\in \operatorname{Fun}^{precone}(F;\mathcal{E})$, and let us produce the required functor $f:\operatorname{precone}(F)\rightarrow\mathcal{E}$.

Take $(I,V,P,\alpha,\beta)\in \operatorname{precone}(F)$.  The objects $I,V,$ and $P$ give us the objects $d(I)$, $v(V)$, and $k(P)$ in $\mathcal{E}$.  Then $\alpha$ and $\beta$ and the null-composite sequence $d\rightarrow v\circ F\rightarrow k$ combine to give a null-composite sequence $d(I)\rightarrow v(V)\rightarrow k(P)$ in $\mathcal{E}$.  We define $f(I,V,P,\alpha,\beta)$ to be the exactness defect of this null-composite sequence.  (The \emph{exactness defect} is equivalently the cofiber of $d(I)$ mapping to the fiber of $v(V)\rightarrow k(P)$, or the fiber of the cofiber of $d(I)\rightarrow v(V)$ mapping to $k(P)$.)
\end{proof}

Now, the way we will pass from $\operatorname{precone}(F)$ to $\operatorname{cone}(F)$ is to enforce that this null-composite sequence $\delta\rightarrow \nu\circ F\rightarrow \kappa$ should be a fiber-cofiber sequence.  The exactness defect of that sequence identifies with the functor $\iota:\mathcal{C}\rightarrow \operatorname{precone}(F)$ defined by
$$\iota(X)=(X,F(X),X,\operatorname{id},\operatorname{id}).$$
It's easy to check that this functor is fully faithful, and so $\iota(\mathcal{C})$ is a full stable subcategory of $\operatorname{precone}(F)$.  Thus we can form the Verdier quotient $\operatorname{precone}(F)/\iota(\mathcal{C})$.
\begin{definition}\label{cone}
Let $F:\mathcal{C}\rightarrow\mathcal{D}$ be an exact functor between stable $\infty$-categories.  We define $\operatorname{cone}(F)$ to be the idempotent-completed Verdier quotient $\operatorname{precone}(F)/\iota(\mathcal{C})$, where $\operatorname{precone}(F)$ is as in Definition \ref{precone} and $\iota$ is as above.
\end{definition}

Calculations in $\operatorname{cone}(F)$ are hardly more difficult than in $\operatorname{precone}(F)$.  Indeed, one can explicitly write down an $\iota(\mathcal{C})$-Ind-injective resolution for any object $X=(I,V,P,\alpha,\beta)$ of $\operatorname{precone}(F)$.  Namely, there is an obvious map
$$\iota(I)\rightarrow X$$
in $\operatorname{Ind}(\operatorname{precone}(F))$, and its cofiber is the $\iota(\mathcal{C})$-injective resolution of $X$.  Dually, we can also describe the $\iota(\mathcal{C})$-projective resolution of $X$ as the fiber of $X\rightarrow \iota(P)$ in $\operatorname{Pro}(\operatorname{precone}(F))$.

One upshot of this is the following:

\begin{lemma}\label{conecalc} If $X$ and $Y$ are objects of $\operatorname{precone}(F)$ which each lie in one of the generating full subcategories $\operatorname{ind}(\mathcal{C}),\mathcal{D}$, or $\operatorname{pro}(\mathcal{C})$, then $\operatorname{Map}(X,Y)$ is the same in $\operatorname{cone}(F)$ as in $\operatorname{precone}(F)$ (Proposition \ref{semiorth}), with one exception: if $X\in \operatorname{pro}(\mathcal{C})$ and $Y\in \operatorname{ind}(\mathcal{C})$, then
$$\operatorname{Map}_{cone}(\kappa(X),\delta(Y)) = \operatorname{Map}(X,\Sigma Y),$$
whereas it is $0$ in $\operatorname{precone}(F)$.
\end{lemma}
Here $\operatorname{Map}(X,\Sigma Y)$ carries the only formal meaning it can have; it is the same if calculated in either $\operatorname{Ind}(\operatorname{Pro}(\mathcal{C}))$ or in $\operatorname{Pro}(\operatorname{Ind}(\mathcal{C}))$.

There is also a universal property for $\operatorname{cone}(F)$ analogous to that of $\operatorname{precone}(F)$; we just have to replace ``null-composite sequence'' with ``fiber-cofiber sequence'':

\begin{proposition}\label{coneuniv}
Let $F:\mathcal{C}\rightarrow\mathcal{D}$ be a functor between stable $\infty$-categories, let $\operatorname{cone}(F)$ be as in Definition \ref{cone}, and suppose given an idempotent-complete stable $\infty$-category $\mathcal{E}$.

Let $\operatorname{Fun}^{cone}(F;\mathcal{E})$ denote the $\infty$-category whose objects are tuples $(d,v,k,c)$ where $d,k:\mathcal{C}\rightarrow\mathcal{E}$ and $v:\mathcal{D}\rightarrow\mathcal{E}$ are exact functors and $c$ is a fiber-cofiber sequence
$$d\rightarrow v\circ F\rightarrow k,$$
such that the constant countable coproduct $\oplus_\mathbb{N} d(X)$ and the constant countable product $\prod_\mathbb{N} k(X)$ exist in $\mathcal{E}$ for each $X\in\mathcal{C}$, and let $\operatorname{Fun}^{cone}(\operatorname{cone}(F),\mathcal{E})$ denote the full subcategory of $\operatorname{Fun}(\operatorname{cone}(F),\mathcal{E})$ consisting of those functors which are exact and preserve the countable constant coproducts of $\delta(X)$'s and constant countable products of $\kappa(X)$'s.

Then the functor $\operatorname{Fun}^{cone}(\operatorname{cone}(F),\mathcal{E})\rightarrow \operatorname{Fun}^{cone}(F;\mathcal{E})$, given by composition with $\delta$, $\nu$ and $\kappa$, is an equivalence.\end{proposition}

\begin{proof}
Note that a functor $\operatorname{cone}(F)\rightarrow\mathcal{E}$ preserves the relevant coproducts and products if and only if its restriction to $\operatorname{precone}(F)$ does; this follows from the fact that $\operatorname{precone}(F)\rightarrow \operatorname{cone}(F)$ preserves those coproducts and products, which is a consequence of the above discussion of mapping spaces in $\operatorname{cone}(F)$.  Thus the statement to be proved follows by combining the universal property of $\operatorname{precone}$ (Proposition \ref{preconeuniv}) with the universal property of the Verdier quotient.
\end{proof}

Also, the value of $\operatorname{cone}(F)$ on localizing invariants is easy to describe.  Namely, let $\operatorname{PerfCat}$ denote the $\infty$-category of idempotent-complete small stable $\infty$-categories and exact functors between them.  By a \emph{localizing invariant} of $\operatorname{PerfCat}$ we will mean a functor $A:\operatorname{PerfCat}\rightarrow \mathcal{A}$ such that $\mathcal{A}$ is a stable $\infty$-category and $A$ preserves fiber-cofiber sequences.  (The terminology is as in \cite{BGT}, but we don't require preservation of filtered colimits.)

\begin{proposition}\label{kcone}
Let $F:\mathcal{C}\rightarrow\mathcal{D}$ be an exact functor between idempotent-complete small stable $\infty$-categories.  Then for any localizing invariant $A: \operatorname{PerfCat}\rightarrow\mathcal{A}$, there is a canonical cofiber sequence
$$A(\mathcal{C})\overset{F}{\longrightarrow} \operatorname{A}(\mathcal{D})\rightarrow A(\operatorname{cone}(F)).$$
\end{proposition}

\begin{proof}
Since $\operatorname{A}$ is localizing, by the definition of $\operatorname{cone}(F)$ we have that $\operatorname{A}(\operatorname{cone}(F))$ identifies with the cofiber of $A(\mathcal{C})\overset{\iota}{\longrightarrow}A(\operatorname{precone}(F))$.  On the other hand, the functor $F:\mathcal{C}\rightarrow\mathcal{D}$ factors as the composition of $\iota$ with the projection $\operatorname{precone}(F)\rightarrow\mathcal{D}$.  Thus it suffices to show that this projection induces an equivalence $A(\operatorname{precone}(F))\overset{\sim}{\rightarrow} A(\mathcal{D})$.  By the semi-orthogonal decomposition of Proposition \ref{semiorth} and the fact that localizing invariants turn semi-orthogonal decompositions into direct sums, we need only see that $A(\operatorname{ind}(\mathcal{C}))$ and $A(\operatorname{pro}(\mathcal{C}))$ are both trivial.  But $\operatorname{ind}(\mathcal{C})$ admits countable direct sums and is thus trivial on any localizing invariant by an Eilenberg swindle, and dually for $\operatorname{pro}(\mathcal{C})$.
\end{proof}

We finish our discussion of the cone construction with some remarks and examples.

\begin{remark} If we were to use $\operatorname{Ind}(\mathcal{C})$ and $\operatorname{Pro}(\mathcal{C})$ instead of $\operatorname{ind}(\mathcal{C})$ and $\operatorname{pro}(\mathcal{C})$ in the above constructions, then we would obtain a bigger variant of the precone and cone categories, which could be denoted $\operatorname{PreCone}(F)$ and $\operatorname{Cone}(F)$.  It seems that in practice it doesn't much matter whether one uses $\operatorname{cone}(F)$ or $\operatorname{Cone}(F)$ or something in between.  We chose to focus on $\operatorname{cone}(F)$ because it is essentially small when $\mathcal{C}$ is, and because of its connection with $\operatorname{D}^b(\operatorname{LCA}_\aleph)$.  But for the remainder of these remarks we'll talk about $\operatorname{Cone}(F)$ instead, because its formal properties are slightly more general and convenient.
\end{remark}

\begin{remark}
There are two extreme examples of $\operatorname{Cone}(F)$ which are worth looking at.  The first is when $F$ is the zero functor $\mathcal{C}\rightarrow 0$, where $\mathcal{C}$ is arbitrary.  The association $\mathcal{C}\mapsto \operatorname{Cone}(\mathcal{C}\rightarrow 0)$ is a derived analog of the $\underset{\longleftrightarrow}{\operatorname{lim}}$, or ``locally compact objects", construction of \cite{Be} A.3.  Indeed, an object of $\operatorname{PreCone}(\mathcal{C}\rightarrow 0)$ can, by reindexing, be identified with a map $I\rightarrow \Sigma P$ where $I\in \operatorname{Ind}(\mathcal{C})$ and $P\in \operatorname{Pro}(\mathcal{C})$.  This map can be thought of as classifying a formal extension of $I$ by $P$; the middle term of this imagined extension is thus ``locally compact'' if we think of Ind as signifying discrete and Pro as signifying compact.   Passing from $\operatorname{PreCone}$ to $\operatorname{Cone}$ via the Verdier quotient has the effect of remembering only the middle term of the extension, forgetting how we ``cut it in half" to get the extension itself.  Note also that in this case Proposition \ref{kcone} gives that $A(\operatorname{Cone}(\mathcal{C}\rightarrow 0))\simeq \operatorname{cofib}(A(\mathcal{C})\rightarrow 0)\simeq \Sigma A(\mathcal{C})$ for any localizing invariant $A$; this can be compared the the result of \cite{Sa} in the exact category context.
\end{remark}

\begin{remark} The other extreme is when the functor $F:\mathcal{C}\rightarrow\mathcal{D}$ is fully faithful.  Then there is also the Verdier quotient $\mathcal{D}/\mathcal{C}$, and one can produce a canonical functor $\mathcal{D}/\mathcal{C}\rightarrow \operatorname{Cone}(\mathcal{C}\rightarrow\mathcal{D})$ as follows.  The functor $F$ automatically has an Ind-right adjoint $r:\mathcal{D}\rightarrow \operatorname{Ind}(\mathcal{C})$ and a Pro-left adjoint $l:\mathcal{D}\rightarrow \operatorname{Pro}(\mathcal{C})$.  Then we can define a functor $\mathcal{D}\rightarrow \operatorname{PreCone}(\mathcal{C}\rightarrow\mathcal{D})$ by $d\mapsto (r(d),d,l(d),-,-)$ where the last two bits of data are gotten from the appropriate units and counits for these adjunctions.  This functor factors through the quotient by $\mathcal{C}$ to define the desired functor $\mathcal{D}/\mathcal{C}\rightarrow \operatorname{Cone}(\mathcal{C}\rightarrow\mathcal{D})$.  Proposition \ref{kcone} in this case says that this functor is an equivalence on localizing invariants.  So we can think of $\operatorname{Cone}(\mathcal{C}\rightarrow\mathcal{D})$ as some ``freed up'' version of $\mathcal{D}/\mathcal{C}$, which exists even when $\mathcal{C}\rightarrow\mathcal{D}$ is not fully faithful.
\end{remark}

\begin{remark}\label{boundary}
We can also connect the two examples.  Namely, if $\mathcal{C}\rightarrow\mathcal{D}$ is again fully faithful, we obtain a functor $\partial:\mathcal{D}/\mathcal{C}\rightarrow \operatorname{Cone}(\mathcal{C}\rightarrow 0)$ by composing the functor of the previous remark with the projection $\mathcal{D}\rightarrow 0$.  The reason for the notation is that, under the equivalence $A(\operatorname{Cone}(\mathcal{C}\rightarrow 0))\simeq \Sigma A(\mathcal{C})$, the map induced by $\partial$ on any localizing invariant $A$ identifies with the boundary map in the localization sequence $A(\mathcal{C})\rightarrow A(\mathcal{D})\rightarrow A(\mathcal{D}/\mathcal{C})$.  This is clear by functoriality.\footnote{Actually, there are sign issues, but we can always make these work out by adjusting the identification $\operatorname{cofib}(X\rightarrow 0)\simeq \Sigma X$ used in identifying $A(\operatorname{Cone}(\mathcal{C}\rightarrow 0))\simeq \Sigma A(\mathcal{C})$, which is only canonical up to a sign.  As for our sign convention for boundary maps, let us declare it to be such that the boundary map in the localization sequence for a DVR sends a uniformizer in $\pi_1\operatorname{K}$ of the fraction field to the unit $1\in \pi_0\operatorname{K}$ of the residue field, cf \cite {W} Example 6.1.12.}
\end{remark}

Basically, this cone business serves to realize certain operations on K-theoretic spectra at the more primitive level of categories and functors.

\subsection{The relevant examples}

We start with the base case of $\operatorname{lc}_{\mathbb{Z}}=\operatorname{D}^b(\operatorname{LCA}_\aleph)$.

\begin{theorem}\label{conelcz}
There is a $\mathbb{Z}$-linear equivalence
$$\alpha_\mathbb{Z}:\operatorname{cone}(\operatorname{Perf}(\mathbb{Z})\rightarrow\operatorname{Perf}(\mathbb{R}))\overset{\sim}{\rightarrow} \operatorname{lc}_{\mathbb{Z}},$$
produced from the universal property of the cone (Proposition \ref{coneuniv}) via the data of the cofiber sequence
$$\mathbb{Z}\rightarrow\mathbb{R}\rightarrow\mathbb{T}$$
inside $\operatorname{lc}_{\mathbb{Z}}$ and the action of the ring $\mathbb{R}$ on the middle term.
\end{theorem}

\begin{proof}
More specifically, the object $\mathbb{R}\in \operatorname{lc}_{\mathbb{Z}}$ carries an obvious $\mathbb{R}$-action by multiplication, so it generates a functor $v:\operatorname{Perf}(\mathbb{R})\rightarrow \operatorname{lc}_{\mathbb{Z}}$; similarly, the cofiber sequence $\mathbb{Z}\rightarrow\mathbb{R}\rightarrow\mathbb{T}$ carries a $\mathbb{Z}$-action and so gives the required cofiber sequence of functors needed to apply the universal property of the cone (Proposition \ref{coneuniv}).  Note that the existence of the required coproducts and products follows from the fact that $\mathbb{Z}$ is discrete and $\mathbb{T}$ is compact, so the former admits a constant countable coproduct and the latter a constant countable product.  (Note also that the inclusion $\operatorname{LCA}_\aleph\subset\operatorname{D}^b(\operatorname{LCA}_\aleph)$ preserves these coproducts and products because $\mathbb{Z}$ is projective and $\mathbb{T}$ is injective; the same goes for any discrete or compact group by the usual two-step resolution by free abelian groups and its Pontryagin dual.)

To prove that the resulting functor $\alpha_{\mathbb{Z}}:\operatorname{cone}(\operatorname{Perf}(\mathbb{Z})\rightarrow\operatorname{Perf}(\mathbb{R}))\rightarrow \operatorname{lc}_{\mathbb{Z}}$ is an equivalence, it suffices to show that there is a full subcategory of the source whose image under $\alpha_{\mathbb{Z}}$ generates the target, and on which $\alpha_{\mathbb{Z}}$ is fully faithful.  For this we can take the full subcategory on $\oplus_{\mathbb{N}}\delta(\mathbb{Z})$ and $\prod_{\mathbb{N}}\kappa(\mathbb{T})$.  That the image under $\alpha_{\mathbb{Z}}$ generates follows from Theorem \ref{lczmaps}, and the full faithfulness of $\alpha_{\mathbb{Z}}$ on this full subcategory follows by comparing the calculations of mapping spaces in cones (Lemma \ref{conecalc}) with those in $\operatorname{lc}_{\mathbb{Z}}$ (Theorem \ref{lczmaps}).
\end{proof}

\begin{corollary}
Theorem \ref{koflcz} holds: for any functor $A:\operatorname{PerfCat}_{\mathbb{Z}}\rightarrow \mathcal{A}$ to a stable $\infty$-category which preserves fiber-cofiber sequences, there is a canonical cofiber sequence
$$A(\operatorname{Perf}(\mathbb{Z}))\rightarrow A(\operatorname{Perf}(\mathbb{R}))\rightarrow A(\operatorname{lc}_{\mathbb{Z}}).$$
\end{corollary}
\begin{proof}
This follows by combining the previous theorem with (the $\mathbb{Z}$-linear analog of) Proposition \ref{kcone}.
\end{proof}

Now, for recovering the Artin reciprocity law we are interested in investigating $\operatorname{lc}_F$ for $F$ a finite, local, or global field.  As an intermediary, to help with functoriality, we will also want to look at $\operatorname{lc}_R$ for $R$ the ring of integers in a non-archimedean local field.

\begin{proposition}\label{compare}
\begin{enumerate}
\item Let $F$ be a global field, and $\mathbb{A}_F$ its ring of adeles.  There is a canonical $F$-linear comparison functor
$$\alpha_F:\operatorname{cone}(\operatorname{Perf}(F)\rightarrow\operatorname{Perf}(\mathbb{A}_F))\rightarrow \operatorname{lc}_F$$
determined by the cofiber sequence
$$F\rightarrow \mathbb{A}_F\rightarrow \mathbb{A}_F/F$$
of objects of $\operatorname{lc}_\mathbb{Z}$, with its natural $F$-action and the extension of this action to an $\mathbb{A}_F$-action on the middle term.
\item Let $F$ be a finite field.  There is a canonical $F$-linear comparison functor
$$\alpha_F:\operatorname{cone}(\operatorname{Perf}(F)\rightarrow 0)\rightarrow \operatorname{lc}_F$$
determined by the cofiber sequence
$$F\rightarrow 0 \rightarrow \Sigma F$$
in $\operatorname{lc}_\mathbb{Z}$ with its $F$-action.
\item Let $R$ be the ring of integers of a non-archimedean local field $F$, and let $\operatorname{Perf}_{\mathfrak{m}}(R)\subset\operatorname{Perf}(R)$ denote the fiber of $\operatorname{Perf}(R)\rightarrow\operatorname{Perf}(F)$.  There is a canonical $R$-linear comparison functor
$$\alpha_R:\operatorname{cone}(\operatorname{Perf}_{\mathfrak{m}}(R)\rightarrow 0)\rightarrow \operatorname{lc}_R,$$
determined by the functorial cofiber sequence of $R$-modules in $\operatorname{lc}_\mathbb{Z}$
$$M\rightarrow 0\rightarrow \Sigma M$$
for $M\in \operatorname{Perf}_{\mathfrak{m}}(R)$.
\item Let $F$ be a local field.  There is a canonical $F$-linear comparison functor
$$\alpha_F:\operatorname{Perf}(F)\rightarrow \operatorname{lc}_F$$
determined by the object $F$ of $\operatorname{lc}_\mathbb{Z}$ with its $F$-action.
\end{enumerate}
\end{proposition}
\begin{proof}
For 1, 2, and 3, again this follows from the universal property of cone (Proposition \ref{coneuniv}) once we note that the left-hand term is discrete and the right-hand term is compact in each of our cofiber sequences.  For 4, this follows from the usual universal property of Perf.
\end{proof}
Since K-theory is a localizing invariant, combining with Proposition \ref{kcone} gives:
\begin{corollary}\label{kcompare}
\begin{enumerate}
\item If $F$ is a global field, $\operatorname{K}(\alpha_F)$ gives a map of spectra
$$\operatorname{cofib}(\operatorname{K}(F)\rightarrow \operatorname{K}(\mathbb{A}_F))\rightarrow \operatorname{K}(\operatorname{lc}_F),$$
and therefore on $\pi_1$ we get
$$\mathbb{A}_F^\times/F^\times\rightarrow \pi_1\operatorname{K}(\operatorname{lc}_F).$$
\item If $F$ is a finite field, $\operatorname{K}(\alpha_F)$ gives a map of spectra
$$\Sigma \operatorname{K}(F)\rightarrow \operatorname{K}(\operatorname{lc}_F),$$
and therefore on $\pi_1$ we get
$$\mathbb{Z}\rightarrow \pi_1\operatorname{K}(\operatorname{lc}_F).$$
\item If $R$ is the ring of integers of a non-archimedean local field, then $\operatorname{K}(\alpha_R)$ gives a map of spectra
$$\Sigma \operatorname{K}(\operatorname{Perf}_{\mathfrak{m}}(R))\rightarrow \operatorname{K}(\operatorname{lc}_R),$$
and therefore on $\pi_1$ we get
$$\mathbb{Z}\rightarrow \pi_1\operatorname{K}(\operatorname{lc}_R).$$
\item If $F$ is a local field, $\operatorname{K}(\alpha_F)$ gives a map of spectra
$$\operatorname{K}(F)\rightarrow \operatorname{K}(\operatorname{lc}_F),$$
and therefore on $\pi_1$ we get
$$F^\times\rightarrow \pi_1\operatorname{K}(\operatorname{lc}_F).$$
\end{enumerate}
\end{corollary}

\begin{remark}
The map $\alpha$ (and therefore also $\operatorname{K}(\alpha)$) is an equivalence in cases 1 and 2, but not in cases 3 and 4.  But this failure in the latter cases is essentially a technicality which arises because our definition of $\operatorname{lc}_R$ did not require the topology on $R$ to interact with the topology of the objects in $\operatorname{LCA}_\aleph$.\end{remark}

\begin{remark}
We can put the finite and global cases on a common footing by defining the adele ring of a finite field to be the zero ring $0$.  In fact, there is a common story for any commutative ring $R$ essentially of finite type over $\mathbb{Z}$.  Namely, attached to any such $R$ is an ``adele ring" $\mathbb{A}_R$, which is not an ordinary ring but a co-connective $E_\infty$-ring, and also an object of $\operatorname{lc}_\mathbb{Z}$.  There is a cofiber sequence of $R$-modules in $\operatorname{lc}_\mathbb{Z}$
$$R\rightarrow\mathbb{A}_R\rightarrow (\omega_{R/\mathbb{Z}})^\vee$$
where $\omega_{R/\mathbb{Z}}$ is the relative dualizing complex and $(-)^\vee$ is derived Pontryagin duality $\operatorname{hom}(-;\mathbb{T})$.  Using this one gets a description of $\operatorname{lc}_R$ for any such $R$ in exactly the same terms as for a global field above.

New language is required to set all of this up, but let us just indicate how things look in the case $R=\mathbb{Z}[X]$.  There $\mathbb{A}_R$ can be represented by the homotopy pullback of 
$$\mathbb{Z}((X^{-1}))\rightarrow \mathbb{R}((X^{-1}))\leftarrow \mathbb{R}[X].$$
Note that the objects here do not themselves lie in $\operatorname{LCA}_{\aleph}$, only the homotopy groups of the pullback (namely $\mathbb{Z}[X]$ and $\mathbb{T}[[X^{-1}]]$) do.  This shows that our derived definitions are necessary to obtain the correct theory in higher dimensions: more concretely, the adele ring of a higher-dimensional ring like $\mathbb{Z}[X]$ does not give a class in the K-theory of the exact category of $R$-modules in $\operatorname{LCA}_\aleph$, though it does give a class in $\operatorname{K}(\operatorname{lc}_R)$ after things are set up properly.

Note also that the ring $\mathbb{Z}((X^{-1}))$ is in some sense the ``competion at $\infty$" of $\operatorname{Spec}(\mathbb{Z}[X])$ over $\operatorname{Spec}(\mathbb{Z})$, and the adele ring above is gotten by combining this ``geometric completion at $\infty$" with the "arithmetic completion at $\infty$" $\mathbb{R}$ of $\mathbb{Z}$.  This is an instance of a general pattern with such adele rings.
\end{remark}

We will also need some functorial properties of these comparison functors $\alpha$.  Namely, for any homomorphism of rings $R\rightarrow R'$ there is a forgetful functor $\operatorname{lc}_{R'}\rightarrow \operatorname{lc}_R$, and we would like to see this reflected in the left-hand side of the comparison maps $\alpha$ of Proposition \ref{compare}.  We simply state the results; the proofs are immediate from the definitions, especially if one uses the universal property of cones (Proposition \ref{coneuniv}).

\begin{proposition}\label{comparefunct}
\begin{enumerate}
\item Let $F\rightarrow L$ be a finite extension of local fields or finite fields.  Then the functors $\alpha$ as in Proposition \ref{compare} intertwine the natural forgetful functor $\operatorname{lc}_{L}\rightarrow \operatorname{lc}_F$ with the forgetful functor $\operatorname{Perf}(L)\rightarrow\operatorname{Perf}(F)$.  (There is also a similar statement in the global field case, incorporating the forgetful map $\operatorname{Perf}(\mathbb{A}_{L})\rightarrow\operatorname{Perf}(\mathbb{A}_F)$.)

In particular, on $\pi_1\operatorname{K}$ in the local field case, we get a commutative diagram
$$\xymatrix{ 
L^\times\ar[r]\ar[d] & \pi_1\operatorname{K}(\operatorname{lc}_{L})\ar[d] \\
F^\times\ar[r] & \pi_1\operatorname{K}(\operatorname{lc}_{F})}$$
where the left-hand map is the norm map, whereas in the finite field case we get a commutative diagram
$$\xymatrix{ 
\mathbb{Z}\ar[r]\ar[d] & \pi_1\operatorname{K}(\operatorname{lc}_{L})\ar[d] \\
\mathbb{Z}\ar[r] & \pi_1\operatorname{K}(\operatorname{lc}_{F})}$$
where the left-hand map is multiplication by the degree $[L:F]$.
\item Let $F\rightarrow L$ be a homomorphism from a global field to one of its non-discrete completions $L$ (which is then a local field).  The functors $\alpha$ intertwine the natural forgetful functor $\operatorname{lc}_{L}\rightarrow \operatorname{lc}_F$ with the forgetful functor $\operatorname{Perf}(L)\rightarrow\operatorname{Perf}(\mathbb{A}_F)$ coming from the ring homomorphism $\mathbb{A}_F\rightarrow L$ of projection to the $L$-factor.

In particular, on $\pi_1\operatorname{K}$ we get a commutative diagram
$$\xymatrix{ 
L^\times\ar[r]\ar[d] & \pi_1\operatorname{K}(\operatorname{lc}_{L})\ar[d] \\
\mathbb{A}_F^\times/F^\times\ar[r] & \pi_1\operatorname{K}(\operatorname{lc}_{F})}$$
where the left-hand map is induced by the map $L^\times\rightarrow \mathbb{A}_F^\times$ which is the identity in the $L$-component of $\mathbb{A}_F$ and is $1$ everywhere else.

\item Let $F$ be a non-archimedean local field with ring of integers $R$ and residue field $k$ (which is then a finite field).  The functors $\alpha$ intertwine the forgetful functor $\operatorname{lc}_F\rightarrow\operatorname{lc}_R$ with the functor $\partial:\operatorname{Perf}(F)\rightarrow \operatorname{cone}(\operatorname{Perf}_{\mathfrak{m}}(R)\rightarrow 0)$ induced by the localization sequence $\operatorname{Perf}_\mathfrak{m}(R)\rightarrow \operatorname{Perf}(R)\rightarrow \operatorname{Perf}(F)$ as in Remark \ref{boundary}, and the functors $\alpha$ intertwine the forgetful functor $\operatorname{lc}_k\rightarrow\operatorname{lc}_R$ with the forgetful functor $\operatorname{Perf}(k)\rightarrow \operatorname{Perf}_{\mathfrak{m}}(R)$.
In particular, on $\pi_1\operatorname{K}$, we get a commutative diagram
$$\xymatrix{ 
F^\times\ar[r]\ar[d] & \pi_1\operatorname{K}(\operatorname{lc}_F)\ar[d] \\
\mathbb{Z}\ar[r] & \pi_1\operatorname{K}(\operatorname{lc}_{R}) \\
\mathbb{Z}\ar[r]\ar[u] & \pi_1\operatorname{K}(\operatorname{lc}_k)\ar[u]}$$
where the upper left map is the discrete valuation on $F$ and the lower left map is the identity.
\end{enumerate}
\end{proposition}

\section{The Artin maps}

We start by constructing the fundamental class $\operatorname{j}\in\operatorname{dK}^{\operatorname{\operatorname{Sel}}}(\operatorname{lc}_{\mathbb{Z}})$.  Theorem \ref{koflcz} gives a fiber sequence
$$\operatorname{dK}^{\operatorname{\operatorname{Sel}}}(\operatorname{lc}_{\mathbb{Z}})\rightarrow \operatorname{dK}^{\operatorname{\operatorname{Sel}}}(\mathbb{R})\rightarrow \operatorname{dK}^{\operatorname{\operatorname{Sel}}}(\mathbb{Z}).$$
Thus, to construct a point in $\operatorname{dK}^{\operatorname{\operatorname{Sel}}}(\operatorname{lc}_{\mathbb{Z}})$, it is enough to construct a point in $\operatorname{dK}^{\operatorname{\operatorname{Sel}}}(\mathbb{R})$ and a nullhomotopy of the image of that point in $\operatorname{dK}^{\operatorname{\operatorname{Sel}}}(\mathbb{Z})$.

By the pushout defining $\operatorname{dK}^{\operatorname{\operatorname{Sel}}}$, this is equivalent to giving the following data:
\begin{enumerate}
\item A point in $d_{K(1)}\operatorname{K}(\mathbb{R})$.
\item A point in $d_{K(1)}\operatorname{TC}(\mathbb{Z})$.
\item A homotopy between the images of these points in $d_{K(1)}\operatorname{K}(\mathbb{Z})$.
\item A nullhomotopy of the image of the second point in $d_{TC}\operatorname{TC}(\mathbb{Z})$.
\end{enumerate}
For data 1, we take the $K(1)$-localization of the map $J_{\mathbb{R}}:\operatorname{K}(\mathbb{R})\rightarrow\operatorname{Pic}(S_{\widehat{p}})$ of Theorem \ref{Jthm}.  For data 2, we take the $K(1)$-localization of $J_{\mathbb{Z}_p}$, i.e.\ the map $j_{\mathbb{Z}_p}$ studied in Section \ref{tcdual}.  For data 3, we take the homotopy gotten by $K(1)$-localizing Theorem \ref{Jthm}.  The existence of data 4 is tautological from the construction of the natural transformation $\operatorname{n}:d_{K(1)}\rightarrow d_{TC}$ (Corollary \ref{n}).  Thus we have specified a point $\operatorname{j}\in\operatorname{dK}^{\operatorname{\operatorname{Sel}}}(\operatorname{lc}_{\mathbb{Z}})$.

\begin{remark}
In fact, one can calculate that $\pi_0\operatorname{dK}^{\operatorname{Sel}}(\operatorname{lc}_{\mathbb{Z}})$ is a free $\mathbb{Z}_p$-module on the class of $\operatorname{j}$.  For this one needs to use some input from the etale cohomology of $\operatorname{Spec}(\mathbb{Z}[1/p])$, notably a part of the standard exact sequence for Brauer group of $\mathbb{Q}$.  Using this sort of calculation one can see the existence of the class of $\operatorname{j}$ in $\pi_0\operatorname{dK}^{\operatorname{Sel}}(\operatorname{lc}_{\mathbb{Z}})$ without referring to the explicit construction of \cite{C}, and thereby avoid the use of \cite{C} in this paper at the cost of borrowing more material from traditional class field theory.  But we consider it to be interesting and possibly indicative of some fundamental structure that there is an explicit topological construction which pins down $\operatorname{j}$, not just up to homotopy.  It also raises some further questions which may be related to explicit class field theory; see Remark \ref{beforeK(1)}.
\end{remark}

\begin{definition}
Let $\mathcal{P}\in\operatorname{PerfCat}_{\mathbb{Z}}$.  Since K-theory, TC-theory, and the trace map between them are multiplicative with respect to the tensor product in $\operatorname{PerfCat}_{\mathbb{Z}}$ (\cite{BGT2}), the theory $\operatorname{K}^{\operatorname{Sel}}$ of the introduction is also multiplicative.  Moreover, the pushout diagram defining $\operatorname{dK}^{\operatorname{Sel}}$, and hence $\operatorname{dK}^{\operatorname{Sel}}$ itself, is $\operatorname{K}^{\operatorname{Sel}}$-linear.  Thus the tautological pairing
$$\operatorname{lc}_{\mathcal{P}}\otimes_{\mathbb{Z}}\mathcal{P}\rightarrow\operatorname{lc}_\mathbb{Z}$$
induces a map of spectra
$$\operatorname{K}^{\operatorname{Sel}}(\operatorname{lc}_{\mathcal{P}})\otimes \operatorname{dK}^{\operatorname{Sel}}(\operatorname{lc}_{\mathbb{Z}})\rightarrow \operatorname{dK}^{\operatorname{Sel}}(\mathcal{P}).$$
Define the Artin map
$$\operatorname{Art}_{\mathcal{P}}:\operatorname{K}^{\operatorname{Sel}}(\operatorname{lc}_{\mathcal{P}})\rightarrow \operatorname{dK}^{\operatorname{Sel}}(\mathcal{P})$$
to be the evaluation of this pairing on the class $\operatorname{j}\in\operatorname{dK}^{\operatorname{\operatorname{Sel}}}(\operatorname{lc}_{\mathbb{Z}})$ described above.
\end{definition}

\begin{remark}
By construction, $\operatorname{Art}_{\mathcal{P}}$ is a $\operatorname{K}^{\operatorname{Sel}}$-linear natural transformation of functors $\operatorname{PerfCat}_{\mathbb{Z}}^{op}\rightarrow\operatorname{Sp}$.\end{remark}

\begin{remark}
We will only be interested in the restriction of $\operatorname{Art}_{\mathcal{P}}$ to a map
$$\operatorname{K}(\operatorname{lc}_{\mathcal{P}})\rightarrow \operatorname{dK}^{\operatorname{Sel}}(\mathcal{P}).$$
This restriction is perhaps more tangible from a concrete perspective: if a point in $\operatorname{K}(\operatorname{lc}_{\mathcal{P}})$ is represented by an object $F\in \operatorname{lc}_{\mathcal{P}}$ viewed as a $\mathbb{Z}$-linear functor $\mathcal{P}\rightarrow\operatorname{lc}_{\mathbb{Z}}$, then $\operatorname{Art}_{\mathcal{P}}(F)$ is simply the pullback of $\operatorname{j}$ by the functor $F$, using the functoriality of $\operatorname{dK}^{\operatorname{Sel}}$.
\end{remark}

By Theorem \ref{eiso} and standard Galois cohomological dimension estimates (\cite{Se} II.5 and II.6), if $F$ is a number field, local field, or finite field, then the edge map gives an isomorphism
$$\pi_1\operatorname{dK}^{\operatorname{Sel}}(F)\simeq (\operatorname{G}_F^{ab})_{\widehat{p}},$$
and similarly if $R$ is the ring of integers of a non-archimedean local field with residue field $F$ then
$$\pi_1\operatorname{dK}^{\operatorname{Sel}}(R)\simeq (\pi_1^{et}(R)^{ab})_{\widehat{p}}=(\operatorname{G}_F^{ab})_{\widehat{p}}.$$

On the other hand, Corollary \ref{kcompare} gives us comparison maps
$$\operatorname{C}_F\rightarrow \pi_1\operatorname{K}(\operatorname{lc}_F).$$
Composing with $\pi_1\operatorname{Art}_F$, we obtain homomorphisms
$$a_F: \operatorname{C}_F\rightarrow (\operatorname{G}_F^{ab})_{\widehat{p}}$$
for all such $F$.  Moreover, the functoriality of $\operatorname{Art}$ together with Proposition \ref{comparefunct} shows that these maps satisfy the usual functoriality of Artin maps.  The following lemma is then standard:

\begin{lemma}
Suppose given such system of maps $\{a_F\}$ satisfying the functoriality properties of Proposition \ref{kcompare}.  If the map $a_{\mathbb{Q}_p}:\mathbb{Q}_p^\times\rightarrow \left(\operatorname{G}_{\mathbb{Q}_p}^{ab}\right)_{\widehat{p}}$ has the property that when we restrict it to $\mathbb{Z}_p^\times$ and compose it with the $p$-cyclotomic character $\left(\operatorname{G}_{\mathbb{Q}_p}^{ab}\right)_{\widehat{p}}\rightarrow \mathbb{Z}_p^\times/\mu_{p-1}$ we get the tautological map, then the system of maps $\{a_F\}$ identifies with the $p$-completion of the negative of the usual Artin maps.
\end{lemma}
\begin{proof}
First note that by the norm functoriality in field extensions (Proposition \ref{comparefunct} part 1) and the fact that norm subgroups are open, the maps $a_F$ are automatically continuous in the local and global cases.  Now consider $F=\mathbb{Q}$.  Since $\mathbb{Q}(\zeta_{p^\infty})$ is unramified outside $p$, the functoriality Proposition \ref{comparefunct}  parts 2 and 3 implies that the effect of $a_{\mathbb{Q}}$ on $\mathbb{Q}(\zeta_{p^\infty})$ is trivial when restricted to $\mathbb{Z}_\ell^\times\rightarrow \operatorname{C}_{\mathbb{Q}}$ for $\ell\neq p$.  By continuity, it is then trivial on the product of these groups.  But the quotient $\mathbb{A}_{\mathbb{Q}}^\times/(\mathbb{Q}^\times\cdot\prod_{\ell\neq p}\mathbb{Z}_\ell^\times)$ identifies with $\mathbb{Z}_p^\times$ coming from the $p$-factor, so our hypothesis uniquely determines $a_{\mathbb{Q}}$ on the $p$-cyclotomic extension.  In particular, we see that $a_{\mathbb{Q}}$ sends a uniformizer at $\ell$ to the inverse of the $\ell$-Frobenius in $\operatorname{Gal}(\mathbb{Q}(\zeta_{p^\infty})/\mathbb{Q})_{\widehat{p}}$ for $\ell\neq p$.  Therefore by Proposition \ref{comparefunct} part 3 we find that $a_{\mathbb{F}_\ell}$ sends $1$ to the inverse of Frobenius in the $p$-cyclotomic extension of $\mathbb{F}_\ell$.  Since $\ell$ has infinite order in $\mathbb{Z}_p^\times/\mu_{p-1}$, this implies that $a_{\mathbb{F}_\ell}$ sends $1$ to the inverse of Frobenius in $(\operatorname{G}_{\mathbb{F}_\ell}^{ab})_{\widehat{p}}$.  By the functoriality of Proposition \ref{comparefunct}  part 1 we deduce the same for any finite field of characteristic $\neq p$.  Now we turn to the case of an arbitrary number field $F$.  The uniformizers at primes outside a given finite subset generate a dense subgroup of $\mathbb{A}_F^\times/F^\times$,  and every finite extension of $F$ is unramified outside a finite subset, so from the case of finite fields of charcteristic $\neq p$ and functoriality we deduce that $a_F$ is also the negative of the usual Artin map in this case.  Every local field of characteristic $0$ is the completion of a number field, so we deduce from Proposition \ref{comparefunct}  part 2 that $a_F$ is also the usual Artin map in that case, and from this and Proposition \ref{comparefunct}  part 3 we see that $a_F$ sends $1$ to Frobenius also for $F$ finite of characteristic $p$.  Then we can run the same argument again for global fields and local fields of characteristic $p$ to deduce that $a_F$ is the usual Artin map in these remaining cases.
\end{proof}

To conclude our discussion, we need to show that $a_{\mathbb{Q}_p}$ has the indicated property.  For this it suffices to see:

\begin{lemma}\label{lastone}
Consider the functor $\operatorname{Perf}(\mathbb{Q}_p)\rightarrow\operatorname{lc}_{\mathbb{Z}}$ classifying the object $\mathbb{Q}_p\in \operatorname{lc}_{\mathbb{Z}}$ with its $\mathbb{Q}_p$-action.  The pullback of $\operatorname{j}\in\operatorname{dK}^{\operatorname{\operatorname{Sel}}}(\operatorname{lc}_{\mathbb{Z}})$ along this functor, viewed as a point in
$$\operatorname{dK}^{\operatorname{Sel}}(\mathbb{Q}_p)=d_{K(1)}\operatorname{K}(\mathbb{Q}_p)=d_{K(1)}\operatorname{K}(\mathbb{Z}_p),$$
identifies with $j_{\mathbb{Z}_p}$.
\end{lemma}

Indeed, this lemma implies that $\operatorname{Art}_{\mathbb{Q}_p}:\operatorname{K}(\mathbb{Q}_p)\rightarrow \operatorname{dK}^{\operatorname{Sel}}(\mathbb{Q}_p) = d_{K(1)}\operatorname{K}(\mathbb{Q}_p)$ is the $\operatorname{K}(\mathbb{Q}_p)$-linear map classifying $j_{\mathbb{Z}_p}$.  But Corollary \ref{jcor} implies that $j_{\mathbb{Z}_p}$ sends $[u]$ to $[u]$, which by Corollary \ref{cyclo} translates exactly into the desired claim.

To prove Lemma \ref{lastone}, note that in terms of the identification $\operatorname{cone}(\operatorname{Perf}(\mathbb{Z})\rightarrow\operatorname{Perf}(\mathbb{R}))\simeq \operatorname{lc}_{\mathbb{Z}}$ of Theorem \ref{conelcz}, this functor $\operatorname{Perf}(\mathbb{Q}_p)\rightarrow \operatorname{lc}_{\mathbb{Z}}$ identifies with the composition of the ``boundary functor" $\partial:\operatorname{Perf}(\mathbb{Q}_p)\rightarrow\operatorname{cone}(\operatorname{Perf}_{\{p\}}(\mathbb{Z})\rightarrow 0)$
associated to the localization sequence
$$\operatorname{Perf}_{\{p\}}(\mathbb{Z})\rightarrow\operatorname{Perf}(\mathbb{Z}_p)\rightarrow\operatorname{Perf}(\mathbb{Q}_p)$$
as in Remark \ref{boundary}, followed by the functor on cones induced by the inclusion $\operatorname{Perf}_{\{p\}}(\mathbb{Z})\rightarrow\operatorname{Perf}(\mathbb{Z})$.

Chasing through localization sequences in $L_{K(1)}\operatorname{K}$ and $\operatorname{TC}_{\widehat{p}}$, we get an identification
$$\operatorname{dK}^{\operatorname{Sel}}(\operatorname{cone}(\operatorname{Perf}_{\{p\}}(\mathbb{Z})\rightarrow 0))\simeq \operatorname{fib}(d_{K(1)}\operatorname{TC}(\mathbb{Z})\rightarrow d_{TC}\operatorname{TC}(\mathbb{Z})),$$
such that the pullback from $\operatorname{dK}^{\operatorname{Sel}}(\operatorname{lc}_{\mathbb{Z}})$ corresponds to remembering just the data 2 and 4 from the description of points of $\operatorname{dK}^{\operatorname{Sel}}(\operatorname{lc}_{\mathbb{Z}})$ given above, whereas the pullback to $\operatorname{dK}^{\operatorname{Sel}}(\mathbb{Q}_p)=d_{K(1)}\operatorname{K}(\mathbb{Q}_p)$ corresponds to the natural forgetful map to $d_{K(1)}\operatorname{TC}(\mathbb{Z})\simeq d_{K(1)}\operatorname{K}(\mathbb{Q}_p)$.  This verifies the claim, by the definition of the point $\operatorname{j}\in \operatorname{dK}^{\operatorname{Sel}}(\operatorname{lc}_{\mathbb{Z}})$.

\begin{remark}\label{noselmer}
Much of the complication of this paper, for example the material in section \ref{tcdual} and the fairly subtle definition of Selmer K-homology in Section \ref{selmer}, arose from a desire to have a completely uniform framework for describing Artin maps, both in characteristic $0$ and characteristic $p$.  If one is willing to treat these cases separately, much less background is required.  For characteristic $0$ we only need $d_{K(1)}\operatorname{K}$, not anything to do with TC.  And for characteristic $p$ we only need $d_{TC}\operatorname{TC}$ --- and we can define this duality in an easier way since everything is linear over $\operatorname{TC}(\mathbb{F}_p)_{\widehat{p}} = H\mathbb{Z}_p\oplus \Omega H\mathbb{Z}_p$ instead of the more complicated $\operatorname{TC}(\mathbb{Z})_{\widehat{p}}$.
\end{remark}

\begin{remark}\label{beforeK(1)}
Let us see what our description of the Artin maps says in more concrete terms, when $F$ has characteristic $\neq p$.  Suppose given an element $c\in\operatorname{C}_F$, and an element $g\in\operatorname{Gal}(\overline{F}/F)^{ab}$.  Then from both $c$ and $g$ we can produce a homotopy class of maps of spectra
$$\operatorname{K}(F)\rightarrow \Omega\operatorname{Pic}(S_{\widehat{p}}).$$

For $c$, if we make a similar construction to that of $\operatorname{j}$ but using the unlocalized $J$'s of Theorem \ref{Jthm} we can make a map of spectra
$$\operatorname{K}(\operatorname{lc}_{\mathbb{Z}[1/p]})\rightarrow \operatorname{Pic}(S_{\widehat{p}}).$$
Then the class $c\in \pi_1\operatorname{K}(\operatorname{lc}_F)$ and the pairing $\operatorname{K}(F)\otimes\operatorname{K}(\operatorname{lc}_F)\rightarrow \operatorname{K}(\operatorname{lc}_{\mathbb{Z}[1/p]})$ give the desired map $\operatorname{K}(F)\rightarrow \Omega\operatorname{Pic}(S_{\widehat{p}})$.

On the other hand, for $g$ we can use the construction of Remark \ref{Jedge}, meaning we look at the action of $g$ on the $\ell$-adic etale one-point compactifications of $F$-vector spaces base-changed to $\overline{F}$.

What we have proved in this paper is that after $K(1)$-localization, the classes associated to $c$ and $g$ become equal if and only if $c$ and $g$ match up under the Artin map.  This raises a natural question: is the same statement true without the $K(1)$-localization?  I would conjecture that it is.  For finite fields $F$, the map $\operatorname{K}(F)\rightarrow \Omega\operatorname{Pic}(S_{\widehat{p}})$ associated to $1\in C_F$ identifies with the map $J_F$ of \cite{C} Section 3.1 as a consequence of the product formula proved in \cite{C}, and thus this conjecture reduces to Conjecture 4.5 of \cite{C}.  Already that case seems hard.  The only case for which it seems currently plausible to attack this conjecture is that of imaginary quadratic fields $F$, where one can work with elliptic curves with complex multiplication to mediate between the $J$'s of Theorem \ref{Jthm} and the $J^{et}$'s of Remark \ref{etalej}.
\end{remark}

\end{document}